\documentclass[11 pt]{article}
\usepackage{authblk}

\usepackage[utf8]{inputenc}
\usepackage{amsmath,indentfirst,stmaryrd,qsymbols,graphicx}
\usepackage{psfrag,amsfonts,hyperref,enumerate,amsthm,ulem,dsfont}
\usepackage[dvipsnames]{xcolor}
\usepackage{verbatim}
\usepackage{footmisc}

 \usepackage{scalerel,stackengine}
\stackMath
\newcommand\widecheck[1]{%
\savestack{\tmpbox}{\stretchto{%
  \scaleto{%
    \scalerel*[\widthof{\ensuremath{#1}}]{\kern-.6pt\bigwedge\kern-.6pt}%
    {\rule[-\textheight/2]{1ex}{\textheight}}%WIDTH-LIMITED BIG WEDGE
  }{\textheight}% 
}{0.5ex}}%
\stackon[1pt]{#1}{\scalebox{-1}{\tmpbox}}%
}
\newcommand\W[1]{\widecheck{#1}}%

\graphicspath{{./FIGS/}}

\usepackage[left=2cm,right=2cm,top=2cm,bottom=2cm]{geometry}
\DeclareMathOperator{\cov}{cov}

\DeclareMathOperator{\argmin}{argmin}

\renewcommand{\mod}{{~\sf mod~}}
\def \1{\mathds{1}}

\def \Conj{{\sf Conj}}

\def \hol{{\sf Hol}}

\def \Z{\mathbb{Z}}

\def \Rn#1#2{R_#1[#2]}
\def \DR#1#2{{\bf R}_{#1}[#2]}

\def \Ref{{\sf Feuilletage}}
\def \Qset#1{\mathbb{ Q}_{#1}^{\bullet,\rightarrow}}
\def \Qrec#1#2#3{ Q_{#1}^{(#2, #3)}}
\def \Tset { \mathbb{T}}
\def \Tal#1 {{ \bf T}_#1}
\def \LT#1{ \mathbb{ LT}_{#1}}
\def \t {{T}}
\def \tt#1#2{T_{#1}^{(#2)}}
\def \ttau{{\tau}}
\def \bt{{\bf t}} %toujours utilisé pour l'objet continu
\def \btj #1{{\bf t}^{(#1)}}
\def \btnj #1#2{{\bf T}_{#1}^{(#2)}}

\def \BS{{\bf BS}}
\def \bs{{\bf bs}}
\def \RS{\overrightarrow{\sf Snakes}}
\def \PS{{\sf Snakes}^{\bullet}}
\def \RBS{
{\BS}}
\def \rbs{
{\bs}}

\def \pbs{\bs^{\bullet}}

\def \RR#1{{\bf r}[{#1}]}
\def \Rc{{\bf r}}

\newcommand{\Dyck}{{\sf Dyck}}
\newcommand{\permutation}{{permutation }}

\newcommand{\E}{\mathbb{E}}
\newcommand{\cM}{\mathcal{M}}

\newcommand{\fC}{f_0}
\newcommand\xoutpars[1]{\let\helpcmd\xout\parhelp#1\par\relax\relax}
\newcommand\soutpars[1]{\let\helpcmd\sout\parhelp#1\par\relax\relax}
\long\def\parhelp#1\par#2\relax{%
  \helpcmd{#1}\ifx\relax#2\else\par\parhelp#2\relax\fi%
}

\def \bls{{\tiny $\blacksquare$ }}

\def \R{{\sf R}}

\newcommand{\dVar}{d_{{\sf Var}}}

\newcommand{\cV}{\mathcal{V}}

\newcommand{\cD}{\mathcal{D}}

\newcommand{\by}{{\bf y}}

\usepackage{amsmath,amsbsy}

\newcommand{\D}{D}

\newcommand*\Bell{\ensuremath{\boldsymbol\ell}}

\def \QC{{QC^0[0,1]}}
\def\cro#1{\llbracket#1\rrbracket}

\def \ov#1{\overline{#1}}

\def \H{{\mathcal{H}}}

\def \app#1#2#3#4#5{\begin{array}{rccl} #1:&#2&\longrightarrow&#3\\ &#4&\longmapsto&#5\end{array}}
\def \as{\xrightarrow[n]{(as.)}}
\def \N{\mathbb{N}}
\def \R{\mathbb{R}}
\def \bB{{\bf B}}
\def \bc{{\bf c}}
\def \bh{{\bf h}}

\def \bH{{\bf H}}
\def \bL{{\bf L}}
\def \bQ{{\bf Q}}

\def \bX{{\bf X}}
\def \bY{{\bf Y}}

\def \bT{{\bf T}}

\def \PTC{{\sf PointedTreeClass}}

\def \Bs{{\tiny $\blacksquare$}}

\def \bar{\overline}

\def \Ba{{\bf a}}
\def \BA{{\bf A}}
\def \ba{\begin{align}}
\def \ea{\end{align}}
\def \be{\begin{eqnarray*}}
\def \ee{\end{eqnarray*}}
\def \ben{\begin{eqnarray}}
\def \een{\end{eqnarray}}
\def \bh{{\bf h}}

\def \beq{\begin{equation}}
\def \eq{\end{equation}}
\def \build#1#2#3{\mathrel{\mathop{\kern 0pt#1}\limits_{#2}^{#3}}}

\def \ba{{\bf a}}

\def \by{{\bf y}}

\def \cS{{\cal S}}

\def \dd{\xrightarrow[n]{(d)}}
\def \dis{\displaystyle}

\def \equi{\Leftrightarrow}
\def \eref#1{(\ref{#1})}

\def \P{\mathbb{P}}

\def \bC{{\bf C}}
\def \imp{\Rightarrow}

\def \l{\left}

\def \r{\right}
\def \sous#1#2{\mathrel{\mathop{\kern 0pt#1}\limits_{#2}}}
\def \sur#1#2{\mathrel{\mathop{\kern 0pt#1}\limits^{#2}}}
\def \eqd{\sur{=}{(d)}}
\def \w#1{\widetilde{#1}}

\def \bbT{{\mathbb{T}}}
\def \bbH{{\mathbb{H}}}
\def \bbL{{\mathbb{L}}}
\newqsymbol{`B}{\mathcal{B}}
\newqsymbol{`E}{\mathbb{E}}
\newqsymbol{`I}{\mathbb{I}}
\newqsymbol{`N}{\mathbb{N}}
\newqsymbol{`O}{\Omega}
\newqsymbol{`P}{\mathbb{P}}
\newqsymbol{`Q}{\mathbb{Q}}
\newqsymbol{`R}{\mathbb{R}}
\newqsymbol{`W}{\mathbb{W}}
\newqsymbol{`Z}{\mathbb{Z}}
\newqsymbol{`a}{\alpha}
\newqsymbol{`e}{\varepsilon}
\newqsymbol{`o}{\omega}
\newqsymbol{`t}{\tau}
\newqsymbol{`w}{{\cal W}}

\newcommand{\se}{{\sf e}}

\newcommand{\compact}{ \topsep0pt   \itemsep=0pt   \partopsep=0pt   \parsep=0pt}

\newcounter{c}
\def \bir{\begin{itemize}\compact \setcounter{c}{0}}
\def \itr{\addtocounter{c}{1}\item[($\roman{c}$)]}
\def \eir{\end{itemize}\vspace{-2em}~}

\newcounter{d}
\def \bia{\begin{itemize}\compact \setcounter{d}{0}}
\def \eia{\end{itemize}\vspace{-2em}~}

\newcounter{b}
\def \bi{\begin{itemize}\compact \setcounter{b}{0}}

\def \ei{\end{itemize}\vspace{-2em}~}

\def \bis{\begin{itemize}\compact }
\def \its{\item[{\bls}]}
\def \eis{\end{itemize}\vspace{-2em}~}

\def \bpar#1{\left\{\begin{array}{#1} }
\def \epar { \end{array}\right.}

\newtheorem{lem}{Lemma}[section]
\newtheorem{defi}[lem]{Definition}
\newtheorem{pro}[lem]{Proposition}
\newtheorem{theo}[lem]{Theorem}
\newtheorem*{theo*}{Theorem}

\newtheorem{rem}[lem]{Remark}

\newtheorem{OQ}{Open question}

\begin{document}

  \title{Iterated foldings of discrete spaces and their limits: candidates for the role of Brownian map in higher dimensions}
  %\subtitle{Iterated foldings of discrete spaces}
\author[$\dagger$]{\bf Luca Lionni}
\author[$\ddag$]{\bf Jean-François Marckert}
\affil[$\dagger$]{CNRS UMR 8243 - IRIF,  Universit\'e Paris Diderot - Paris 7
B\^atiment Sophie Germain, 75205 Paris Cedex 13, France.}
\affil[$\ddag$]{Univ. Bordeaux, CNRS, UMR 5800 - LaBRI, Bordeaux INP, 351 cours de la Libération, 33405 Talence Cedex-France}  

%\runauthor{L. Lionni and JF. Marckert}

\date{}
\maketitle

 \begin{abstract} In this last decade, an important stochastic model emerged: the Brownian map. It is the limit of various models of random combinatorial maps after rescaling: it is a random metric space with Hausdorff dimension 4, almost surely homeomorphic to the 2-sphere, and possesses some deep connections with 
   Liouville quantum gravity in 2D.
In this paper, we present a sequence of random objects that we call $\D$th-random feuilletages (denoted by $\RR{\D}$), indexed by a parameter $\D\geq 0$ and which are candidate to play the role of the Brownian map in dimension $D$.  The construction relies on some objects that we name iterated Brownian snakes, which are branching analogues of iterated Brownian motions, and which are moreover limits of iterated discrete snakes. In the planar $D=2$ case, the family of discrete snakes considered coincides with some family of (random) labeled trees known to encode planar quadrangulations.
   
   Iterating snakes provides a sequence of random trees $(\bt^{(j)}, j\geq 1)$. The $D$th-random feuilletage $\RR{D}$ is built using $(\bt^{(1)},\cdots,\bt^{(D)})$: $\RR{0}$ is a deterministic circle, $\RR{1}$ is Aldous' continuum random tree, $\RR{2}$ is the Brownian map, and somehow, $\RR{D}$ is obtained by quotienting $\bt^{(D)}$ by $\RR{D-1}$.

  A discrete counterpart to $\RR{\D}$ is introduced and called the $\D$th random discrete feuilletage with $n+\D$ nodes ($\DR{n}{\D}$). The proof of the convergence of $\DR{n}{\D}$ to $\RR{\D}$ after appropriate rescaling in some functional space is provided (however, the convergence obtained is too weak to imply the Gromov-Hausdorff convergence). An upper bound on the diameter of $\DR{n}{\D}$ is $n^{1/2^{\D}}$. Some elements allowing to conjecture that the Hausdorff dimension of $\RR{\D}$ is $2^\D$ are given.
\end{abstract}
%\keywords{Brownian maps, Brownian snakes, Iterated Brownian motion, Iterated snakes, Combinatorial maps, Maps in higher dimension, quantum gravity}

{\small \paragraph{Acknowledgements:}
 This works has been partially supported by ANR GRAAL (ANR-14-CE25-0014). This project has received funding from the European Research Council 
(ERC) under the European Union’s Horizon 2020 research and innovation 
programme (grant agreement No.~ERC-2016-STG 716083, "CombiTop").
 LL thanks Dario Benedetti  and Valentin Bonzom for useful discussions on the physical motivations.
}
% subject classification: 60F17, 60J80
\small
\tableofcontents
\normalsize
\section{Introduction}
\label{sec:intro}

\subsection{Presentation of the main objects}
\label{sec:PMO}

The question at the origin of this paper is the following: are there any random continuous objects likely to play the role of the ``Brownian map'' in higher dimensions?

The question is probably ill posed since it may not be so clear what the dimension of the Brownian map is: on the one hand it is indeed a.s.~homeomorphic to the 2-sphere (Le Gall \& Paulin \cite{LGP}), and on the other hand, by Le Gall \cite[Theo. 6.1]{LG7}, it has Hausdorff dimension 4... and due to its huge fluctuations, it is likely not possible\footnote{On $\R^\D$, the largest tuple of points $(x_1,\cdots,x_k)$ such that $i\neq j\imp d(x_i,x_j)=1$ is bounded by $(\D+1)$ when it is not the case on maps.}
to embed it isometrically in $(\R^\D,\|.\|_2)$ for any finite $\D$.
\par

Fixing the topological dimension as the base of our considerations, we are not aware of any family of combinatorial objects that would be the right candidates to play the role, in dimension $\D>2$, of the combinatorial maps which provide, in the $\D=2$ case, the Brownian map as a scaling limit. To our knowledge, the previous attempts, in a theoretical-physics context, either led to Aldous' continuous random tree,  to the Brownian map\footnote{This is very likely the case for certain simple models of random $4$-dimensional triangulations when using the distance on the dual graph \cite{Enhanced, Lionni:17}, and is conjectured for another model of random $3$-dimensional triangulations using the distance in the triangulation.}, or to a crumpled phase with ``infinite Hausdorff dimension"\footnote{\label{foot:note-0}More precisely, this last case corresponds to families of random $\D$-dimensional triangulations whose diameter is a.s.~bounded when the number of simplices goes to infinity, for which a scaling limit cannot be defined.}  \cite{Ambjorn1992, Thorleifsson1999, Gurau2014}.
We propose the construction, for every integer $\D\geq 1$, of a continuous random object which we call the $\D$th random feuilletage $(\RR{\D})$. By construction, $\RR{\D}$ will appear as the limit (for a topology discussed further) of a discrete analogue, which we call $\D$th random discrete feuilletage  $\DR n\D$, when a size parameter $n$ goes to $+\infty$. The latter can be viewed as obtained by $\D-2$  series of foldings of a random discrete surface.
The sequence $(\RR{\D},\D\geq 1)$ is encoded and built thanks to another new sequence of objects $(\rbs[\D],\D \geq 1)$: the $\D$th Brownian snake $\rbs[\D]$  is  --  {\it mutatis mutandis} -- a branching analogue of the $D$th iterated Brownian motion. The first Brownian snake, $\rbs[1]$, is the usual Brownian snake with lifetime process the normalized Brownian excursion $\se$.
The $D$th Brownian snake appears as the limit of 
a discrete counterpart after some appropriate normalizations, which we call $\D$th random discrete snake $\RBS_n[\D]$   (and $\rbs_n[\D]$ for the normalized version).
 \par
    By construction, $\RR{0}$ can be thought to be a deterministic cycle, $\RR{1}$ and  $\RR{2}$   will respectively be seen to coincide with Aldous' continuum random tree and with the Brownian map, and for $\D\geq 3$, we think that $\RR{\D}$ is a candidate to play the role of the Brownian map ``in dimension $\D$''. 
   \\

  \noindent {\bf What are these objects?} 
  The complete and rigorous construction will take pages, but let us try to provide some insights on this construction. First, some words about the terminology: ``feuilletage'' is a french word.\\ 
    \bls\,  It is the french word for ``foliation'', which is a mathematical term used in differential geometry to denote some equivalence relation on manifolds: depending on the context, the equivalence classes (the leaves) correspond to parts of the initial manifold; they are themselves equivalent to some ``regular'' spaces.\\
    \bls\,  In the ``art français de la patisserie'' (french art of pastry cooking), \it le feuilletage \rm or \it la pâte feuilletée \rm is the name for the puff pastry dough, which is the main ingredient of many sweet or salted pastries, as mille-feuille, galette des rois, bouchée à la reine, pâté lorrain, ... and even croissants (with some adjustments).
    The dough is obtained by placing some butter (6 mm thick say) on half a simple dough composed by flour, water and salt, shaped in a rectangular form (20 cm $\times 40$ cm, thick 5 mm say). Then, the dough is folded to cover the butter, and flattened into its initial shape $R$. It is then folded again, flattened, folded, flattened... Each time, the number of layers of butter is multiplied by 2. The pastry chef stops his/her work when the number of layers $2^n$ is large enough: 128 for example. After cooking, if this difficult recipe is well done, the layers are separated: we get the ``feuilletage''.\\
     The construction we will propose is similar to this feuilletage, of course, up to the ingredients and to the precise gestures of the cook... thus our choice of naming.
\medskip

In the following paragraphs, we often omit some precisions, such as how the different objects involved are rooted for instance. The precise definitions 
of the various objects will be given in Sec.~\ref{seq:IBSIM} for the continuous objects, and Sec.~\ref{sec:ISIMCO} and \ref{sec:GraphSection} for the discrete objects.\\
   \bls\, We start by discussing the content of  Fig.~\ref{fig:col1}, in which a well known and simple bijection is sketched: a planar tree having $n$ edges 
   and then $2n$ corners (say, rooted at a corner 0) can be encoded by a non-crossing partition\footnote{We use encodings of planar trees by non-crossing partitions whose Kreweras complements are matchings (disjoint sets encode all the vertices of the tree). For more details, see the end of Sec.~\ref{sec:maps-as-perm}}  on $\{0,\cdots, 2n-1\}$. The integers correspond to the corners of the tree when one turns around starting at 0, and then the partition is a way to present together the different corners of each node. 
   \begin{figure}[h!]
     \centering  \includegraphics[scale=1]{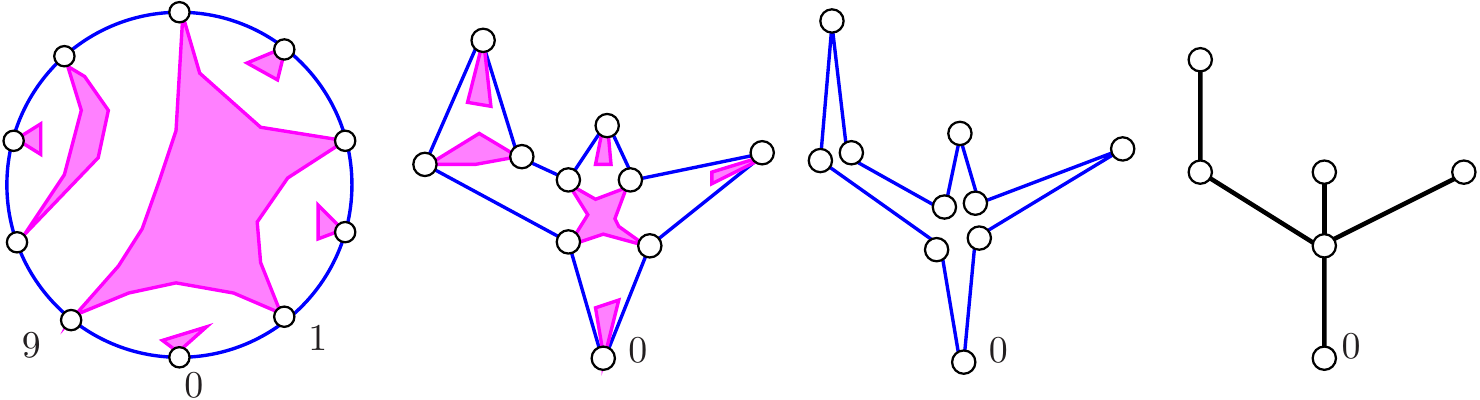}
     \caption{\label{fig:col1}A planar tree with 6 nodes and then 10 corners, seen as a folded circle. The corresponding parts of the non-crossing partitions are $\{0\}, \{1,3,5,9\}, \{2\},\{6,8\},\{7\}$.}
   \end{figure} As may be seen in  Fig.~\ref{fig:col1}, this bijection can be used to present a tree as a circle folded multiple times: these foldings are  encoded by the non-crossing partition on a finite subset of the (continuous)  circle (or on the discrete circle with $2n$ points). Through this bijection, trees and non-crossing partitions are essentially the same combinatorial objects. 
     For the sake of studying asymptotics of trees, one usually prefers to use contour processes instead of non-crossing partitions (Fig.~\ref{fig:col2}), since it allows gaining access to the toolbox of usual linear stochastic processes.
      It is now folklore that the contour process can be glued from below to recover the tree (Fig.~\ref{fig:col2}), so that the two points of view of Fig.~\ref{fig:col1} and Fig.~\ref{fig:col2} are equivalent.
  \begin{figure}[h!]
     \centering  \includegraphics[scale=0.9]{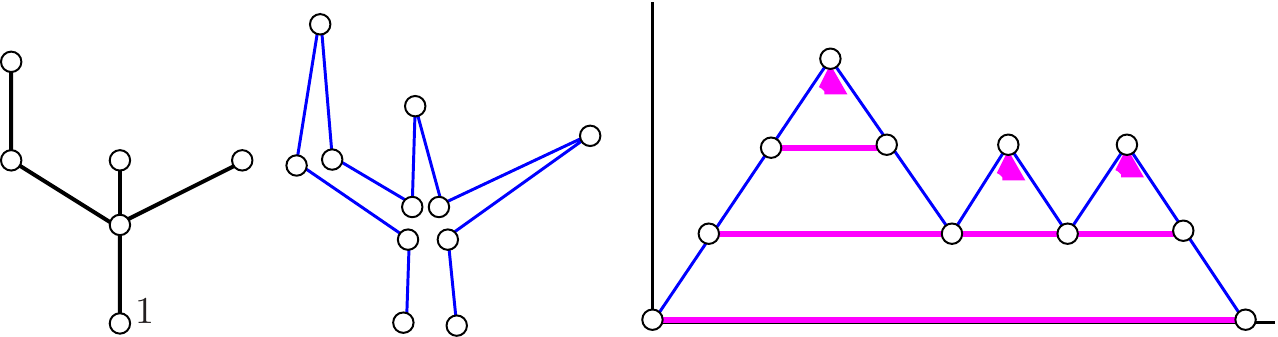}
     \caption{\label{fig:col2}A planar tree and its contour process.}
   \end{figure}\\
     \bls\,  Let us now focus on  Fig.~\ref{fig:col3}, which is the analogue of Fig.~\ref{fig:col1}, when the initial object is a tree instead of a circle. In this case, a tree is folded multiple times according to a non-crossing partition on its corners (for a corner numeration obtained by turning around it).
   \begin{figure}[h!] \centering  \includegraphics[scale=1]{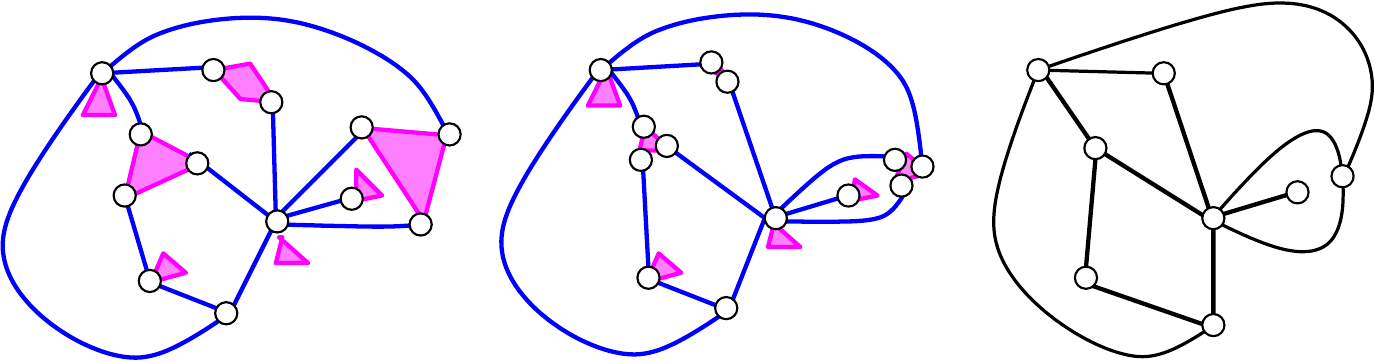}
     \caption{\label{fig:col3}A planar tree equipped with a non-crossing partition on (some corners) of its vertices. The identification of the corners of the tree belonging to the same part allows constructing a map.}
   \end{figure}
   Again, it is now classical that such a construction can be done on the plane by avoiding edge-crossings: the resulting object is a planar map (see Fig.~\ref{fig:col3}).  The converse is also true: it is possible to unplug the edges of any planar map in any order, until there are no more cycles, while preserving the connectedness. The resulting map is a tree, and keeping track of the ancient connections can be done thanks to a non-crossing partition on the corners of that tree.
   Hence, proper foldings of a tree, that is, two series of foldings of a circle, allow constructing a map. This point of view is powerful and at the origin of the first definition of the Brownian map  \cite{MM}: the non-crossing partition encodes a tree $\tt n 1$ (right of Fig.~\ref{fig:col4}), and together with the initial tree $\tt n2$, we have two objects that in turn can be encoded by linear processes. 
   In Fig.~\ref{fig:col4}, the black tree is $\tt n 1$, the blue tree $\tt n 2$: the nodes of the tree $\tt n 2$ are ``glued" in the corners of $\tt n1$. 
     \begin{figure}[h!] \centering  \includegraphics[scale=1]{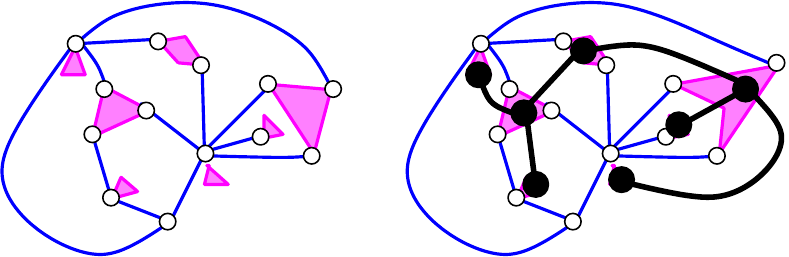}
     \caption{\label{fig:col4}Identification of the non-crossing partition with a tree. }
   \end{figure}
Starting from a uniform planar quadrangulation with $n$ faces, it is possible to construct the random tree
$\btnj n 1$ (the non-crossing partition) as well as the random tree 
$\btnj n 2$ (which contains all the edges of the map) using a bijection which allows controlling the distributions of 
$(\btnj n 1, \btnj n 2)$: 
$\btnj n 1$ is uniform in the set of rooted planar trees with $n$ edges\footnote{More precisely, this is the case when starting from a uniform rooted pointed quadrangulation with $n$ faces.}, 
$\btnj n 2$ has $2n$ edges;  the standard diameter of 
$\btnj n 1$ is $\sqrt{n}$ and that of 
$\btnj n 2$
is $n^{1/4}$ (see \cite{MM} for the representation of quadrangulations with $(\btnj n 1,\btnj n 2)$, construction relying on the Cori-Vauquelin-Schaeffer \cite{CV, Schae} correspondence).  \\ 
   \bls\, We presented a tree and a map respectively   as foldings of a circle and foldings of a tree: this will lead, up to some details (roots, sizes, degree of faces, etc),   to the two first discrete random feuilletages $\DR{n}{1}$ and $\DR{n}{2}$. 
  Moreover, $\DR{n}{0}$ can be considered to be the initial circle $\Z/2n\Z$. To build the sequence $(\DR{n}{\D},\D\geq 3)$, we will fold again and again: the $\D$th object will be constructed by a series of foldings of the $(\D-1)$th one, using an exterior source of randomness which will be a random non-crossing partition of the (nodes) corners of the $(\D-1)$th object. To be more precise, it is worth mentioning that we will rather  use the $(D-1)$th feuilletage to fold a tree, even if at the level of this preliminary presentation, these two ways of doing appear similar. \par
 In order to iterate the construction, a single possibility appears to resist all the requirements. Let us get  a glimpse of a 3-discrete 
 feuilletage $\Rn n 3$  (a deterministic combinatorial object in the support of  $\DR{n}{3}$):  it is obtained by a series of foldings of a planar quadrangulation, using an additional ``non-crossing" partition on its vertices.
 More precisely, the idea is to take three trees $(\tt n 1,\tt n 2,\tt n 3)$, with respectively $n$, $2n$ and $4n$ edges. The tree $\tt n 3$ will contain all the edges of $\Rn n 3$, $\tt n 2$ will encode a non-crossing partition on $\tt n 3$, and $\tt n 1$ a non-crossing partition on $\tt n 2$. Hence, folded by $\tt n 2$, the pair $(\tt n 2,\tt n 3)$ forms a planar quadrangulation 
  $\Qrec n23$,
  and  the additional foldings of the nodes of $\tt n 2$ by the  tree $\tt n 1$ (or the equivalent non-crossing partition) forms  $\Rn n 3$. Notice that the set of nodes of the quadrangulation $\Qrec n23$ coincides with the nodes of $\tt n 2$, so that identifying the nodes of $\tt n 2$ (using $\tt n 1$) provides an important number of additional identifications: the number of nodes in $\Rn n 3$ roughly coincides with those of $\tt n 1$, and then represents half the nodes of $\Qrec n23$. This construction gives us access to the toolbox of stochastic processes, required for considering the asymptotics of these objects. 
\begin{figure}[h!] 
\centering  
\includegraphics[scale=1]{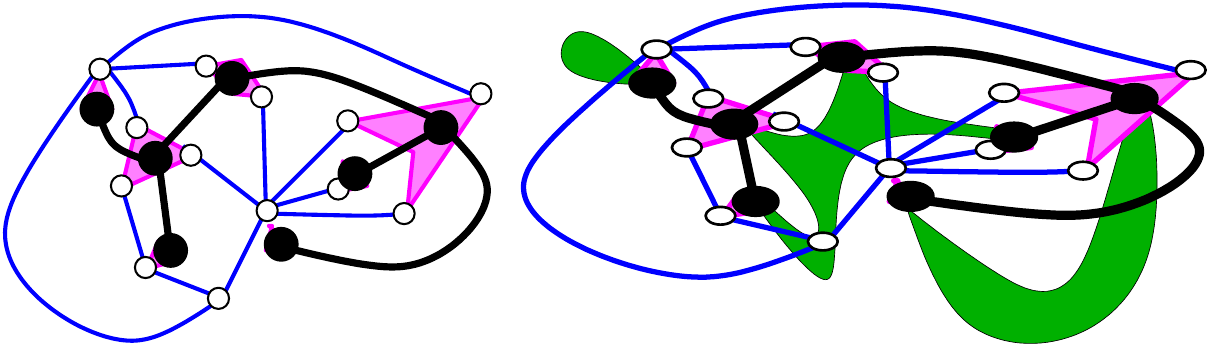}
\caption{\label{fig:col5}A planar map constructed from two trees $\tt n 1$ and $\tt n 2$. On the second picture, an additional non-crossing partition is added on the nodes (corners) of $\tt n 1$. For our construction, the ``green partition'' will be encoded by the tree ${T'}^{(1)}_n$, the pink one, corresponding to the black tree, will be the tree ${T'}^{(2)}_n$, and the blue tree, the one whose edges remain finally, is the tree ${T'}^{(3)}_n$. }
 \end{figure}
Of course, in view of Fig.~\ref{fig:col5}, the obtained structure is not planar, because the foldings of $\Qrec n23$ using $\tt n 1$  
  create a linear number of additional node identifications.\par 
  To produce the subsequent $\Rn n \D$ for $\D\geq 3$, we will just take a sequence of trees $(\tt n 1,\tt n 2,\ldots,\tt n \D)$, where $\tt n i$ will have $2^{i-1} n$ edges, and proceed to the identification of the nodes  of $\tt n \D$: for any $j<\D$, the tree $\tt n j$ is used to identify the corners of $\tt n {j+1}$ (as do non-crossing partitions), producing $\D$ successive series  of foldings of the vertices  of $\tt n {\D}$  starting from the circle. This allows identifying the nodes of the obtained object $\Rn n \D$ as those  of $\tt n 1$ (roughly), and the edges of $\Rn n \D$   
  as those of $\tt n \D$.  \\
  \bls\,  An issue in the construction sketched above, is to define a distribution on the set of objects under investigation for which the main characteristics of interest are tractable. We propose a construction for which the natural scalings\footnote{The asymptotic dependence of their diameters in the number of edges.} of the random trees $\btnj n 1,  \btnj n 2, \ldots, \btnj n \D$ 
  are respectively  $n^{1/2}, n^{1/4},\cdots,n^{1/2^\D}$.  We think that this  iterative construction of trees is interesting on its own: it is somewhat similar to the construction of the  iterated Brownian motions.

  Take a uniform rooted planar tree $\btnj n 1$ with $n$ edges, and use this tree as the underlying tree of a branching random walk with increments uniform in $\{0,1,-1\}$. That is, conditionally on $\btnj n 1=\t$ some fixed planar tree, equip each node  $u$ of $\t$ (different from the root) with a random variable $X_u$ uniform in  $\{0,1,-1\}$. A labeling  of each node $u$ is then obtained by summing the variables on the path from the root to $u$.  
  \begin{figure}[h!]
\centerline{\includegraphics{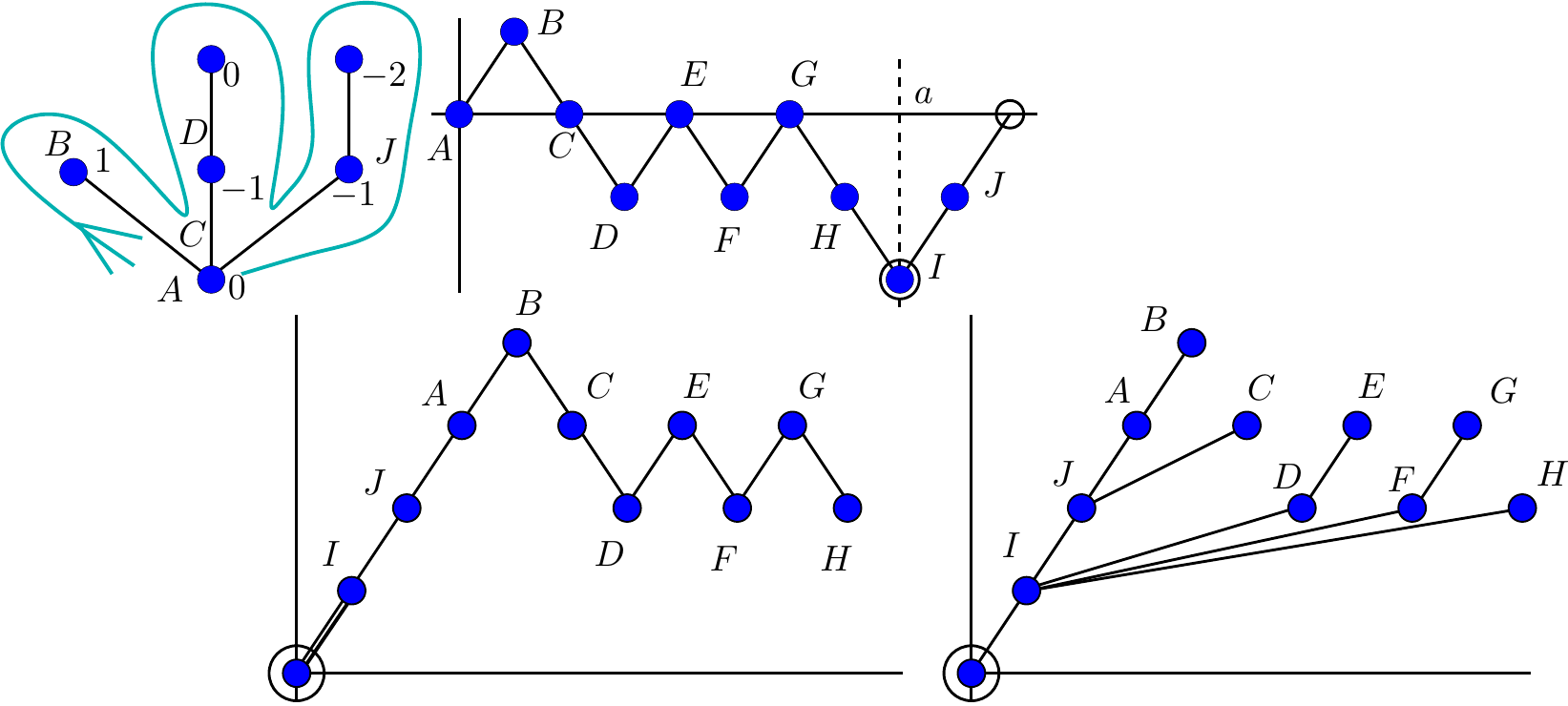}}
\caption{\label{fig:wlt}Definition of a new tree with the labels of a labeled tree. To a labeled tree (upper left) is associated a label sequence (upper right), and the conjugated labeled sequence (bottom left) is the height sequence of a new tree (bottom right).}
\end{figure}
On each branch, the labeling forms a random walk which starts at the label 0 of the root (see Fig.~\ref{fig:wlt}). To construct a second tree $\tt n 2$ from this labeling (see \cite{CV, Schae, MM}), walk around the tree as done in Fig.~\ref{fig:wlt}, and record the successive labels encountered (one per corner); we get the so-called label process $\bL_n^{(1)}$  (second picture in  Fig.~\ref{fig:wlt}).  Register $a= \min \argmin \bL_n^{(1)}$, the first time the label process reaches its minimum (dotted lines on the second picture). Now, on the third picture of  Fig.~\ref{fig:wlt}, a new process starting at 0 is obtained by adding the point $(1,1)$,  and then, by appending the increments after $a$ of $\bL_n^{(1)}$, and then those before $a$ (we perform the so-called conjugation of paths). The process obtained this way is positive on $\cro{1,n}$ and has increments $+1,-1,0$: it is the height process of a random tree $\btnj n 2$. 
Every realization $\tt n 2$ of $\btnj n 2$  is a planar tree in its own right: it can be used as the underlying tree of a {\it  branching random walk}
with the same increment distribution as before. The label process of this branching random walk $\bL_n^{(2)}$, 
 can in turn be used, after conjugation, as the height process of the third tree $\btnj n 3$, which will be used as the underlying tree of a new branching random walk... These iterations allow building successively, $\btnj n 1, \bL_n^{(1)}, \btnj n 2, \bL_n^{(2)}, \btnj n 3, \bL_n^{(3)}, \ldots$ \par 
For such a tree $\tt n 2$ constructed ``on'' a tree $\tt n 1$, since the process used to define the height process of $\tt n 2$ is the sequence of labels of the corners of  $\tt n 1$, and since each node of $\tt n 2$ (but its root) comes from one of the corners of $\tt n 1$ (see Fig.~\ref{fig:wlt}), there is an obvious way to identify the nodes of $\tt n 2$: identify two nodes of $\tt n 2$ if they come from different corners of the same node of $\tt n 1$ (in Fig.~\ref{fig:wlt}, the nodes of each one of the following sets $\{A,C,G\},\{B\},\{D,F\},\{E\},\{H,J\},\{I\}$ will be identified).
Hence, when a sequence of trees $(\tt n 1,\ldots,\tt n \D)$ has been defined as above, it is possible to identify the nodes of $\tt n j$ using the corners of $\tt n {j-1}$ (which then defines a non-crossing partition of the corners of $\tt n j$), as wanted.\\
\bls\, The distribution of each process involved can be described, and the asymptotic distributions characterized. In fact,  the construction is even simpler in the continuous setting, because many combinatorial details disappear in the limit. In a few words: start from $\btj  1$, Aldous' continuum random tree. This tree is then used as the underlying tree of a spatial\footnote{``Spatial" is an adjective that is used to distinguish the different processes into play: prosaically, it is just a usual linear Brownian motion.} Brownian branching process, which amounts to equipping  $\btj1$ with a compatible spatial Brownian motion indexed by the branches of  $\btj1$. 
In the literature, the tree  $\btj1$ equipped with this Brownian labeling is called Brownian snake with lifetime process the normalized Brownian excursion $\se$. The process $\se$ is the contour process of the underlying tree  $\btj1$, and the label process $(\Bell_x^{(1)}, 0\leq x \leq 1)$ in this setting is the process that gives the values of the spatial Brownian motion in accordance with the Brownian excursion. To iterate the construction, it suffices to conjugate the label process in order to get a non-negative process, ${\bh}^{(2)}$, which can be used as the height process of a continuum random tree $\btj2$, which in turn, can be used as the underlying tree of a Brownian snake with label process $\Bell^{(2)}$, and so on. Iteratively, we construct $\btj1,\Bell^{(1)},\bh^{(2)},\btj2,\Bell^{(2)},\bh^{(3)},\btj3,...$.
We believe that these iterated Brownian snakes are interesting on their own. Once these objects are defined, it is possible to use them to define iterated continuous random  feuilletages: for instance $\RR{3}$ is defined as the random tree $\btj3$, whose corners are identified using $\btj2$, whose corners are in turn identified using $\btj1$, in the same way as we proceeded in the discrete setting (see Fig.~\ref{fig:Bij-Sc12}).
\begin{figure}[h!]
\centering  \includegraphics[scale=0.25]{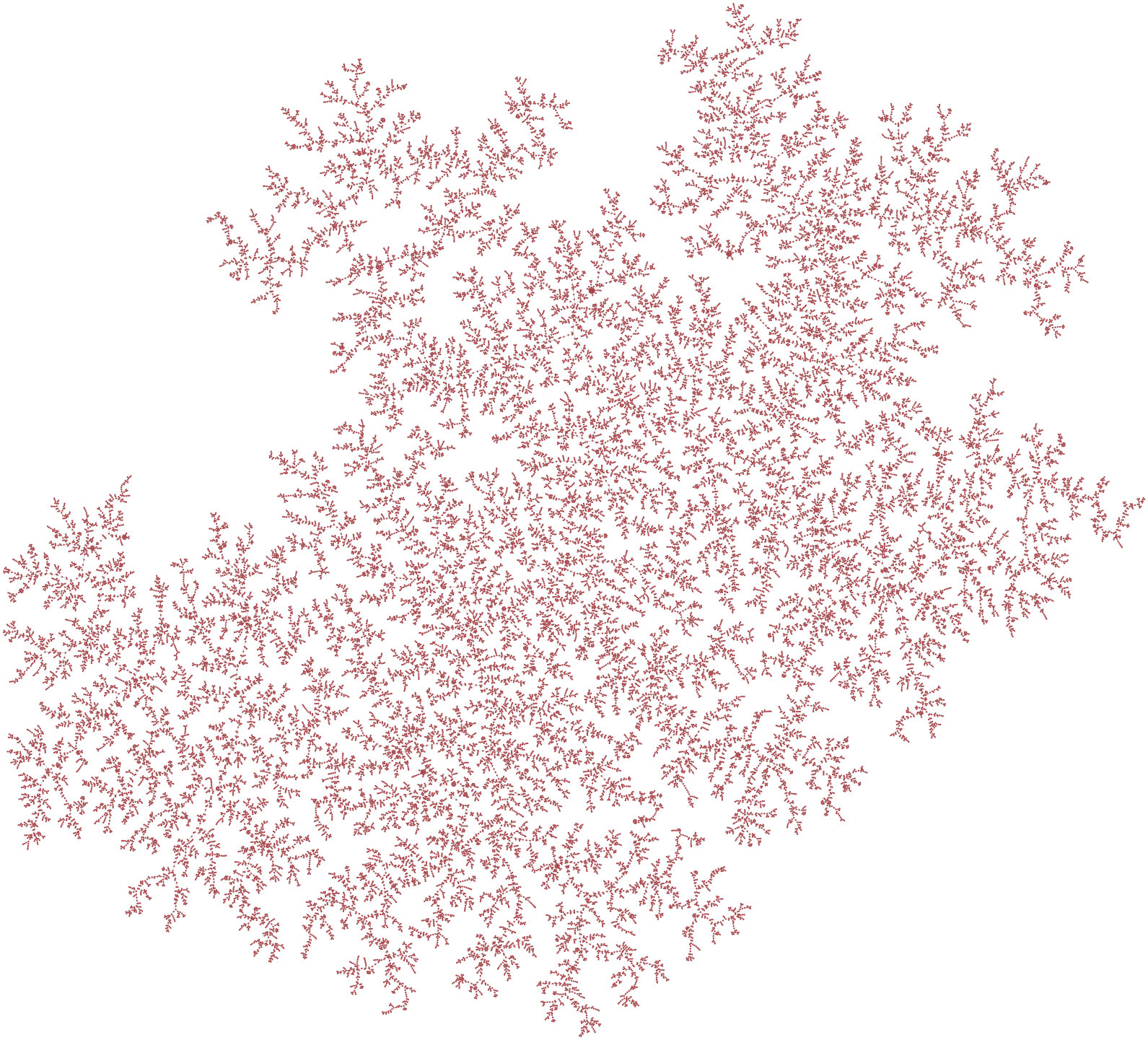}
\qquad
\includegraphics[scale=0.13]{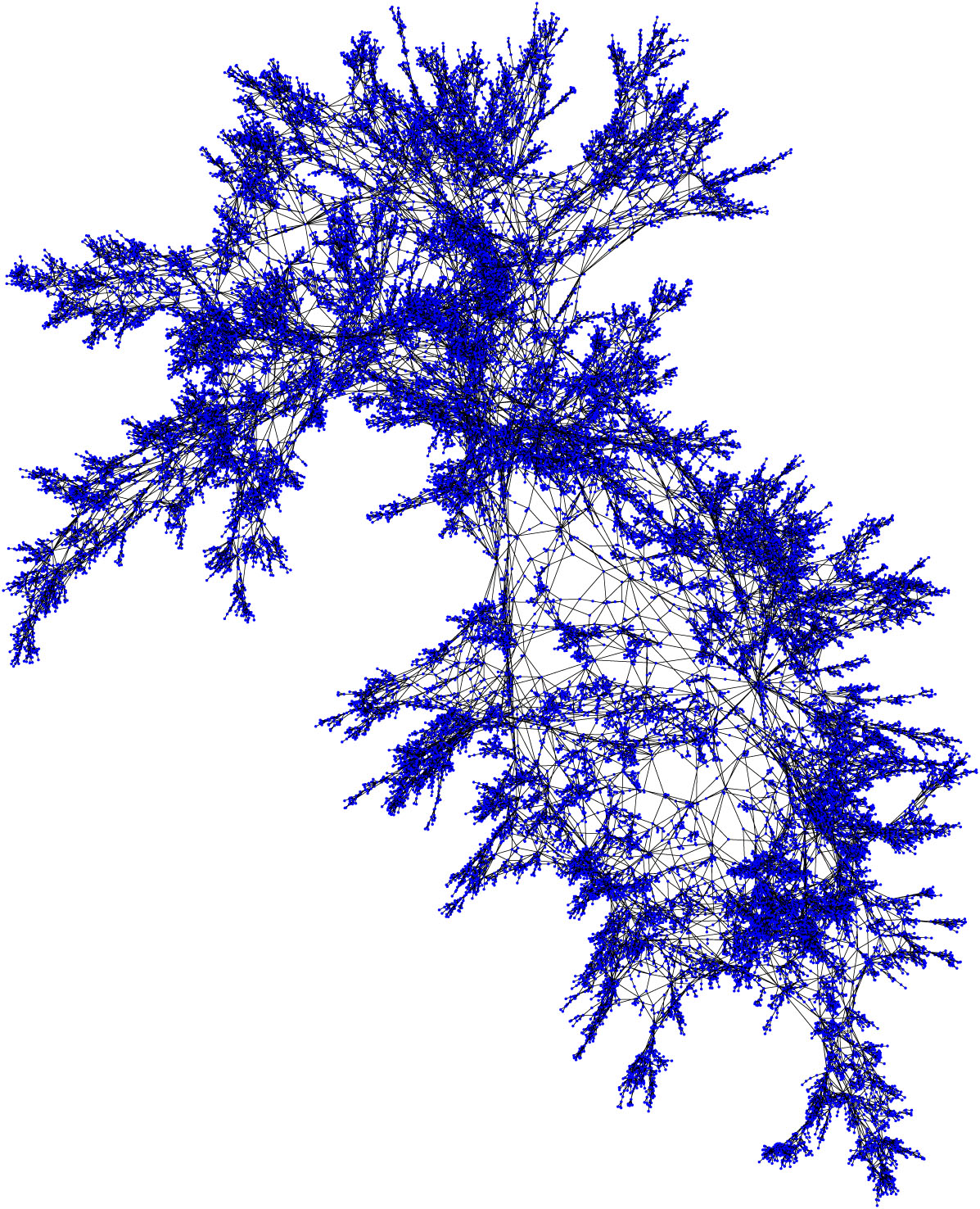} 
\qquad
\includegraphics[scale=0.13]{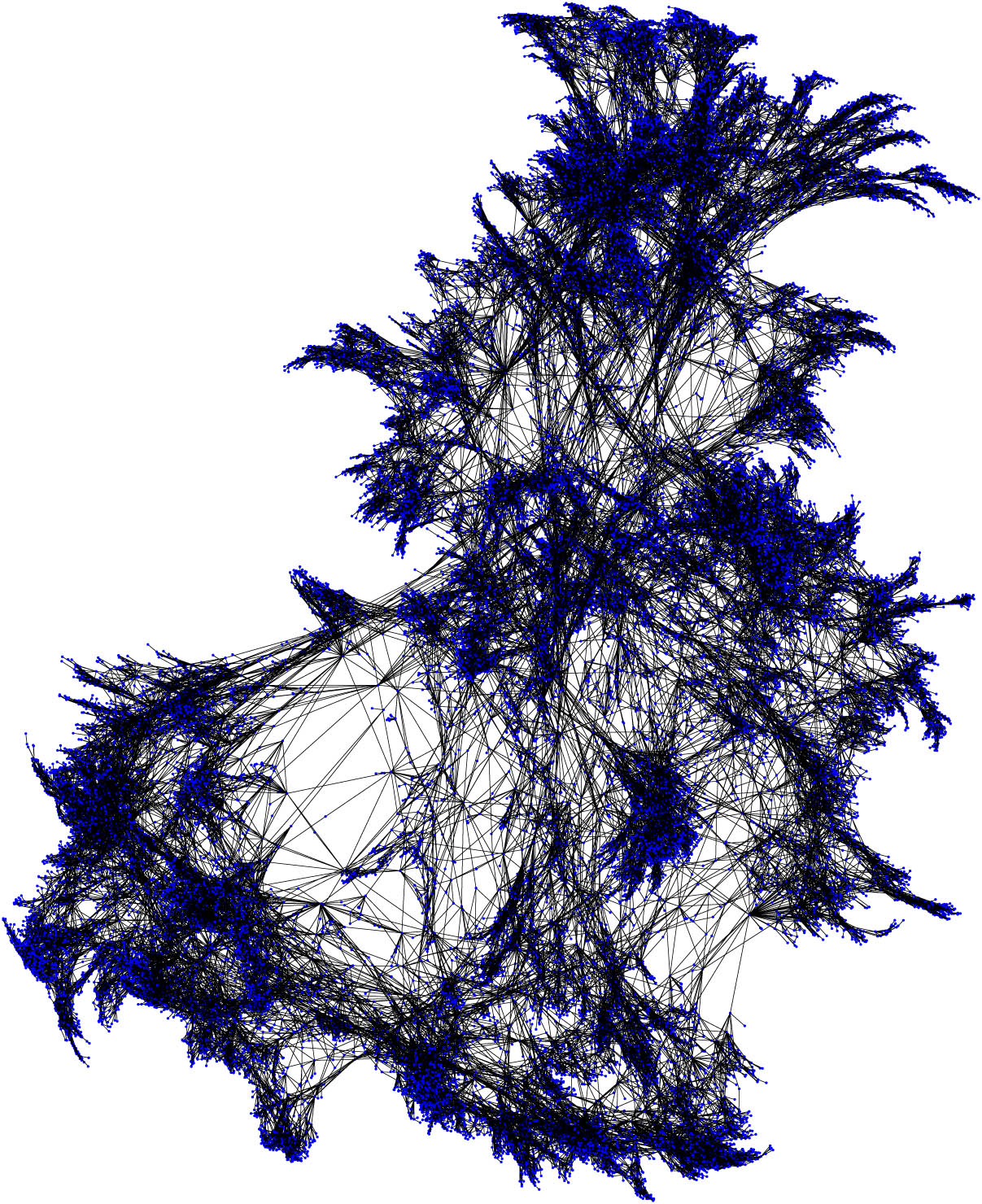}~\\

\centering \qquad\includegraphics[scale=0.125]{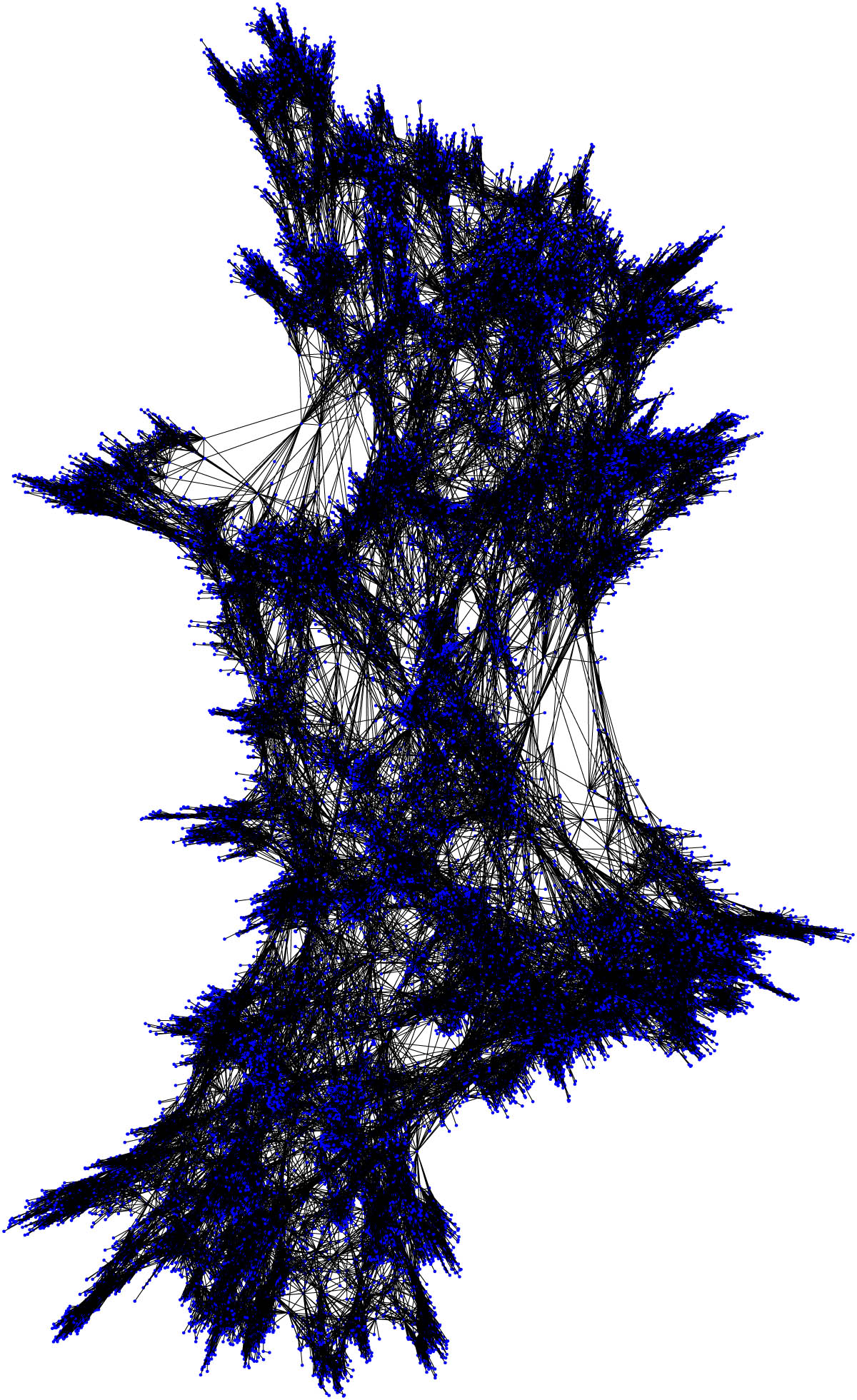} 
\qquad
 \includegraphics[scale=0.135]{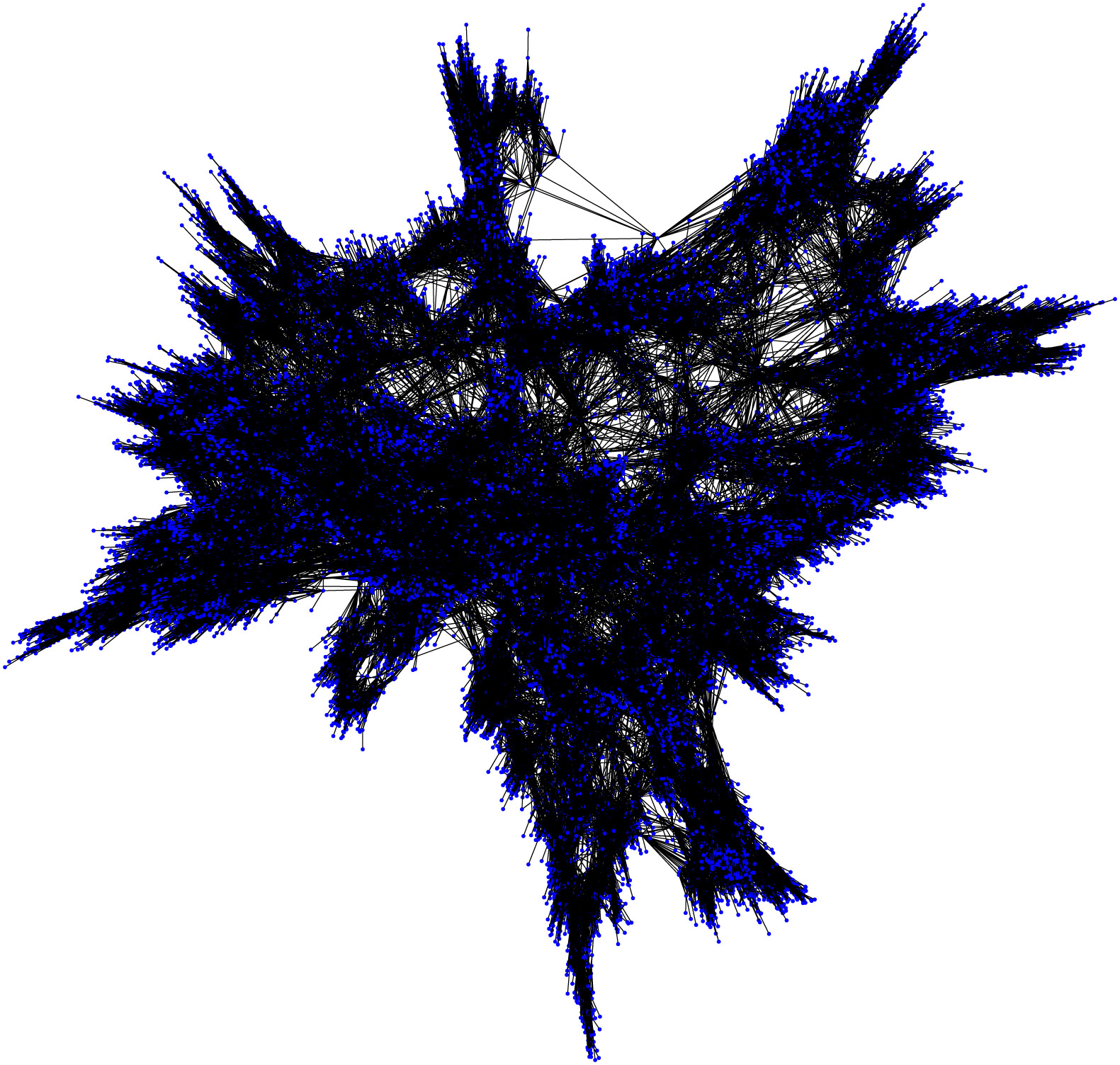}
 \quad
\includegraphics[scale=0.135]{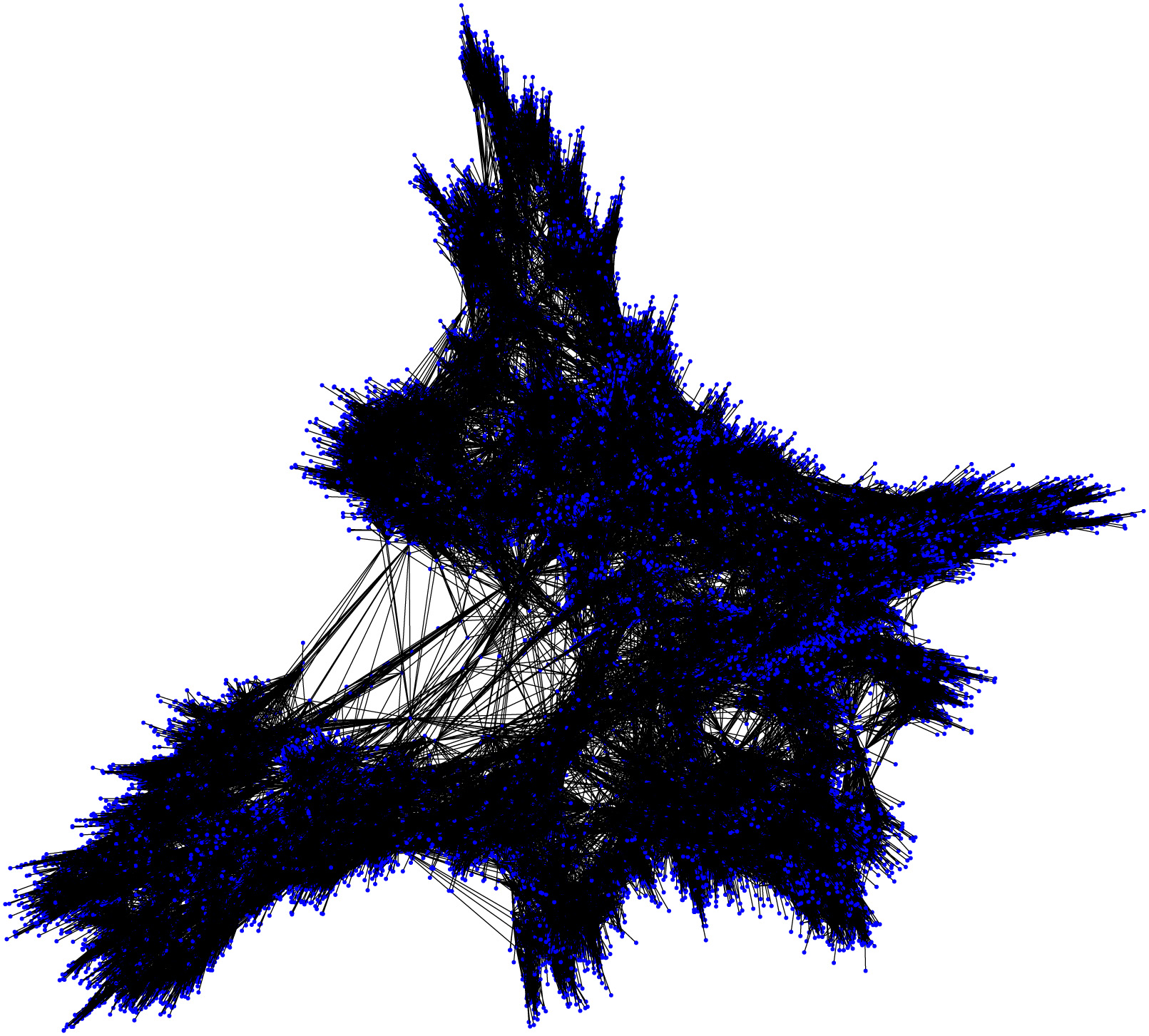}
 \caption{\label{fig:Bij-Sc12}Some images of randomly generated  $\Rn n \D$ for $ \D$ going from 1 to 6 (loops and multiple edges have been removed, keeping a unique single edge between adjacent vertices). The size of the vertex set is $50 000$. }
\end{figure}

\subsection{Folding a planar map to get a 3D object?}
This informal section is devoted to explaining that some constructions similar to that of the elements of $R_N[3]$ allow constructing ``3D discrete objects'', for a certain notion of $3D$ explained below. \par
 There are many equivalent definitions and points of view on planar maps, and ``the planarity'' is more or less intuitive depending on the chosen representation (see Sec.~\ref{sec:GraphSection}). All representations that amount to gluing polygons, or to drawing a graph on the sphere, lead to a straightforward conviction that planar maps are indeed discrete 2D spherical objects. 
\begin{figure}[!h]
\centerline{\includegraphics{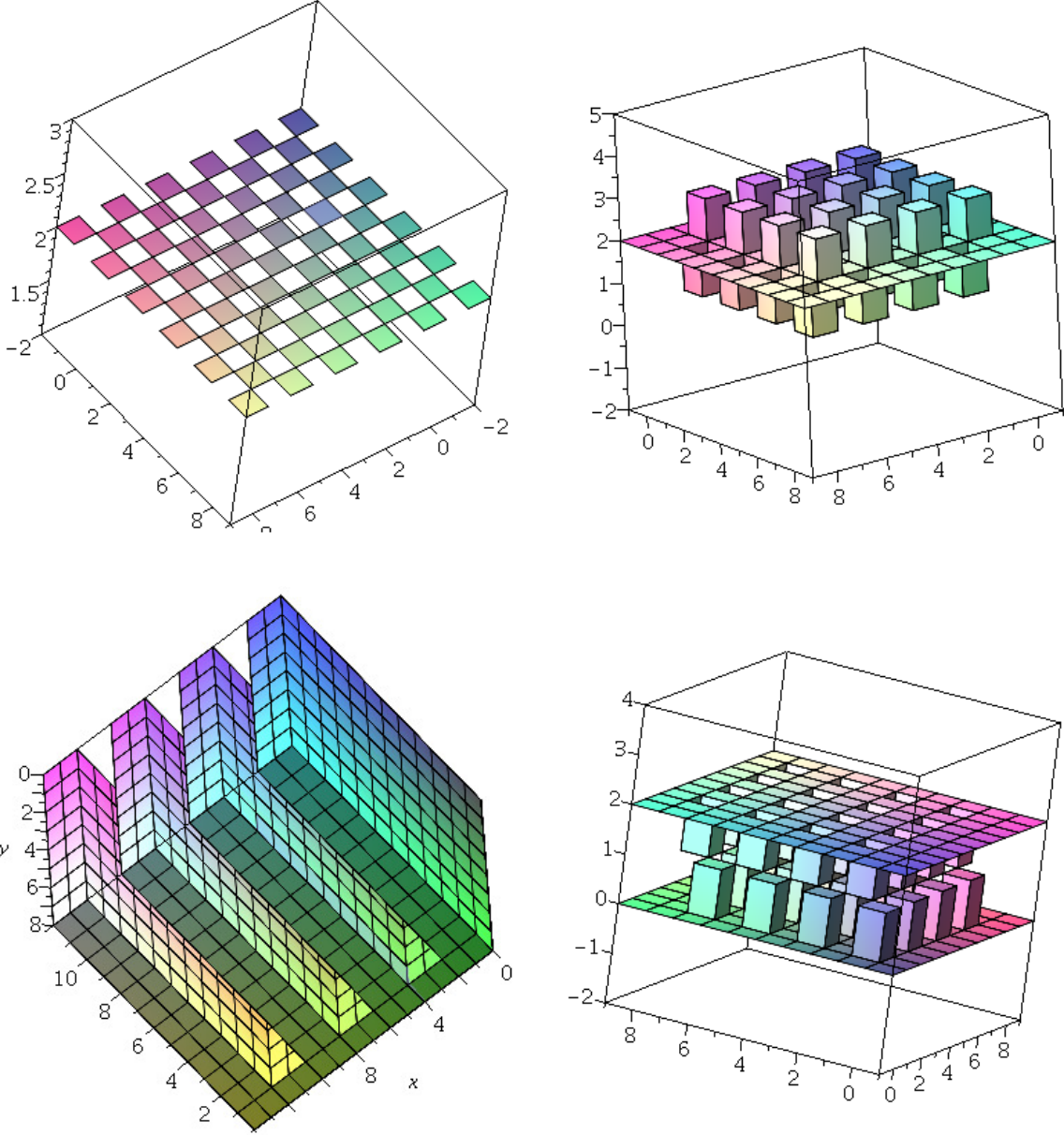}}
\caption{\label{fig:33D} }
\end{figure}

Similarly, if one says that a graph is 2D spherical when it can be drawn on the 2-sphere without edge crossings, a tentative definition of $\D$-dimensional  spherical graph could be that the graph is the 1-skeleton of a $\D$-dimensional cellular (pseudo) complex with the topology of the $\D$-sphere. But this is one possible answer, and even gluing tetrahedra along faces provides objects whose topology is in general very difficult to study.\par

 We propose to discuss the fact that some ``3D spherical like objects'' (in the sense above) possess a representation similar to the third feuilletage.  First, consider the graph $G_n$ having as vertex set the points of the  discrete cube $\{1,...,n\}^3$, and as edges, the pairs of points at Euclidean distance 1 from each other. In the sense above,  $G_n$ is a "3D-spherical-like" graph, as it is the 1-skeleton of a gluing of cubes which discretizes a 3-ball. And more generally, it may be argued that any reasonable definition of spherical graph dimension should give dimension 3 or higher to a graph having $G_n$ as subgraph.
 Now, we will present $G_n$ as a refolded map, very similar to some elements of $R_N(3)$ for some $N$. 

 The idea is the following: consider the chessboard type figure represented in Picture 1,  Fig. \ref{fig:33D}, obtained by taking a section of $\R^2$ (or $\Z^2$), in which a unit square out of two has been removed. Each of the lacking square is then used in Picture 2 as the basis of a cube with the same lacking face: half of these cubes are placed above the plane, half below in such a way that the west-south corners of the cubes above (resp.~below) forms, up to some translation, a section of $(2\Z)^2$. The obtained object is a quadrangulation (with a big square boundary) since all its inner faces have degree 4. Now, change a bit of point of view in Picture 2, and view what is represented as a kind of fabric with a texture: the fabric being the plane, the texture being made of cubes above and below the fabric. Now, imagine a large piece of fabric which is folded and sewn as on Picture 3: it is a quadrangulation with 4 ``layers'', each of them made by two large strips of fabric ``at distance 2''. Each layer is also at distance 2 from the next layer. Here, ''distance 2'' has to be understood for the usual metric in $\R^3$ since clearly this picture can be embedded isometrically in $\R^3$. Now add the texture to the fabric! Picture 4 figures what happens inside one layer (between two strips that are face to face) or between two layers of the fabric (outside the big quadrangulations, between the two strips that are face to face). Up to some details concerning the parity of the strips, the cubes that have size 1 and which are placed on ``planes'' at distance 2, intersect at their vertices only.
  Observe again Picture 4, and imagine the $\Z^3$ lattice in between the two planes: notice that every edge of the lattice belongs to exactly one of the cubes.
  What is shown on Picture 4 represents what happens inside each of the 4 layers, and between 3 inter-layers of Picture 3, so, it concerns a total section of width 13 of $\Z^3$.\par 

  Hence, it is possible to fold and identify vertices in  quadrangulations  to construct some objects whose underlying graphs contain  $G_n$.\par
  As a matter of fact, it may be argued that this example is probably not in the set $\DR{N}3$ for any $N$ large (one of the reasons being that the extracted trees $(t_2,t_3)$ of the textured fabric with the Cori-Vauquelin-Schaeffer bijection provides a tree $t_2$ whose height process has not the required  increments  $+1, -1$ or 0). But we hope that it illustrates the fact that such ``3D-like objects" can arise in the construction we propose.

\subsection{Contents of the paper}

{\paragraph{Convention.}\it All the objects we will introduce in the paper will be rooted, unless specified otherwise (as for instance in Sec.~\ref{sec:pointed}).} \medskip

The last two sections of the introduction are dedicated to a discussion on the context in which our approach takes place, as well as the motivations from theoretical physics.

In Sec.~\ref{seq:IBSIM}, after providing some notions about real trees (Sec.~\ref{sec:realtrees}) and snakes (Sec.~\ref{sec:brown-snake}), we describe directly in the continuum the iterative construction of the $\D$th Brownian snake $\rbs {} [\D]$ and the $\D$th random tree $\btj  \D$ (Sec.~\ref{sec:Iter-snakes}), and then of the $\D$th random feuilletage $\RR{\D}$ (Sec.~\ref{sec:Iter-Refold}). 

We start Sec.~\ref{sec:ISIMCO} by defining planar trees and their encodings using height and contour processes (Sec.~\ref{sec:qsdqsd}), and then iterated discrete snakes  $\RBS_n[ \D]$ and the iterated random discrete trees $\btnj n \D$ (Sec.~\ref{sec:Iter-snakes}), before introducing   the iterated random discrete feuilletages $\DR{n}{\D}$ (Sec.~\ref{sec:Iter-Refold}), as well as their normalized versions (Sec.~\ref{sec:nvotk}).

While the objects introduced in the previous section are rooted, Sec.~\ref{sec:pointed} is dedicated to the definition of {\it pointed} trees (Sec.~\ref{sec:pointed-snakes}), snakes (Sec.~\ref{sec:pointed-snakes}) and feuilletages (Sec.~\ref{sec:measure}). In this last section, we show the convergence in law of the pointed iterated random discrete feuilletages to the pointed iterated continuous random feuilletage  for a topology characterizing the convergence of the encoding trees (which does not imply the convergence for the Gromov-Hausdorff distance).

The proofs of the results in the previous sections are gathered in Section \ref{sec:proofs}.

While we rather use the encodings of trees by processes in the previous sections, in Sec.~\ref{sec:GraphSection} we review in the combinatorial map picture the objects introduced previously. We define a class of combinatorial objects that generalize combinatorial maps using nested non-crossing partitions and that contains the iterated discrete feuilletages. 

In Sec.~\ref{sec:Simulations}, we describe the simulations we performed in the aim of providing pictures and of making some statistics on the asymptotic diameter and on the volume of the balls of our random feuilletages, and comment on the problems encountered.

\subsection{Some references and comments on the approach}

The point of view we develop here could appear somewhat artificial, but it is actually very similar to the first works on the Brownian map \cite{CS,MM}. Planar quadrangulations and other similar simple families of maps were the only objects whose combinatorics was well understood at this time, 
 particularly thanks to the existence of a bijection (CVS) between the set of planar quadrangulations with $n$ faces and some sets of labelled trees (Cori \& Vauquelin \cite{CV}, Schaeffer \cite{Schae}). Building on this, Chassaing \& Schaeffer \cite{CS} made the first connection between the scaling limit of uniform planar quadrangulations and the Brownian snake with lifetime process the normalized excursion, from what they obtained the right scale $n^{1/4}$ for the diameter of uniform planar quadrangulations with $n$ faces.  The second author and Mokkadem introduced in \cite{MM} the Brownian map as the rescaled limit of random quadrangulations with $n$ faces (under a distribution close to the uniform distribution). The question of the convergence to the Brownian map for a nice topology implying the metric convergence (as the Gromov-Hausdorff topology) was known by the authors, but out of reach at this time. Another topology was used, a topology absolutely faithful to the CVS bijection: building on the latter, \cite{MM} proved that the set of quadrangulations with $n$ faces is in bijection with a set of pairs of trees $(\tt n 1, \tt n2)$ 
 as presented above.  
The pair of random trees $(\btnj n 1,\btnj n 2)$ associated with unif.~planar quadrangulations with $n$ faces converges in distribution after normalization to a limiting pair of continuous random trees $(\btj  1, \btj2)$. 
The Brownian map is defined in \cite{MM} to be  $\btj2$ 
quotiented by the non-crossing partition defined by $\btj1$ 
in a way analogous to what is done in the discrete case. 
This approach provides a direct construction of the Brownian map equipped  with a natural distance.\par

In \cite{MatingTrees}, Duplantier, Miller and Sheffield (DMS) investigate random surfaces obtained from the mating of two continuum random trees $T_{\se}$  and $T_{\tilde{\se}}$  encoded by two independent normalized Brownian excursions $\se$ and $\tilde{\se}$ (see Sec.~\ref{sec:ISIMCO}). 
They then identify in these trees the nodes with corner $t$, for all $t\in[0,1]$. If $(E_0,d_0)$ and $(E_1,d_1)$ are two metric spaces, and $\sim_R$ is an equivalence relation on $E_0\cup E_1$, then the quotient space $E=(E_0\cup E_1)/ \sim_R$ is a topological space. One can try to define a distance $d$ on $E$ by:
for $x,y\in E$,
$$
d(x,y) = \inf_k ~\inf_{x_1,\ldots,x_{2k}} ~ \inf_{(\varepsilon_j,\, 1\leq j \leq k) \in \{0,1\}^k} ~ \sum_{j=1}^{k} ~d_{\varepsilon_{j}}(x_{2j-1},x_{2j}),
$$
where $x_1\;\sim_R\;x,~ x_{2k}\; \sim_R\;y$, and, for any $j$, $x_{2j} ~ \sim_R~ x_{2j+1}$, and $x_{2j-1}$ and $x_{2j}$ are both elements on the same set $E_{\varepsilon_j}$. 
In other words, geodesic paths in $E$ are ``limits'' of connected paths in $E$ formed by alternating sections totally included in $E_0$ or in $E_1$. 
It turns out that this way of defining a distance on $E$ fails in general because it may happen that $d(u,v)=0 \not\imp u=v$ (as detailed in footnote \ref{foo:quo}).
In the case of the mating of trees, which is known to be homeomorphic to the $2$-sphere,  DMS do not consider the metric induced by this construction, but consider the topological properties of this space, together with a special space-filling path coming from the construction; they call this space the peanosphere. They study some stochastic processes defined on this rich structure, and make many connections, among others with the Brownian map, Gaussian free field, and Liouville quantum gravity in 2D. \medskip

Hence, the Brownian map has been defined before the proof of the convergence of its inner metric (Le Gall \cite{legall2013} and Miermont \cite{Miermont2013}),  before the proof of its property to be homeomorphic to the 2-sphere (Le Gall and Paulin \cite{LGP}, Miermont \cite{miermont2008}), and also before the proof of its connection with Liouville quantum gravity. To define the Brownian map before these considerations was probably a necessary step in this research field since knowing the limit even for a ``bad topology'' is always an advantage. In the same way the peanosphere is constructed by taking ``a formal topological'' limit of some discrete analogue constructed using binary trees with $n$ leaves, somehow independently of standard considerations concerning invariance principles since no proof of convergence is given: only ``the limit'' is considered.

What we propose here is to proceed as in the first construction of the Brownian map: we present a combinatorial model and its limit (the $\D$th random feuilletage). The first properties we are able to prove provide some first clues that this construction could indeed be analogous to the Brownian map. We hope that it could also lead to the construction of some peanospheres in higher dimensions.

\subsection{Motivation from theoretical physics.}
\label{sec:intro-phys}

 From a theoretical-physics perspective, the definition of an analogue of the Brownian map in higher dimensions is sought in the context of discrete approaches to quantum gravity, which aim at describing gravity at a microscopic ``quantum" level\footnote{Here we consider discrete quantum gravity as a statistical field theory: we use path integrals of the form $\int e^{-S}$ and not $\int e^{iS}$.}. In theories  such as  dynamical triangulations \cite{Ambjorn1992, David1992, QuantumGeom} or random tensor models \cite{GurauInvitation, GurauBook}, some family of $\D$-dimensional discrete spaces such as  $\D$-dimensional triangulations (simplicial pseudo-complexes), is seen as the set of random discrete space-times that can occur at a microscopic level. Their probabilities of occurrence are provided by Einstein's general relativity, or more precisely by discretizing {\it \`a la Regge} the Einstein-Hilbert action, which is the field-theoretical formulation of general relativity. 
Thereby,  the theory of gravity induces a distribution\footnote{Strictly speaking, each  $\D$-dimensional triangulation is assigned a positive weight which provides a classification of the discretized space-times. However, these weights cannot always be normalized to define a distribution.} over the set of  $\D$-dimensional triangulations of the same size, whose weights depend exponentially on their discrete curvature\footnote{A ``canonical" discrete curvature \cite{Regge} is defined on a  $\D$-dimensional triangulation by assuming that all edges have the same length. Then, the discrete curvature locally depends only on the number of $\D$-simplices around each $(\D-2)$-simplex.} (see e.g.~\cite{QuantumGeom, Lionni:17}).

Such theories, which attempt to describe quantum gravity as a continuum limit of a statistical system of random discrete space-times, are seen in analogy with the description of gas thermodynamics in statistical physics in the grand-canonical ensemble: the simplices are viewed as ``particles of space-time" and the  $\D$-dimensional triangulations as accessible states for the statistical system, each with a given Boltzman weight, obtained by discretizing the Einstein-Hilbert action. While this justifies this elegant approach, it is in fact the continuum/scaling limit that would be physically interpreted as a quantum theory of gravity, if it had consistent properties. One would then like to recover general relativity as an effective theory by taking a suitable limit from the scaling limit of these random discrete space-times\footnote{\label{foo:coarse} If these random discrete space-times converge towards a certain continuum limit (scaling limit),finding out if general relativity is recovered as an effective theory in a certain ``non-quantum" limit could involve defining suitable observables on the scaling limit, that would converge to their classical (i.e.~non-quantum) values throughout a coarse-graining process, or knowing how to describe this continuum limit in a field theoretic way, and then renormalizing this theory to translate it to our scales. But there are no known spaces so far to serve as toy-models for addressing this question. }. 

 For $\D>2$, in a certain regime (small Newton constant), this distribution selects a narrow class of very highly curved $\D$-spheres\footnote{All results discussed here are in the Euclidean case, in which time is not considered. Introducing ``time" can be done by requiring some additional causality condition on the  $\D$-dimensional triangulations \cite{AMBJORN2001CDT, LollCDTReview}. Numerical simulations seem to indicate that the continuum limit in dimension 4 has promising properties, however no exact result exists so far.} whose scaling limit  is the CRT \cite{Ambjorn1992, Thorleifsson1999, Gurau2014} (this is called the branched-polymer phase in physics).  
Other distributions over the set of  $\D$-dimensional triangulations have also been investigated, for instance based on topological criteria. A possible candidate could indeed have been the full set of triangulations of the $\D$-sphere, however this seems to lead to very singular $\D$-dimensional triangulations asymptotically, whose diameter are bounded \cite{Ambjorn1992, Thorleifsson1999}. This regime is called the crumpled phase in the physics literature, and it is expected that no scaling limit can be defined.
 In a well-defined continuum limit, one would like to recover something which resembles a random  emergent continuum ``$\D$-dimensional" space-time (as discussed previously, this is a very vague statement, however for many reasons, we should rule out the CRT and the crumpled phase).  The important question would then be how to find out whether it leads to general relativity in a certain limit (see the footnote \ref{foo:coarse}), however until that day, there are no known examples of continuum random spaces that could  serve this purpose. At this level, regardless of the precise notion of dimension for scaling limits of random graphs, even the construction of scaling limits of random graphs that are neither trees, nor surfaces of any genus, would be an  important step forward, by providing toy-models to address this question.

On the other hand, in dimension $D=2$, the link between combinatorial maps, matrix models, and later the Brownian map on one hand, and quantum gravity in dimension $\D=2$ on the other, has been investigated since the 80's \cite{MatrixReview, KPZ, CFTQG2, CFTQG3, DuplSheff}. It was then proven in 2016 \cite{MilSheff, MilSheff2, MilSheff3} that the Brownian map is indeed equivalent to Liouville quantum gravity \cite{Liouv}, a theory of random continuum surfaces introduced by Polyakov in the context of string theory \cite{Polya}. However, the proof of this equivalence does not rely on the fact that the initial distribution for the random maps is induced by general relativity. {Indeed, what matters in this framework is the convergence to the Brownian map (both for prooving the equivalence with Liouville quantum gravity and for the physical interpretation), which is ensured as long as, at the discrete level, only spheres 
 (families of combinatorial maps in $D=2$ with genus 0)
   are considered, with a uniform distribution for discrete spheres of the same size \cite{legall2013, BrownMapUniform, abraham2016, BrownMapOdd}. But spheres could  be selected from other criteria and distributions than that induced by Einstein-Hilbert\footnote{In which case the surfaces of genus 0 are selected in the regime of small Newton constant.}, e.g.~topological\footnote{It is equivalent in dimension 2 but not in higher dimensions.}, and only the physical interpretation at the discrete level would eventually be affected, not the conclusions at the continuum level.  
   Pushing this reasoning further, the convergence towards the Brownian map is universal: modifying the discrete models \cite{marckert2007InvPrinc, FirstPassPerc, MarzoukDegree}    
 leads to the same scaling limit, and thereby to Liouville quantum gravity (and this could even be expected for other models of random metric spaces). 
   Because of the failure in higher dimensions of the approach consisting in discretizing the Einstein-Hilbert action and using this distribution on  $\D$-dimensional triangulations, and because it is currently unclear how to find a class of discrete 3-dimensional spheres that could have a suitable scaling limit, if it exists,  the facts listed above provide a} strong motivation for trying to build interesting random continuum spaces from a more direct approach.

This is the aim of the present paper: working the other way around, we build directly candidates $(\RR{\D},\D\geq 1)$  to play the role of the Brownian map in higher dimensions $\D$, as limits of the $\D$th random discrete feuilletages $(\DR n{\D},\D\geq 1)$, which are obtained by iterated series of foldings of an initial discrete surface. Even if our construction does not rely on some classical combinatorial objects such as $\D$-dimensional triangulations, the discrete objects $\DR n\D$ we present have many good combinatorial properties, starting with the coincidence of $\DR n1$ with uniform rooted planar trees with $n$ edges, and of $\DR n2$ with rooted-pointed uniform quadrangulations with $n$ faces. 
    We stress again that  to our understanding, it is the continuum/scaling limit that would be understood as a quantum theory of gravity, not the discrete spaces, and if a theory was to be built directly from a suitable random continuous generalization of the Brownian map, and if this theory provided consistent results, for instance when defining observables and extracting their classical  (i.e.~non-quantum)  limits, then huge progress would have been done towards quantizing gravity, whether this random continuous generalization of the Brownian map was obtained as a limit of  $\D$-dimensional random triangulations or not. 
For a suitable definition of our discrete objects, we faced the difficult task of qualifying what a good limit would be. We isolated three important features they should have or that their limit should have, to be  suitable in the context of quantum gravity: \\
\bls\, The (limiting) continuous space's ``dimensions'' should be suitable in some sense. As mentioned above, it seems to us that the topological dimension could be a good notion, if it is well defined. {In dimension 3,} this excludes, for instance, the CRT or the Brownian map.\\
\bls\, The distances in the $\D$th large discrete random space should scale in a suitable way: for $\D\geq 3$, we do not expect the distances to scale as in large uniform planar trees or as in large planar maps, for instance. Also, we expect the Hausdorff dimension of the $\D$th limiting continuous object  ${\sf Haus}(\D)$ to be an increasing function of  $\D$.  As a matter of facts, we conjecture the distances in $\DR n\D$ (under a particular distribution we define) to be of order $n^{1 /{2^\D}}$ (in any case we will see that  $n^{1 /{2^\D}}$ is an upper bound on the diameter of $\DR n\D$).
This value seems suitable for the $\D$th discrete object.  \\
\bls \,``Uniformity" and universality: as mentioned above, in order to have a physical interpretation at the discrete level, it seems important to be able in the future to exhibit some sets of $\D$-dimensional triangulations selected according to ``natural criteria'' (e.g.~uniform in a set of triangulated spheres selected according to a physically motivated criterium), which would converge towards the space $\RR{\D}$ we built.
The construction of $\DR n\D$ for $n\geq 3$ is not given in these terms, that is in terms of gluing of  $\D$-dimensional simplices. However, our  iterative construction generalizes those of Schaeffer or BDG \cite{Bouttier2004}
  between labeled trees and maps, valid in $\D=2$. While these bijections produce uniform maps, they are not formulated in terms of gluings of elementary building blocks (such as polygons). \\
  
    It is worth noting that the critical exponent we find for the $\D$th random discrete feuilletages, associated with their asymptotic enumeration (called string susceptibility in physics), is
  \ben
  \gamma_\D=3/2 - \D.
  \een
This exponent generalizes the well-known universality class exponents $\gamma_1=1/2$ for random trees, and $\gamma_2=-1/2$ for random planar maps.

   The approach of this paper is to construct scaling limits of random graphs, in a way which allows keeping track of the distances. While this approach, based on the repeated use of the Cori-Vauquelin-Schaeffer bijection, renders the question of the topology of the scaling limit less intuitive than when considering scaling limits of discrete topological spaces, it is quite clear that the random feuilletages are neither random trees, nor random surfaces of any genus. Therefore, {\bf if the random feuilletages are indeed non-trivial, they would  provide the first example of random continuum spaces relevant in the context of $D$-dimensional quantum gravity for $D>2$}, i.e.~with which we could start understanding what it would mean for a Brownian space to obtain general relativity in a certain suitable ``non-quantum" limit (see the footnote \ref{foo:coarse}).

\section{Iterated Brownian snakes and  iterated random feuilletages}
\label{seq:IBSIM}
\setcounter{equation}{0}
All the random variables are assumed to be defined on a common probability space $(\Omega, {\cal A}, `P)$. \medskip 

{\bf Notations : }
We will denote by $\cro{a,b}$ the set $[a,b]\cap \Z$ equipped with its natural order. For an ordered finite set $I$, the notation $X(I)$ stands for the sequence $(X(i), i \in I)$ taken under the index order; hence, $X(\cro{0,5})=(X_0,X_1,\cdots,X_5)$. Finally, we will denote by $(x_n)$ the infinite sequence $(x_1,x_2,\cdots)$. \\
By convention $x+y \mod p$ stands for $(x+y) \mod p$.\par
  For a function $g:\R\to \R$ (or defined on an interval $[a,b]\subset \R$ only),  we denote by
  \[\widecheck{g}(x,y)= \min \{ g(u), u \in [ \min \{x,y\}, \max\{u,v\}]\}\]
  the minimum of $g$ on the interval with extremities $x$ and $y$.

\subsection{Continuous trees}
\label{sec:realtrees}

We start with a digression concerning the so-called ``iterated Brownian motion'':  take a sequence of independent two-sided Brownian motions $(\bB_i ,i\geq 0)$, meaning that for any $i$, $(\bB_i(s), s\geq 0)$ and $(\bB_i(-s),s\geq 0)$ are two standard linear Brownian motions starting at 0.
The  $\D$th Brownian motion (see e.g. Burdzy \cite{Bur93}) is the one dimensional process defined by
\ben
{\bf I}^{(\D)}(t)=\bB_\D (\bB_{\D-1}( \cdots  (\bB_1(t))\cdots)),~~~ t \in \R.
\een
The construction we propose for the $\D$th  Brownian snake $\rbs[\D]$  can be viewed as a kind of tree-like  counterpart to ${\bf I}^{(\D)}$: as explained in Section \ref{sec:PMO}, we will produce a sequence of labeled trees $\bigl(\bt^{(i)},\Bell^{(i)}\bigr)$, building $\bt^{(i+1)}$ thanks to $\Bell^{(i)}$, ``a Brownian labeling of $\bt^{(i)}$''. Up to some changes of roots, constructing $\bt^{(D)}$ will require $D-1$ iterated Brownian labelings.
The construction is tuned in such a way that the pair $(\bt^{(1)},\bt^{(2)})$ corresponds exactly to the random trees encoding the Brownian map.

Before that, we need to recall some facts concerning continuous trees and real trees.

\subsubsection*{Representation of (real) trees using continuous functions} 
The first brick we need to later define iterated snakes and feuilletages is the notion of continuum tree. We need to review some classical aspects of 
 the latter, which can be found in the literature in many references, notably in relation with the CRT or the Brownian map (see e.g.~Le Gall \& Duquesne \cite{DuLG}, Miermont \cite{Miermont2013}, Le Gall \cite{legall2013, LG19} or Miermont \& Le Gall \cite{MJFLG}, ...).

Compact $\mathbb{R}$-trees are compact metric spaces $(T,d)$ such that for every $a,b$ in $T$, there exists a unique injective function $f_{a,b}:[0,d(a,b)]\rightarrow T$, for which $f_{a,b}(0)=a$ and $f_{a,b}(d(a,b))=b$.\par

In the sequel, we present some continuous trees encoded by functions; these objects are rooted-ordered real trees (see Duquesne \cite{Duq} and references therein  
for a complete discussion on the relation between compact real trees and trees encoded by real valued fonctions).

Consider $C[0,1]$, the set of continuous functions $f:[0,1]\to\R$.  Let 
\ben
C^{+}[0,1]= \{ f\in C[0,1]~,~f([0,1])\subset \R^+, f(0)=f(1)=0\}
\een  be the subset of $C[0,1]$ of non-negative functions, null at 0 and 1.
With each function $g$ in $C^+[0,1]$, we define an equivalence relation in $[0,1]$ by
$$x\sous{\sim}{g}y ~~\Longleftrightarrow g(x)=g(y)=\W g(x,y).$$
\begin{figure}[!h]
\centering
 \includegraphics[scale=1]{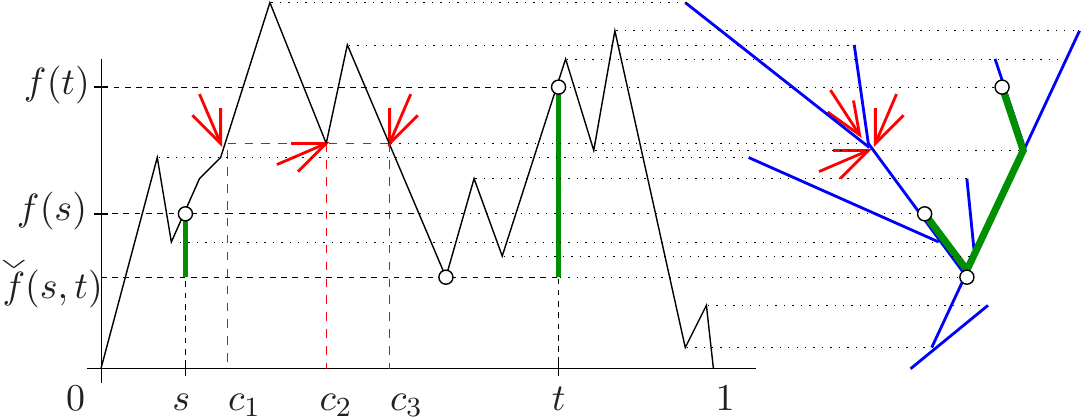}
 \caption{\label{fig:Arbre}Tree associated with a function taken in $C^+[0,1]$: the green path in the tree is visible in the functional encoding;
the red arrows point toward the three corners $c_1,c_2$ and $c_3$ of a vertex in the tree. }
\end{figure}

The set of  equivalence classes modulo $\sous{\sim}g$, 
$$T_g=[0,1]/\sim_g$$ is connected and possesses no cycle: it is a {\bf tree}, and its elements are therefore called {\bf nodes}  or {\bf vertices}. An example is shown in Fig.~\ref{fig:Arbre}. The canonical surjection  $c_g$ from $[0,1]$ into $T_g$ is denoted 
\[\app{c_g}{[0,1]}{T_g}{x}{c_g(x)=\dot x:=\{y\in [0,1]~,~x\sous{\sim }gy\}}.\]
It is the continuous analogue of the depth first traversal $c_T$ for discrete trees defined below  \eref{eq:Cc}. 
Seing $x\in[0,1]$ as the {\bf corner} of a node, $c_g(x)$ is precisely that node (and a node corresponds to a set of corners).
The set of nodes $T_g$ can be turned into a totally ordered set, by setting 
$$
\dot x<\dot y~~ \textrm{ iff }~~\inf \dot x<\inf \dot y,
$$
that is, the order of two nodes is  
inherited from the order of their first corners.
The class $\dot 0$ is the root of $  T_g$: the tree $  T_g$ is {\bf rooted}.\par
Let $x$ and $y$ be elements of $\dot x$ and $\dot y$. The node $\dot z \in  T_g$ defined by $z\in[x,y]$ and $g(z)=\W{g}(x,y)$ does not depend on the chosen representatives $x$ and $y$: $\dot z$ is called the highest common ancestor of $\dot x$ and $\dot y$. 
From there, we can define the notions of {\bf ancestor}, {\bf descendent}, {\bf subtree}, {\bf branch}, as in the discrete case (Sec.~\ref{sec:qsdqsd}). For instance, the set of ancestors of $\dot x$ in $T_g$ is defined as
\ben\label{eq:Ancestors}
{\sf Ancestors}_g(\dot x) = \{\dot y \in T_g~,~\inf(\dot y) \leq \inf (\dot x)\leq \sup(\dot x) \leq \sup(\dot y)\}.
\een
It is also the ancestral branch from $\dot x$. 
A node $\dot x$ is said to be an internal node if $\inf \dot x<\sup \dot x$, and it is said to be a leaf if $\inf \dot x=\sup \dot x$ (for a general function $f$, positions of local maxima correspond to leaves, but some leaves are not position of local maxima).\footnote{The normalized Brownian excursion can be obtained by rescaling the excursion of the Brownian motion which straddles 1. From here, it can be seen that the trajectory of the Brownian excursion inherits from the Brownian motion of many features. For example, it has a countable number of local minima or maxima (Ex.Chap.~III, 3.26 Revuz-Yor). Besides, the Brownian motion has the strong Markov property, and the property that the set $Z=\{t:B_t=0\}$ is a.s.~a closed set without isolated point (see Chap.~III Prop 3.12 Revuz-Yor), allows seeing that a.s., 0 (or any other point $x\in[0,1]$) is not a local maximum or minimum: for any $x\in[0,1]$, one has a.s.~$\inf \dot x=\sup \dot x$, even if a.s.~$x$ is not a local maximum. }

The distance between $\dot x$ and $\dot y$ is defined as
\ben
\label{eq:zgzr}
d_{ T_g}(\dot x, \dot y):=D_g(x,y),
\een
where 
\ben
\label{eq:Dg}
\app{D_g}{[0,1]^2}{\R^+}{(s,t)}{D_g(s,t):= g(s)+g(t)-2 \widecheck{g}(s,t)}.
\een
This map $D_g$ is  well defined since the r.h.s.~of \eref{eq:zgzr} does not depend on the elements $x$ and $y$ chosen in the classes $\dot x$ and $\dot y$. The fact that $d_{T_g}$ is indeed a distance is easy to check.
The function $g$ is called {\bf the contour process} of $T_g$ since 
\ben\label{eq:gcont}
d_{T_g}(\dot 0,\dot x) =D_g(0,x)=g(x), \textrm{ for any } x \in [0,1]
\een
which is the characterizing property of the contour process in the discrete case (see later Def.~\ref{def:contour-discr}).

  \subsection*{Trees as measured spaces}
  Denote by  ${\cal M}[0,1]$ the set of probability measures on $[0,1]$.
\begin{defi}
  Consider $\mu \in {\cal M}[0,1]$, $g\in C^+[0,1]$ and $T_g$ the associated tree. The pair $(T_g,\mu)$ is called a measured tree.
\end{defi}
For $g\in C^+[0,1]$, since $[0,1]$ is the corner set of the tree $T_g$, and $c_g:[0,1]\to T_g$ is the map which sends any corner $x$ on the associated vertex $c_g(x)\in T_g$, the measure $\mu$ is a measure on the corner sets, and its push-forward measure by $c_g$ is a measure on $T_g$.
There are two main reasons to enrich trees with measures, both being linked with discrete trees:\\
\bls\ In the discrete case up to a normalization,  for many models of random trees including those studied in the present paper, the number of nodes visited in the contour process between time $a$ and $b$ is (approx.) a proportion $b-a$ of the nodes (and also a proportion $b-a$ of the corners). Adding a measure component allows accounting for this and then expressing that, at the limit, the same property holds for the continuum random tree $T_g$ (the ``limiting measure'' being the Lebesgue measure).\\
\bls\ The second reason is the need to distinguish between discrete and continuous trees! If $g$ is not the null function, the set $[0,1]/\sim g$ has the cardinality of $\R$, regardless of whether $T_g$ ``is used to model'' a  discrete tree or not. Hence, when one embeds the set of discrete trees in the set of continuous trees using contour processes, their discrete nature is lost. Corner measures allow recovering corner positions, and then allow one to cover discrete and continuous objects by a single notion, which is sometimes compulsory (to prove convergence results, for example).
~\\
Let 
\[K=\{T_g:=(T_g,d_{T_g},\mu_g)~,~g \in C^+[0,1], \mu \in {\cal M}[0,1]\}\] be the set of such  rooted trees, considered as metric spaces, and equipped with a corner measure (in the following, we use the same notation for a real tree and the corresponding measured metric space).
\par
The set of trees $K$ is a metric space: we transport the metric and the topology from $(C[0,1],\|.\|_\infty)\times ({\cal M}[0,1],\dVar)$ (where $\dVar$ is the total variation distance) onto $K$, by setting the following distance on $K$: for $g$ and $f$ in $C^+[0,1]$,
\[d_K(T_g,T_f)=\|g-f\|_{\infty}+\dVar(\mu_f,\mu_g).\]
This makes of the set of trees $K$ a Polish space.

We define formally what we will call Aldous' continuum random tree (Aldous' CRT):
\begin{defi}We call Aldous' CRT, the tree $T_{\se}=(\se,d_{\se},\lambda)$ where $\lambda$ is the Lebesgue measure on $[0,1]$  (and where $\se$ is a Brownian excursion).
\end{defi}
We end this introduction to continuous trees by defining the change of root:
\begin{defi}\label{def:rer}
  For $g\in C^+[0,1]$ and $x\in[0,1]$, the tree $  T_g$ rerooted at its corner $x$ is the tree $T_h$ for 
  \[\left\{\begin{array}{ccl} h(s)&:=& D_{g}(x+s \mod 1,x), \textrm{ for } s \in [0,1],\\
             \mu_h(.) &=& \mu_g (x+. \mod 1) 
             \end{array}\right.. \] 
\end{defi}
This definition fits perfectly with \eref{eq:gcont}, since the distance to the corner $x$ in $T_g$ is indeed the function ``distance to the root" in $T_h$.

\begin{rem}[Important]\label{rem:mesvsnotmes} All along the paper we will encounter many continuous trees $T_g=(g,d_g,\mu_g)$ for which the corner measure {\bf will always be the Lebesgue measure on [0,1]}. We will also work with some normalized discrete trees with $N$ edges: in this case, the corner measure {\bf will always be the uniform corner measure}, that is, the  deterministic measure $\lambda_N=1/(2N) \sum_{k=0}^{2N-1} \delta_{k/(2N)}$ on $[0,1]$. When $N\to +\infty$,
\ben\label{eq:convLambda}
\lambda_N\to \lambda
\een 
for the classical weak convergence in ${\cal M}[0,1]$ (equipped with the total variation distance), so that this additional ``measure component'' does not modify the proof of convergence for trees, for snakes, and after that for feuilletages. Nevertheless, the presence of the measures allows defining properly the feuilletages, whose definition, in the discrete case, must take into account the actual nodes locations.\\
To avoid heavy notations, we will however often drop the measure component, but will recall its presence when it is crucial. In the proofs, the presence of measures will be simply completely dropped, since \eref{eq:convLambda} alone allows taking care of the convergence of the measure components.
\end{rem}

%%%%%%%%%%%%%%%%
\subsection{Brownian snake}
\label{sec:brown-snake}
The Brownian snake is a classical object from probability theory, which before being used to defined the Brownian map in \cite{MM}, was mainly used in relation with superprocesses (see e.g.~\cite{DuLG}). To build the random  iterated feuilletages, we will define ``a notion of iterated snakes"; first, let us review some of the aspects of the Brownian snake. For references on (non-iterated) snakes, see e.g.~Le Gall \cite{legall2005} for continuous snakes, and  Marckert \& Mokkadem \cite{MMsnake}, Janson \& Marckert \cite{janmarck05} for discrete snakes and their convergence.\par
Consider the following set (of ``bridges'')
\ben
C^{0}[0,1]=\{g \in C[0,1]~,~g(0)=g(1)=0\}.
\een
\begin{defi} 
Let $g \in C^+[0,1]$ and $  T_g$ be the associated tree. A labeling of the rooted tree $  T_g$ is a map $\ell \in C^0[0,1]$ which satisfies
\ben\label{eq:ww}
 s\sim_g t \imp \ell(s)=\ell(t).
\een 
In other words, the labels of the corners of a  node coincide.
A pair $(g,\ell)$, where $\ell$ is a labeling of $  T_g$, is called tour of a continuous snake. 
We denote by
\ben
\RS = \l\{ (g,\ell) \in C^+[0,1] \times C^0[0,1]~,~\ell \textrm{ is a labeling of }   T_g\r\}
\een
the space of tours of continuous snakes, equipped with the uniform convergence topology.
\end{defi}
  \begin{rem}
As detailed previously, a corner measure is sometimes considered, so that the elements of $\RS$ will sometimes be viewed as 3-tuples $(g,\ell,\mu)$ instead (here $\mu \in{\cal M}([0,1])$, equipped with the distance
    \[D((g,\ell,\mu),(g',\ell',\mu'))=\|g-g'\|_{\infty}+\|\ell-\ell'\|_{\infty}+\dVar(\mu,\mu').\]
    To avoid too much heaviness, we remove this third component as long as it is not explicitly needed.
    \end{rem}
    In the literature, a snake is the name given to a family of trajectories $(w_x,x\in[0,1])$ indexed by the corners of a tree. The snake and its tour are related as follows. Taking $(g,\ell)$ in $\RS$, the snake with tour $(g,\ell)$ is the family of trajectories
$(w_x,x \in [0,1])$, where the lifetime of $w_x$ is $g(x)$, and 
\ben
w_x(h)= \ell(z)\quad  \textrm{ for } 0\leq h \leq g(x),
\een
where $z$ is one of the corners of the ancestors of $\dot x$ at height $h$. 
\begin{rem}\label{rem:SvsTS}
The natural maps that associate tours of snakes and snakes (both ways) are homeomorphic under natural topologies (see Marckert \& Mokkadem \cite{MMsnake}). The topology on the set of snakes is more involved than that 
for tours of snakes, since snakes are families of killed trajectories, when tours of snakes are just elements of $C^+[0,1]\times C^0[0,1]$. The homeomorphism evoked above makes it possible to transfer all convergence results obtained on tours of snakes to snakes. In the following, we will only deal with tours of snakes and we will often call them simply snakes, by abuse of language.
\end{rem}

\begin{defi}\label{def:TBS} Consider $\bX$ a random process taking its values in $C^+[0,1]$; we call tour of the Brownian snake with lifetime process $\bX$, the pair $(\bX, \Bell)$ where, conditionally on  $\bX=g\in C^+[0,1]$,   the process $(\Bell(s),s \in [0,1])$ is a centered Gaussian process with covariance function
\ben\label{eq:zfz}\cov(\Bell(x),\Bell(y))= \W{g}(x,y).\een
\end{defi}
\noindent\bls\, Hence the pair $(\bX, \Bell)$ is a.s.~the tour of a snake, since for $0\leq x \leq x'\leq 1$, 
\[g(x)=g(x')=\widecheck{g}(x,x') \imp\ \cov\bigl(\Bell(x) -\Bell({x'}), \Bell(x) -\Bell({x'})\bigr) = 0\imp\ \Bell(x)=\Bell({x'}).\]
However, the a.s.~continuity of $\Bell$ is not granted: it depends on the regularity of $\bX$ (all considerations on the H\"older exponents in the paper are developed for this reason),\vspace{+0.5ex}\\ 
\bls\, The (tour of the rooted) standard Brownian snake with lifetime process $g$ corresponds to the labeling of a continuum random tree with contour process $g$, by a Brownian motion starting at the root of the tree with the property that a node at height $h$ is labeled by a Brownian motion at time $h$, and for $\dot u$ and $\dot v$ two nodes of the tree $  T_g$, the Brownian trajectories $(\bB_s, 0\leq s \leq g_{u})$ and $(\bB_s', 0\leq s \leq g_{v})$ coincide  
on $[0,\W{g}(u,v)]$, and
\[\l(\bB_s-\bB_{\W{g}(u,v)}, \W{g}(u,v)\leq s \leq g_{u}\r) \textrm{ and }\l(\bB_s'-\bB_{\W{g}(u,v)}, \W{g}(u,v)\leq s \leq g_{v}\r)\textrm{ are independent}.\]
\bls\, The (tour of the rooted) {\it Brownian snake with lifetime process the normalized Brownian excursion} is the process which corresponds to this definition for $\bX={\se}$, that is, when the underlying continuous tree is Aldous' continuum random tree $T_{\se}$  (see e.g.~\cite{DuLG,MMsnake,janmarck05}).

\subsection{Iteration of snakes}
\label{sec:Iter-snakes}

To iterate the construction, we will associate a tree to the label process $\ell$ of a snake with contour process $(g,\ell)$. To this end assume for a moment that $\ell$ is continuous, and is an element of $C^{0}[0,1]$.

If $g$ is not the null function, the tree $T_g$ contains at least a non-trivial branch $b$, so that $\ell$ is a.s.~not positive on $[0,1]$ (since its range contains that of the Brownian motion on $b$), and then $\ell$ is not in $C^+[0,1]$, and therefore $\ell$ is not  the contour process of a tree. Pushed by combinatorial and technical considerations, we proceed as follows.
For a function $f$ in $C^0[0,1]$, define
\ben\bpar{lcl}
m(f) & := & \inf\{f(x)~,~x\in [0,1]\},\\
a(f) & := & \min \argmin f =\min \{ x~,~f(x)= m(f)\}.\epar
\een
The value $a(f)$ is the first hitting time of the minimum $m(f)$ for the function $f$ (left of Fig.~\ref{fig:CONJ2}).
\begin{figure}[!h]
\centering
 \includegraphics[scale=1]{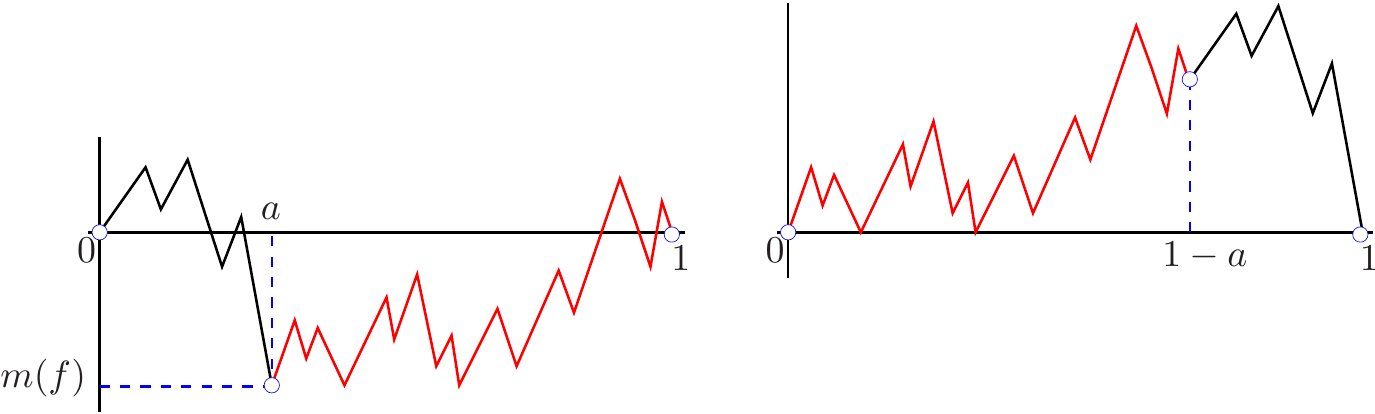}
\caption{\label{fig:CONJ2}Illustration of the map $f\mapsto \Conj(f)$: it exchanges the part of the graph before and after $a(f)$. }
\end{figure}

\paragraph{ Leading idea:} 
Consider  the so-called conjugation of paths, the map 
\ben
\label{eq:def-conj}\app{\Conj}{C^0[0,1]}{C^+[0,1]}{f}{x\mapsto f\l[(a(f)+x) \mod 1\r]-m(f)}.\een
This map, illustrated in Fig.~\ref{fig:CONJ2}, is notably known for sending  Brownian bridges onto Brownian excursions (if ${\sf b}$ is a  standard length-1 Brownian bridge, then
$\Conj({\sf b})\eqd {\se }$, 
Vervaat \cite{Vervaat1979}).

For $f\in C^0[0,1]$, since $\Conj(f)\in C^+[0,1]$, $\Conj(f)$ is naturally the contour process of a tree.
\begin{rem}\label{eq:fsddq}
Let $\mu$ be the corner measure of a tree $T_f$ and let $g=\Conj(f)$. The push-forward measure of $\mu$ by $\Conj$ is $\mu'$ defined by
$$\mu'(\cdot)= \mu( a(f)+\cdot \mod 1).$$
The measure $\mu'$ is the ``same corner measure'' as $\mu$ in the sense that it puts the same weight to the corners that are in correspondence on $T_g$ and $T_f$.
  \end{rem}
Hence, starting from any function  $f$ in $C^0[0,1]$ {\it including the label process of a snake} (taking that it belongs to $C^0[0,1]$),
one can consider the tree with contour $\Conj(f)$, and use it as the underlying tree of a branching Brownian motion.

Regarding the assumption that $\ell\in C^0[0,1]$, since $g(0)=g(1)=\W{g}(0,1)=0$, then a.s.~the label of the root is $\ell(0)=\ell(1)=0$. As for the continuity of $\ell$, the a.s.~existence (or not) of a continuous version for $\ell$, depends on the regularity of $g$. This is one of the (relative) difficulties of this construction.

\begin{defi}
  The space of rooted continuous $ \D$-snakes is defined to be $\RS^{ \D} :=\bigl(\RS\bigr)^{ \D}$,  equipped with the uniform topology.
  When we specify the corner measures, a rooted  $\D$-snake has the following form: $\l[\l(f^{(i)},\ell^{(i)},\mu^{(i)}\r),1\leq i \leq D\r]$.
  \end{defi}

\subsubsection*{Iterated rooted Brownian snakes}

\begin{defi} \label{defi:fzfe} For any positive integer $\D$,  we call $\D$th  Brownian snake the process
  \ben \label{eq:dethhe}
  \rbs[ \D]:=\l(\l[ \bh^{(1)},\Bell^{(1)}\r], \ldots,\l[ \bh^{( \D)},\Bell^{( \D)}\r] \r)
  \een
  taking its values in the space $\RS^{ \D}$, where:
  \bir 
  \itr the first tree is the continuum random tree: $\bh^{(1)}\eqd \se$,
  \itr for any $j$, conditionally on $\l[ \bh^{(1)},\Bell^{(1)},\ldots, \bh^{(j-1)},\Bell^{(j-1)},  \bh^{(j)}\r],$ the process $\Bell^{(j)}$ 
   is the label process of the rooted Brownian snake  with lifetime process $\bh^{(j)}$ (as defined in Def. \ref{def:TBS}),
  \itr  for $j\geq 2$, the contour of the $j$th random tree $\btj j:=T_{\bh^{(j)}}$ is
  \ben
  \bh^{(j)}= \Conj(\Bell^{(j-1)}).
  \een
  \eis ~\\
For any $j$, the corner measure on $\btj j$ is $\lambda$, the Lebesgue measure on $[0,1]$. 
\end{defi}

Hence, the standard Brownian snake with lifetime process the normalized Brownian excursion coincides with $\rbs[1]$ (see e.g.~\cite{legall2005,MMsnake,janmarck05}). Definition \ref{defi:fzfe} really defines an existing object only if the $\Bell^{(j)}$'s are all a.s.~continuous: this property is needed to define the continuous contour of the tree $\btj {j+1}$ using $\Bell^{(j)}$, and it needs to be proved.
\begin{theo}\label{theo:biendefini} For any $j\geq 1$,  $\Bell^{( j)}$ is a continuous process and then for any $\D\geq 1$, the process $\rbs[\D]$ is well defined.
\end{theo} 
We will indeed prove the continuity of the processes  $\Bell^{( j)}$ in Sec.~\ref{sec:proof-of-cont}.
\begin{rem}
\Bs\; There is no natural process $\rbs[0]$, since the Brownian bridge is not the label process of any tree. It is somehow the label process of the circle $\R/\mathbb{Z}$, which is not a tree.\\
\Bs\; The fact that $\Bell^{(j)}$ reaches its minimum only once a.s.~for $j\geq 2$ (on a leaf of $\bh^{(j)}$)
is unclear, but we conjecture that it is true; this property holds for $\Bell ^{(1)}$ (Le Gall \& Weill \cite[Prop. 2.5]{LGW}).\\
\Bs\; Constructing a version of $\rbs[ \D]$ conditioned by the non-negativity of all the $\Bell^{(i)}$, which is a singular conditioning, and which has been done by Le Gall \& Weill \cite{LGW} in the case of the Brownian snake $\rbs[1]$,  does not seem to be an easy task. This positive processes could be used to define more easily some feuilletages (as done below) without having to deal with what we call below ``tree synchronisations''.
\end{rem}

\subsection{The iterated random feuilletages $\RR \D$ }
\label{sec:Iter-Refold}

\subsubsection{ $\RR 2$ is the Brownian map}
\label{sec:fthz}

Let $\rbs[1]=(\bh^{(1)},\Bell^{(1)})$ be the standard Brownian snake; let again $\bh^{(2)}=\Conj(\Bell^{(1)})$.
Consider the tree $\btj 2 =T_{\bh^{(2)}}$, and set
$$
\Ba^{(1)}=\min \argmin \Bell^{(1)}. 
$$
By definition \eqref{eq:def-conj} of the map  $\Conj$, since $\bh^{(2)}(\cdot)=\Bell^{(1)}(\Ba^{(1)}+ \;\cdot \mod 1)- \min \Bell^{(1)}$, the corner $x-\Ba^{(1)} \mod 1$ of {$\btj 2$} corresponds to the corner $x$ in the tree $\btj 1$.
\begin{defi} Let $\rbs[2]=\l(\l[ \bh^{(1)},\Bell^{(1)}\r], \l[\bh^{(2)},\Bell^{(2)}\r] \r)$ be the 2nd Brownian snake.   The Brownian map $\RR 2$ is the topological space defined as $[0,1]/\sim_{2}$, where $\sim_{2}$ is the coarsest equivalence relation that extends the two following equivalence relations:
  \be
  x&\sim_{\bh^{(1)}}& y,\\
  \l(x-\Ba^{(1)} \r)\mod 1 &\sim_{\bh^{(2)}}& \l(y-\Ba^{(1)}\r) \mod 1.
  \ee
\end{defi}
These two relations are equivalent to
\be
D_{\bh^{(1)}}(x,y)&=&0,\\
D_{\bh^{(2)}}\l( x-\Ba^{(1)} \mod 1,y-\Ba^{(1)} \mod 1\r)&=&0,
\ee
where  $D_g$ was introduced in \eqref{eq:Dg} to define the distance $d_{  T_g}$ in the tree $  T_g$.

Let us discuss further the appearance of $\Ba^{(1)}$ in the considerations and the implications for iterations.

\subsubsection{Trees synchronization and $\RR \D$} 
\label{sec:TS}

One can argue that the worst hassle in the construction of the $D$th Brownian snake is the successive use of $\Conj$, which brings some extra random shifts at each iteration: these shifts are inherited by the $D$ random rooted feuilletages (in the $D=2$ case, these considerations appear in relation with the rooted pointed Brownian map). We will then take a moment to write the details of what we will call {\it trees synchronization}. Later on, we will see that the somehow non-continuity of this synchronization procedure will be at the origin of another complication, which will lead to the definition of pointed counterparts to the rooted snakes and to the random feuilletages.\par
Take a $\D$th Brownian snake $\rbs[\D]$ (with the same notation as in Def.~\ref{defi:fzfe}).
Set 
\ben
\Ba^{(m)} = \min \argmin \Bell^{(m)},~~\textrm{ for } m\in\cro{1, \D},
\een
so that again, $\bh^{(m+1)}(.)= \Conj(\Bell^{(m)})=\Bell^{(m)}\bigl((\Ba^{(m)}+. ) \mod 1\bigr) - \min \Bell^{(m)}$, and for any $m$,
\[\btj m =T_{\bh^{(m)}}.\]
In order to trace back all the shifts coming from the successive change of roots, we set
\ben
\BA^{(m)}=\Ba^{(1)}+\cdots+\Ba^{(m-1)},~~~\textrm{ for }m \geq 1.
\een
For example, the corner $x-
 \BA^{(3)} \mod 1$ of the tree  ${\btj 3 }$ corresponds to the corner $x-\BA^{(2)} \mod 1$ of ${\btj 2 }$, which in turn corresponds to the corner $x$ of  ${\btj 1 }$.

 \begin{defi}\label{defi:iBm} Let $\rbs[\D]=\l(\l[ \bh^{(1)},\Bell^{(1)}\r], \cdots,\l[ \bh^{(\D)},\Bell^{(\D)}\r]\r)$ be the $\D$th  Brownian snake.  We call $\D$th   random feuilletage $\RR{\D}$ the topological space
\ben
\RR{\D}:= [0,1]/\sim_{\D},
\een
where  $\sim_{\D}$ is the coarsest equivalence relation on $[0,1]$ refining all the following equivalence relations $\sim_{[m]}$ for $1\leq m \leq \D$, defined for  $x,y\in [0,1]$  by  
\ben\label{eq:qsfyur}
x \sim_{[m]} y 
\equi D_{\bh^{(m)}}\l(x-\BA^{(m)}\mod 1, y-\BA^{(m)}\mod 1\r)=0.
\een
\end{defi}
Hence, $x\sim_{\D} y$ if and only if there exists a finite sequence of identification points $((x_m,j_m),1\leq m\leq N) \in \l([0,1]\times \cro{1,\D}\r)^{N}$ such that, for $x_0:=x$, $x_{N+1}:=y$,
\ben \label{eq:jm}
x_m \sim_{[j_m]} x_{m+1} \textrm{~~ for } 0\leq m \leq N.
\een
\begin{rem} Formula \eref{eq:qsfyur} defines the feuilletage. In the discrete case we will transform \eref{eq:qsfyur} so that only corners corresponding to ``discrete nodes'' are identified. This will amounts to restricting  \eref{eq:qsfyur} to the support of the corner measure into play. This way of doing applies to continuous snakes too.
    \end{rem}

  \subsubsection{Some potential metrics on $\RR{\D}$ }
Here are two (potential) metrics on $\RR{\D}$ compatible with its topology:
\ben\label{def:d1}
d^{(1)}_{\RR{\D}}(x,y) = \inf_{r\geq 1}\;\inf_{1\leq m_1,\cdots,m_r\leq \D} \;\inf_{0\leq x_0,\cdots,x_{r} \leq 1} \sum_{j=0}^r D_{\bh^{(m_j)}}\l(x_{j}-\BA^{(m_j)}\mod 1,x_{j+1}-\BA^{(m_j)}\mod 1\r)
\een
where $x_0\sim_{\D} x$, $x_{m+1}\sim_{\D}y$, and
\ben\label{def:d2}
d^{(2)}_{\RR{\D}}(x,y) = \inf_{m}\;\inf_{0\leq x_0,\cdots,x_{2m+1} \leq 1} \sum_{i=0}^m D_{\bh^{(\D)}}\l(x_{2i}-\BA^{(\D)}\mod 1,x_{2i+1}-\BA^{(\D)}\mod 1\r)
\een
where $x_{2i+1}\sim_{\D} x_{2i+2}$, $x_0\sim_\D x$, $x_{2m+1}\sim_\D y$. Identifications can be viewed as distance-free jumps in the space $\RR{\D}$: they combine identifications coming possibly from several different trees $\btj {j_m}$. 
\medskip

\noindent \bls\, The metric $d^{(1)}_{\RR{\D}}$ is more symmetric: a traveller who wants to go from $x$ to $y$ has to walk on one of the trees $\btj {j}$ for $1\leq j \leq \D$; when it does so, the distance is given by the metric on this tree. If he is at a given moment at $a\in[0,1]$, he can jump at $b\in[0,1]$ if $a\sim_{\D} b$ without paying anything, 
or in other words, if $a$ and $b$ are two corners of the same node in one of the trees $\btj {j}$ for $1\leq j \leq \D$.   He can change tree whenever he wants to go on his travel, and the final distance for a path is the minimum on all possible trips of the sums of all the distances made on each of these trees.\\ 
\bls\, For the metric $d^{(2)}_{\RR{\D}}$, the traveller can only walk on the tree  $\btj {\D}$, but whenever he wants, if he is at $a\in[0,1]$, he can jump at $b\in[0,1]$ without paying anything if $a\sim_{\D}b$.

\begin{OQ} Are the distances $d^{(1)}_{\RR{\D}}$ and $d^{(2)}_{\RR{\D}}$  non-trivial for any $\D>2$? (that is, is the diameter of $\RR{\D}$ under $d^{(j)}_{\RR{\D}}$ a.s.~positive)?
\end{OQ}

\bis
\its The advantage of $d^{(1)}_{\RR{\D}}$ is that it is non-increasing in $ \D$, since when one passes from $ \D$ to $ \D+1$, the set on which the minimum is taken is larger for the inclusion order. 
 ``Geometrically'', new identifications are provided by the $( \D+1)$th tree. Moreover, with each subset of indices $\{i_1,\cdots,i_m\}$ included in $\cro{1, \D}$ one can associate the space; 
\[\RR{i_1,\cdots,i_m}:=((\ldots([0,1]/\sim_{[i_1]})/\ldots )/\sim_{[i_m]})\] with the analogue of distance $d^{(1)}$ given by taking the infimum in \eref{def:d1} only on these indices.
\its For $ \D=1$, $\RR{1}$ coincides topologically with Aldous' continuum random tree, and the metrics $d^{(1)}_{\RR{1}}$ and $d^{(2)}_{\RR{1}}$ are equal and coincide with the standard metric on this space.
\its For $ \D=2$, $\RR{2}$ coincides topologically with the Brownian map and  $d^{(2)}_{\RR{2}}$ corresponds to the standard metric on this space.
\eis

But we must say that we do not know the answers to the following questions:
\begin{OQ} For $ \D>2$, is it true that $d^{(j)}_{\RR{\D}}(x,y) =0 \imp x\sim_{ \D} y$ for the distance $j=1$ or $2$? It is true for $ \D=2$ as a consequence\footnote{\label{foo:quo} When one quotients a topological space as we did when we introduced $\sim_{ \D}$, it may happen that the ``distance inherited from the initial distance'' on this space is not a distance: for example, consider $E= [0,1]$ (or $[0,1]\cup[2,3]$), equipped with the usual distance $|.|$,  and quotiented by the equivalence relation $x\sim y$ iff $x=y$ or $x,y \in \mathbb{Q}$ (in other words,  identify all rational numbers). Clearly, the quotient space $E^{\star}$ is not reduced to a single point, but $d^{\star}(x,y)=0$ for all $x,y\in E^{\star}$ under ``the inherited distance $d^\star$'', since for any $x$, $d(x,\mathbb{Q})=0$ for all $x\in \R$, if $d(x,y)=|x-y|$ is the usual distance on $\R$. Hence, $d^\star$ is not a distance, since $d^\star(x,y)=0 \not\imp x=y$. If one further quotients $E^{\star}$ by $x\sim^{\star} y$ when $d^\star(x,y)=0$, then the space $E^{\star}$ becomes trivial, reduced to a single point.} of  Miermont \cite{Miermont2013} and Le Gall \cite{legall2013}.
\end{OQ}
As a consequence of Theorem \ref{theo:ini-rec}, which allows seeing that $\bh^{( \D)}$ is a.s.~H\"olderian with exponent $1/2^ \D-`e$, for any $`e>0$, it may be shown that the tree $ (\btj {\D},D_{\bh^{( \D)}})$ has Hausdorff dimension smaller than $2^{ \D}$.
  For $ \D\in\{1,2\}$, these upper bounds fit with the right values \cite{GallHaus,TDLFLG}.
  These bounds are also the Hausdorff dimensions of  $\l(\RR{\D} ,d^{(2)}_{\RR{\D}}\r)$ for $ \D \in \{1,2\}$.
\begin{OQ} What are the Hausdorff dimensions of the random trees  $(\btj {\D},D_{\bh^{( \D)}})$ and of the random feuilletages $(\RR{\D} ,d^{(2)}_{\RR{\D}})$ for $ \D>2$? If $2^\D$ is indeed the Hausdorff dimension of $(\btj {\D},D_{\bh^{( \D)}})$, then this value provides a lower bound for the dimension of $\l(\RR{\D} ,d^{(2)}_{\RR{\D}}\r)$... We conjecture that both spaces indeed have Hausdorff dimension  $2^{ \D}$.
    \end{OQ}

\begin{rem}  About the redundancy of the iterated Brownian snake: from $f$ to $\Conj(f)$, a change of origin has been done. If one has only $\Conj(f)$ in hand, the ``shift'' $a(f)$ cannot be recovered. It turns out that for our applications to \Rc, the shift is needed to ``synchronize'' the identifications provided by the different trees. Working directly and only with $f$ -- which is possible since $f$ determines $\Conj(f)$ --  is a bit annoying because it demands reintroducing $\Conj(f)$ everywhere, since the iteration we propose relies on the tree encoded by $\Conj(f)$.
  \end{rem}
  \begin{rem}When dealing with asymptotic discrete snakes, we will observe that the change of origin is not continuous, that is $\|f_n-f\|_{\infty}\to 0 \not\imp a(f_n)\to a(f)$ \footnote{To circumvent this problem, it would suffice to prove that a.s., $\#\argmin(\Bell^{(j)})=1$, for the iterated process $\Bell^{(j)}$ defined previously.}.  As a consequence, even if the sequence $(f_n)$ converges in $C[0,1]$,  the sequence of trees $(\Conj(f_n))$ may not converge in  the set of rooted trees, $K$. This issue explains the complications that will appear progressively in the sequel. The strategy we have adopted to treat them is to use the redundancy provided by the presence of $h_{j+1}=\Conj(f_j)$ together with $f_j$ in the iterated Brownian snake. The discontinuity of the map $a(\cdot)$ will result in the loss of the identity of the root corner in the iterated trees while the root vertex will still be well known: pointing a tree amounts to considering as equivalent two trees rooted at different corners of the same root vertex. Pointing is compatible with the snake construction in which the root vertex is labeled 0, whatever the considered root corner. We will therefore progressively turn our intention to pointed snakes, pointed feuilletages, and finally, state our main theorems for theses objects.
    \end{rem}

\section{Iterated snakes and feuilletages: combinatorial objects}
\label{sec:ISIMCO}\setcounter{equation}{0}
\noindent{\bf Notation : } The $i$th increment of any sequence $(x_i,i\geq 0)$  is denoted $\Delta x_i=x_{i}-x_{i-1}$.\\
{\bf Convention : } We make a great use of continuous processes $X$ obtained by linear interpolation of some random sequence of the form $(X_k,\, k\in\cro{0,n})$ or  of the form $(X_{k/n},  k\in\cro{0,n})$. We will keep the same notation $X$ for the continuous and discrete version, but we will name ``process'' the interpolated version, and ``sequence'' the discrete one (without additional warning).  
\medskip

The main aim of this section is to present the discrete iterated snakes and discrete  iterated feuilletages.

\subsection{Planar trees and their encodings}
\label{sec:qsdqsd}

\paragraph{Rooted planar trees.} For $\N^\star=\{1,2,\cdots,...\}$, consider $U=\{\varnothing\}\cup \bigcup_{k\geq 1} \N^\star{}^k$ the set of words on the alphabet $\N^{\star}$. For any word $w=w_1...w_k$ in $U$ where $w_j\in \N^{\star}$, $|w|=k$ is the length of $w$, also called the depth of $w$. For $u$ and $v$ in $U$, $uv$ stands for the concatenation of $u$ and $v$. 
\begin{defi} A rooted planar tree $T$ is a finite subset of $U$, containing $\varnothing$, stable by prefix (if $uv \in T$ for $u,v\in U$, then $u\in T$), and such that if $ui\in U$ for $u\in U$ and $i\in \N^{\star}$, then $uj\in T$ for $1\leq j \leq i$.
\end{defi}
An example is shown on the left of Fig.~\ref{fig:Arbre2}. The elements of $T$ are called nodes or vertices. For $u\in T$ and $j\in \N^{\star}$, if $uj\in T$,  then $uj$ is called a child of $u$, and $u$ is the parent of $uj$. The number of children of $u$ is $c_u(T)=\#\{j \in  \N^{\star}, uj \in T\}$. The prefixes of $u$ are called the  ancestors of $u$.
The size of a tree $T$, denoted by $|T|$ is its cardinality (its number of nodes). We also set
\[\|T\|=|T|-1,\]
the number of \underbar{edges} of $T$.
Denote by $\bbT$ the set of trees, and by $\bbT_n$ the subset of those  with $n$ {\bf edges}:
\[\bbT_n=\{ T \in \bbT~,~ \|T\|=n\}.\]
It may be proved by induction that for $n \geq 1$, the cardinality of $\bbT_{n}$ is the $n$th Catalan number:
\[ \#\bbT_{n}=C_{n}=\binom{2n}n/(n+1).\]
\begin{figure}[!h]
\centering
 \includegraphics[scale=0.85]{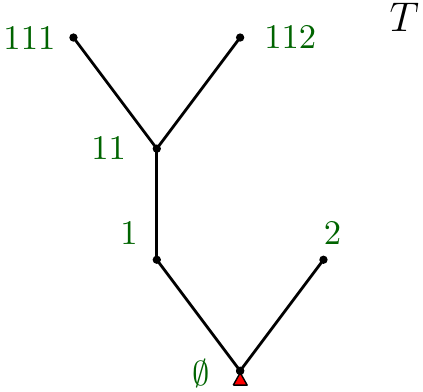}
 \hspace{0.8cm}
 \includegraphics[scale=0.8]{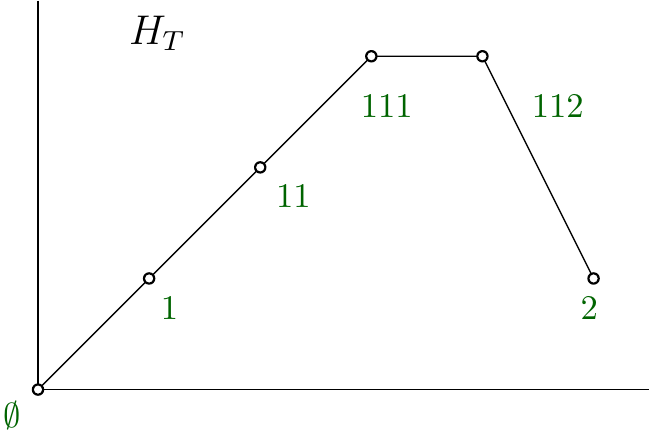}
 \hspace{0.8cm}
  \includegraphics[scale=0.8]{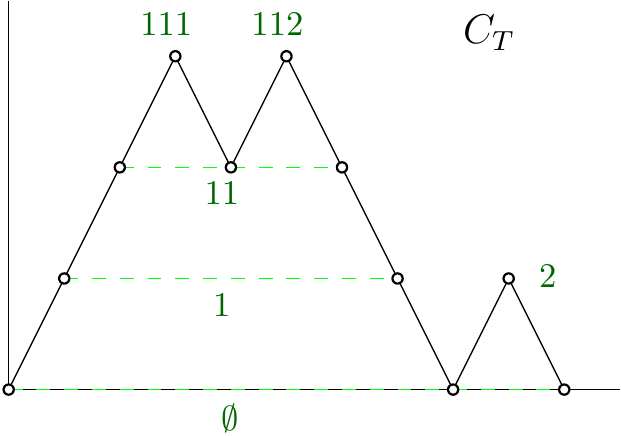}
\caption{\label{fig:Arbre2}From left to right: a rooted planar tree $T$, and the corresponding height process and contour process. 
}
\end{figure}

\vspace{-0.8cm}
\paragraph{Height sequence.} The lexicographical order on $U$ induces an ordering on any tree, and allows using some bijections to represent trees as sequences.

\begin{defi}
The height sequence of $T$ is the sequence $H_T\cro{0, \|T\|}$ of the successive heights of the nodes of $T$ sorted in lex.~order $u_0=\varnothing,u_1,\ldots,u_{\|T\|}$:
\ben
\label{HeightProcess}
H_T(k)=|u_k|, ~~\textrm{ for } 0\leq k \leq \|T\|.
\een
\end{defi}
Here is a classical result (see e.g.~\cite{DuLG, mm01}):
\begin{lem}
\label{lem:TtoH}
For any $n$, the map which to a tree associates its height sequence,
 $$
  \app{\Phi_n^{T\to H}}{\bbT_n}{\bbH_n}{T}{H_T},
  $$ is a bijection,
  where
 \ben
\bbH_n = \{H(\cro{0,n}), H_0=0,
 \Delta H_{i} \leq 1 \textrm{ and } H_{i} >0 \textrm{ for }i\geq 1 \}.
 \een
\end{lem}
An example of a tree and corresponding height sequence is shown in Fig.~\ref{fig:Arbre2}. 
\begin{figure}[!h]
\centering
 \includegraphics[scale=1]{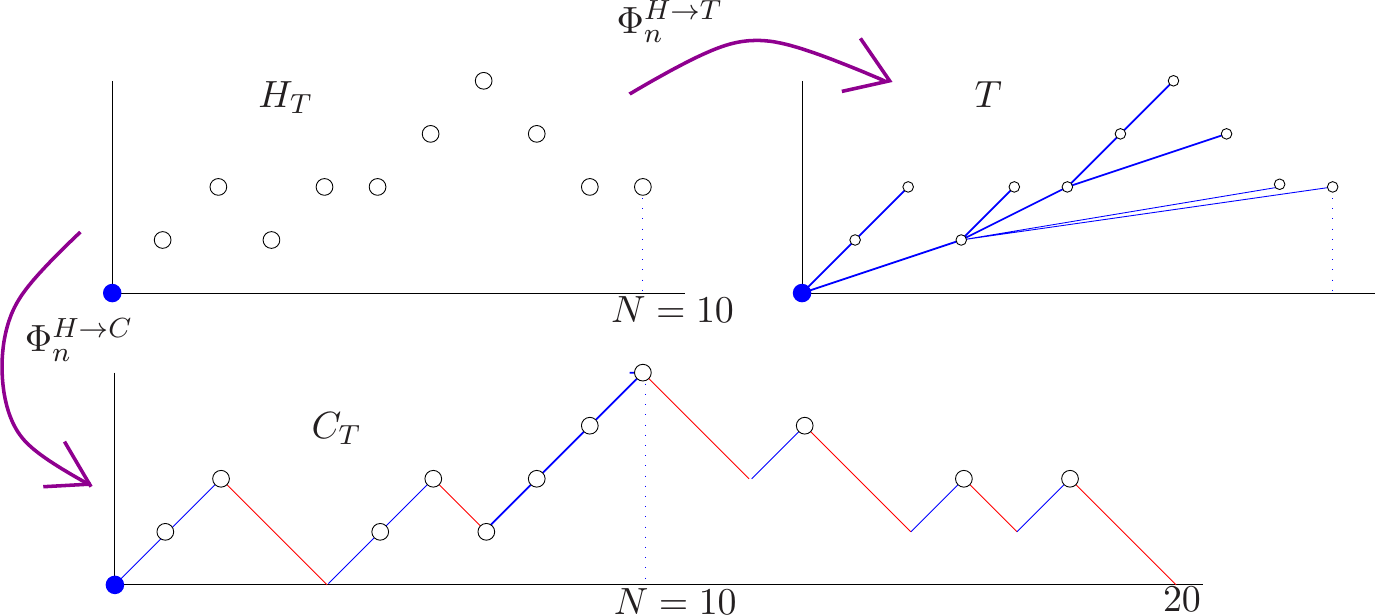}
\caption{\label{fig:HtoT}Illustration of the map $\Phi_n^{H\to T}$, and of the map  $\Phi_n^{H\to C}$;  the sequence of successive final heights of increasing steps in the contour process is exactly the height process. }
\end{figure}
The reverse bijection  $\Phi_n^{H\to T}:\bbH_n\to \bbT_n$  allows {\it constructing a tree from its height sequence}: 
\ben
\Phi_n^{H\to T}= \l(\Phi_n^{T \to H}\r)^{-1}.
\een
It will play an important role (see Fig.~\ref{fig:HtoT}): take any sequence $H(\cro{0,n})$ in $\bbH_n$, and draw the points $z_k=(k,H_k)$ for $k \in \cro{0, n}$ in the plane, as usual using a classical coordinate system: for any $1\leq k\leq n$, draw the segments
$[z_k,z_{\rho(k)}]$ with $\rho(k)=\max\{j <k, H_j=H_k-1\}$.  An example is given in Fig.~\ref{fig:HtoT} for the map $\Phi_n^{H\to T}$.

\paragraph{Contour sequence.} The {\it depth first traversal} of $T$ is a function
 $$ c_T:\cro{\,0,2\, \|T\|\,}\to T,$$ defined as follows: first $c_T(0)=\varnothing$. Assume that the image of $\cro{0,j}$ has been defined for some $0\leq j< 2\|T\|$, two cases arise:\\
-- if  $c_T(j)$ has some non-visited children, that is some children not in the list $c_T(\cro{0,j})$, then $c_T(j+1)$ is the smallest of these children for the lex.~order,\\
-- if all the children of $c_T(j)$ have been visited, then $c_T(j+1)$  is the parent of $c_T(j)$.
\begin{defi}
\label{def:contour-discr}
  The contour sequence  $C_T(\cro{\,0, 2\, \|T\|\,})$ of $T$ is the sequence defined by
\ben\label{eq:Cc}
C_T(k)=|c_T(k)|, \textrm{ for } k \in \cro{0,2\|T\|},
\een
that is the successive heights of the nodes of $T$ when turning around clockwise (see Fig.~\ref{fig:Arbre2} $(iii)$).
\end{defi}

For any $k\in\cro{0,2\|T\|-1}$, the pair $(c_T(k),c_T(k+1))$ is an edge of $T$. 
If $T$ is drawn in the plane, it is suitable to consider that $c_T$ is a walk around the tree, and that for any $0\leq k \leq 2\|T\|$,
\ben\label{eq:corner}
\Bigl(c_T\bigl( k-1 \mod 2 \|T\|\bigr),~c_T(k),~c_T(k+1 \mod 2 \|T\|)\Bigr)
\een
is the $k$th \textit{corner} of the tree. We will call this corner $c_T(k)$ for simplicity.\par
The following result is a classical result in combinatorics (see e.g.~\cite{mm01}).
\begin{lem}\label{lem:TtoC} For any $n \geq 1$, the map which to a tree associates its contour sequence,
\ben
\app{\Phi_n^{T\to C}}{\bbT_n}{\Dyck_{2n}}{T}{C_T},
\een
is a bijection, where $\Dyck_{2n}$ is the set of Dyck paths with $2n$ steps:
\ben\label{eq:Dyck}
\Dyck_{2n}=\{S(\cro{0,2n}),~S_0=S_{2n}=0 ,~\Delta S_i \in \{-1,1\} ~\textrm{and }~ S_i\geq 0,~  \forall i \in\cro{1,2n} \}.
\een
\end{lem}

The {\it distance between two nodes} $c_T(k)$ and $c_T(k')$ in the tree can be expressed in terms of $C_T$:
\ben
d_T\bigl(c_T(k),c_T(k')\bigr)=C_T(k)+C_T(k')-2 \widecheck{C}_T( k,k').
\een

As represented in Fig. \ref{fig:HtoT}, there is also a direct way to associate the contour process of a tree $T$ to its height process.
  \begin{lem}
  The map $\app{\Phi_n^{H\to C}}{\bbH_n}{\Dyck_{2n}}{H\cro{0,n}}{C\cro{0,2n}}$,  which sends a height process $H\cro{0,n}$ (of a tree $T$) to the corresponding contour process (the one of the tree $T$) is a bijection.
\end{lem}
\begin{proof}
The bijection is simple: both processes start at 0, and the contour process ends at 0. The successive values in $H\cro{1,n}$ correspond to the first visit times of the nodes according to the lexicographical  order. These heights are then the successive heights of the contour process $C\cro{0,2n}$ at the times $t_i$ such that $C_{t_i}=C_{t_i-1}+1$, since the contour process increases every time a new node is visited. \end{proof}
We will come back to this bijection in the proof of Theorem \ref{theo:ini-rec3}.

\subsection{Discrete snakes and iteration of discrete snakes}

\paragraph{Discrete snakes. } Labelings and snakes are defined as in the continuous case.
\label{eq:ds}
\begin{defi}\label{def:LT} Let $T$ be a planar tree, and $c_T=(c_T(k), 0\leq k \leq 2\|T\|)$ be its contour sequence.  A labeling of $T$ is a sequence $\ell_T=(\ell_T(k), 0\leq k \leq 2\|T\|)$ such that
  \[ c_T(k)=c_T(k')\imp \ell_T(k)=\ell_T(k'), \textrm{ for all } k,k'\in\cro{0,2\|T\|}.\]
The tour of the corresponding discrete snake is defined as $(C_T,L_T)$,  where
  \be
  C_T(k) &=& |c_T(k)|,\quad 0 \leq k \leq 2\|T\|,\\
  L_T(k) &=& \ell(c_T(k)),\,\,   0 \leq k \leq 2\|T\|,
  \ee
meaning that $L_T(k)=\ell(c_T(k))$ is the label of the $k$th node visited by the depth first traversal.

  \end{defi}
Let $\mu$ be a probability distribution on $\R$, and let $T$ be a planar tree (deterministic for the moment). The {\bf standard branching random walk} with underlying tree $T$ and $\mu$-distributed spatial increments is defined as follows: attribute to each node $u$ of $T$ different from the root $\varnothing$ a random variable $\bY_u$ where $(\bY_u, u\in T\setminus \{\varnothing\})$ is a family of independent random variables with common distribution $\mu$, and set $\bY_{\varnothing}=0$.
Now, consider the ``spatial'' labeling $\Bell=(\Bell(u),u\in T)$ of $T$ defined by
\ben
\Bell(u)=\sum_{v \preceq u} \bY_v,~~\textrm{ for any } u \in T,
\een
where the sum is taken on the set of ancestors $v$ of $u$.
Hence, $\Bell(\varnothing)=0$, and along each branch of $T$ the labels evolve as a random walk with increment distribution $\mu$.  
This definition extends to random trees, by sampling first the underlying tree $T$ at random, and by constructing the branching random walk using spatial increments independent of the tree $T$. \par

In the sequel, we will consider only branching random walks with increment distribution 
\ben\label{eq:nu}
\nu:=\frac{1}{3}\l(\delta_1+\delta_0+\delta_{-1}\r),
\een
so that the child of a vertex with label $l$ has label $l-1, l$, or $l+1$ with equal probability. For each branching random walk with underlying tree $T$ (random or not) and spatial increments $\nu$-distributed,  consider the spatial labeling $\Bell=(\Bell(u),u\in T)$.

When $T$ has $n$ edges, the label process $\bL_T$ of a labeled tree $(T,\Bell)$ is an element of $\bbL_{2n}$ where, for any $N\geq 0$,
\ben\label{eq:clN}
\bbL_{N}=\{L(\cro{0,N}), L_0=L_N=0 \textrm{ and } \Delta L_{j} \in \{-1,0,1\} \textrm{ for any }j 
\}.
\een
The (tour) of the random discrete snake with $n$ edges, is $(C_{\bT},\bL_\bT)$ for $\bT$ taken uniformly in $\Tset_n$.

\paragraph{Discrete conjugation map.} We here define a map  $\Phi_N^{L\to H}$ similar to the conjugation map \eqref{eq:def-conj} in the discrete setting: it sends $\bbL_N$ onto $\bbH_N$.
\begin{defi}\label{def:LtoH}
For each $N\geq 0$, the discrete conjugation map $\Phi_N^{L\to H}$ is the map
\ben
\app{\Phi_N^{L\to H}}{\bbL_N}{\bbH_{N}}{L(\cro{0,N})}{H(\cro{0,N}):=\Phi_N^{L\to H}(L(\cro{0,N}))}
\een
where $H$ is defined by 
   \ben\label{eq:qfqeds}
 \bpar{lll}  H(0)&=& 0\\
   H(j)&=&1+L\bigl(A+j-1 \mod N\bigr)-L(A), \textrm{ for } j \in \cro{1,N},
   \epar\een
   where 
$A = \min \argmin L(\cro{0,N})$ (since $L(0)=L(N)$, necessarily $0\leq A<N$). An example is shown in Fig.~\ref{fig:Conj}. 
\end{defi}
\begin{figure}[!h]
\centering
 \includegraphics[scale=1]{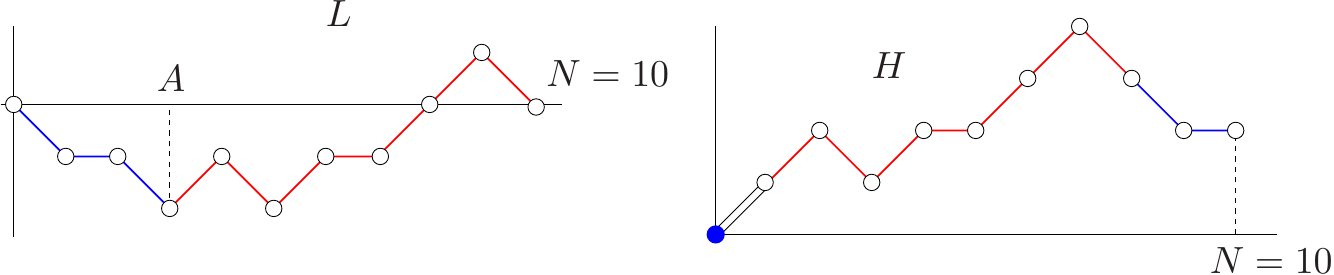}
\caption{\label{fig:Conj}The map $\Phi_{10}^{L\to H}$.
}
\end{figure}
\noindent \bls\, In the discrete case, height processes and contour processes are different objects, and as a matter of fact, they cannot be obtained by a simple conjugation of Lukaciewish walks (which are ``bridge type trajectories'' of random walks with increments in $\N\cup\{-1\}$ conditioned to end at $-1$ at the end),    since the height process ends at a positive position, and contour processes of discrete trees have only steps $\pm 1$ (see e.g.~\cite{mm01}, in which Lukaciewish walks are called depth first queue processes). \\
\noindent \bls\, When $A+j-1$ passes from $N-1$ to $N$,  $A+j-1 \mod N$ passes from $N-1$ to 0. Since $L(N)=L(0)=0$, this does not provoke any bad border effect on the increments $(\Delta H_j,1\leq j \leq N)$, which all belong to $\{+1,-1,0\}$.\\
\bls\, Height processes are not exactly elements of the set $C^0[0,N]$ (of continuous functions on $[0,N]$  starting and ending at 0), so that the nature of  $\Phi_N^{L\to H}$ is a bit different from that of $\Conj$. The composition  $\Phi_N^{H\to C}\circ \Phi_N^{L\to H}$ will be used and is closer in nature to the conjugation map (see e.g.~in Prop.\ref{pro:uqn}).

\paragraph{Iterated discrete snakes.} We define {\it iterated discrete snakes} as we did in the continuous case.

\begin{defi}\label{def:rds} For any positive integer $ \D$,  we call $ \D$th random discrete snake of size $n$, the process
  \ben \label{eq:dsiche}
  \RBS_n[ \D]:=\l(\l[ \bC^{(1)}_n,\bL^{(1)}_n\r], \ldots,\l[ \bC^{( \D)}_n,\bL^{( \D)}_n\r] \r),
  \een
  such that
  \bis
  \its $\bC^{(1)}_n$ is uniform in $\Dyck_{2n}$ (the contour process of a uniform planar tree with $n$ edges),
  \its if $\bC_n^{(j)}$ is defined for some $j\geq 1$, $\bL^{(j)}_n$ is the label process of a branching random walk with increment distribution $\nu$ as given in \eref{eq:nu} with underlying tree $\btnj n j$, the tree with contour process $\bC^{(j)}_n$,
  \its if for some $j\geq 2$, $\bL^{(j-1)}_n\in\bbL_{2^{j-1}n}$ is known, then the 
 $j$th tree $\btnj n j$  is defined by its height process, obtained by conjugation of $ \bL^{(j-1)}_n$,
  \ben
  \bH^{(j)}_n:= \Phi_{2^{j-1}n}^{L\to H} (\bL^{(j-1)}_n) \in\bbH_{2^{j-1}n},
  \een
  and its contour process is $\bC^{(j)}_n= \Phi_{2^{j-1}n}^{H\to C} \l(\bH^{(j)}_n\r)\in  {\sf Dyck}_{2^{j}n} $.
\eis
\end{defi}

Some simulations for the processes of $\RBS_{5000}[4]$ are shown in Fig.~\ref{fig:Bij-Sc15}. \\

\begin{figure}[h!]
\centering \includegraphics[scale=0.34]{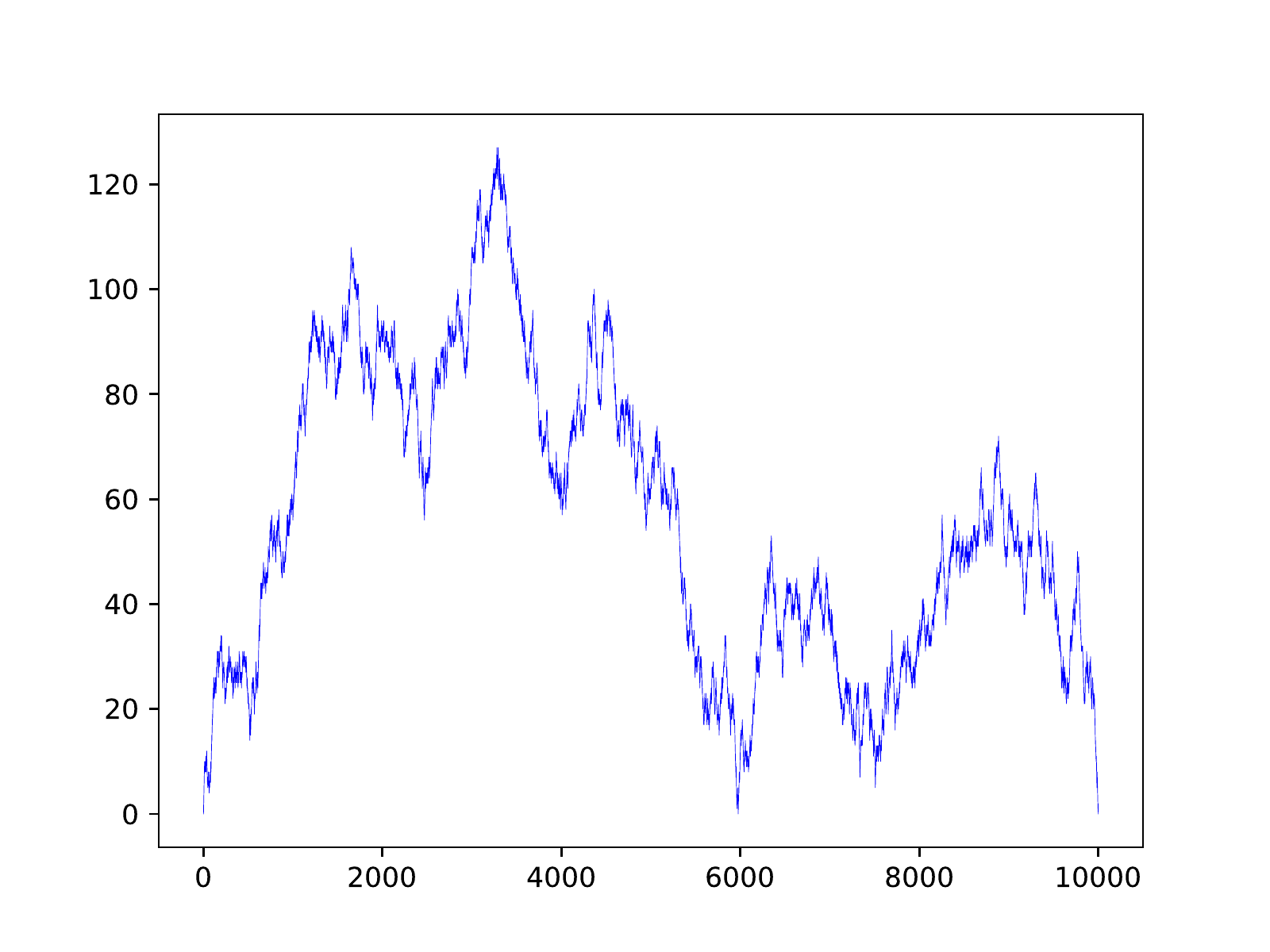}  \includegraphics[scale=0.34]{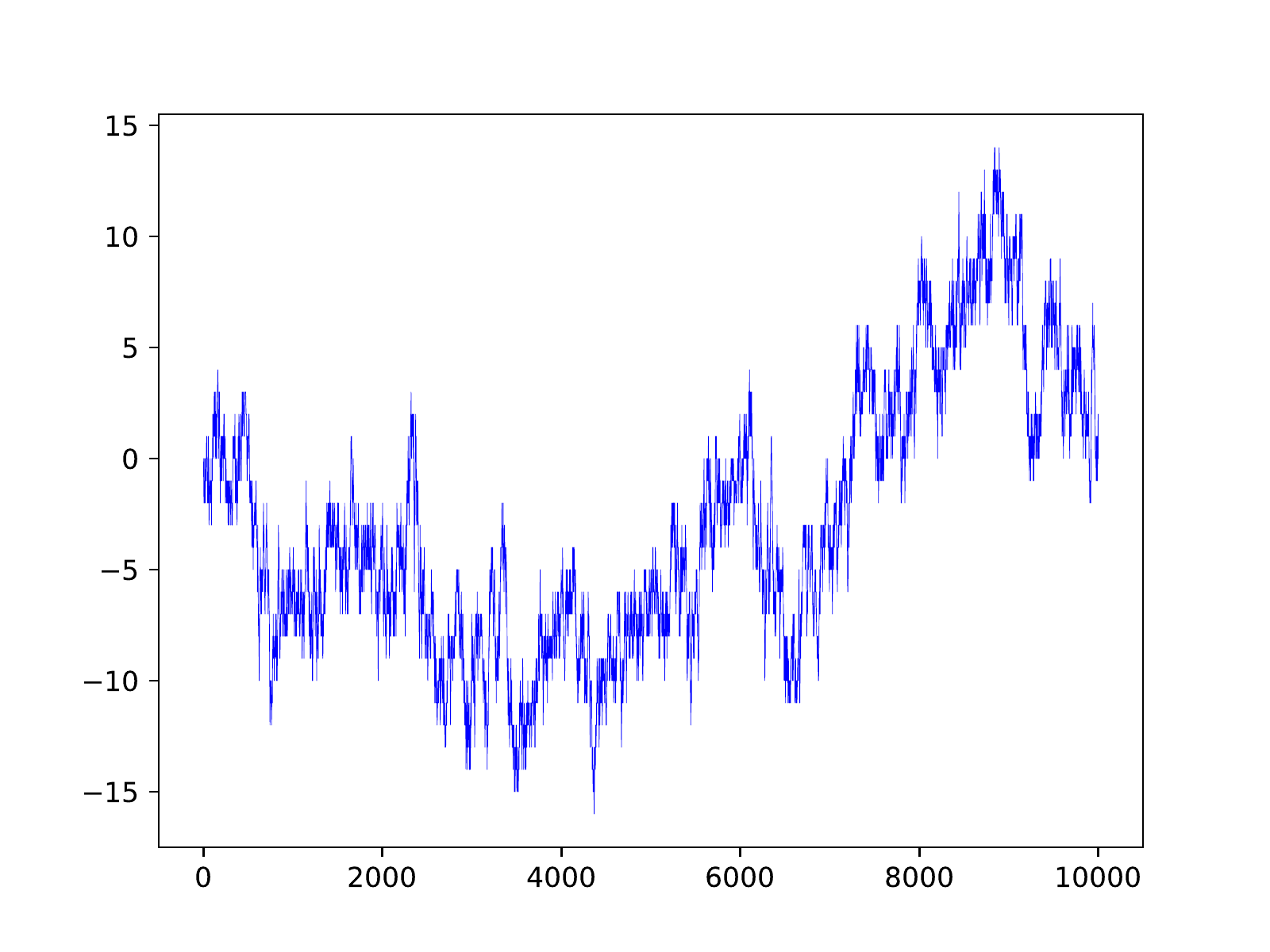} ~\\
\centering \includegraphics[scale=0.34]{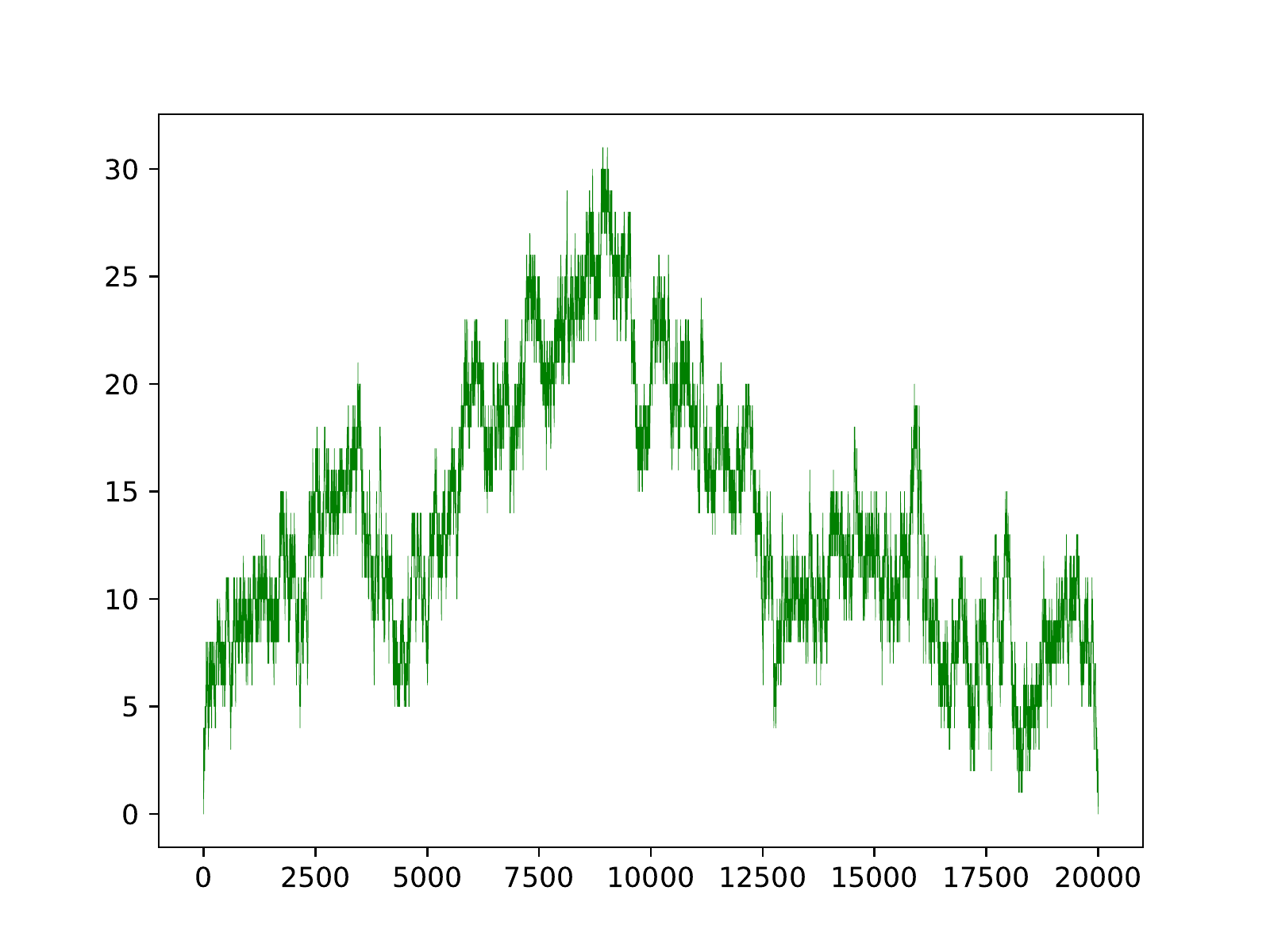}  \includegraphics[scale=0.34]{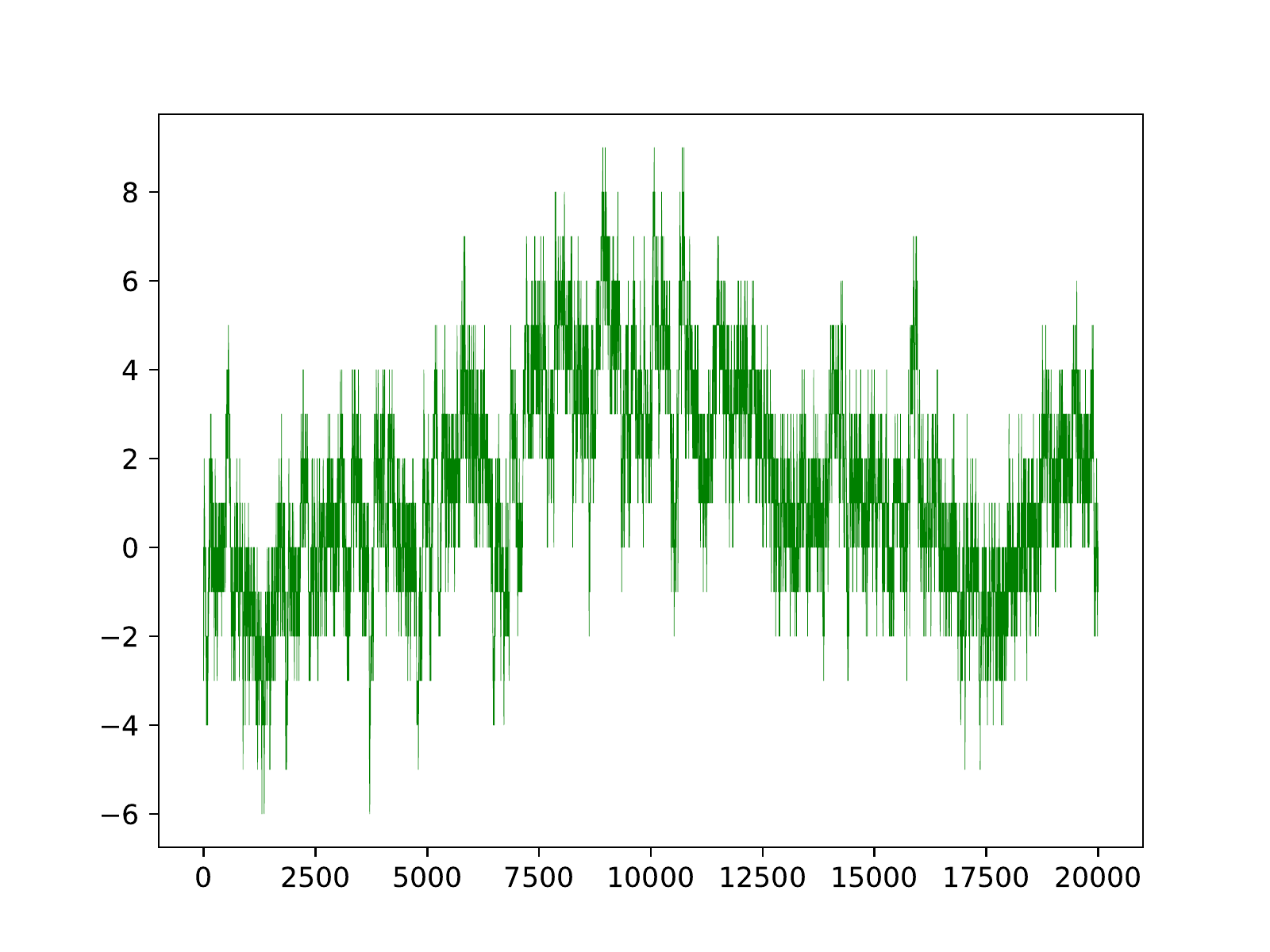} ~\\
\centering \includegraphics[scale=0.34]{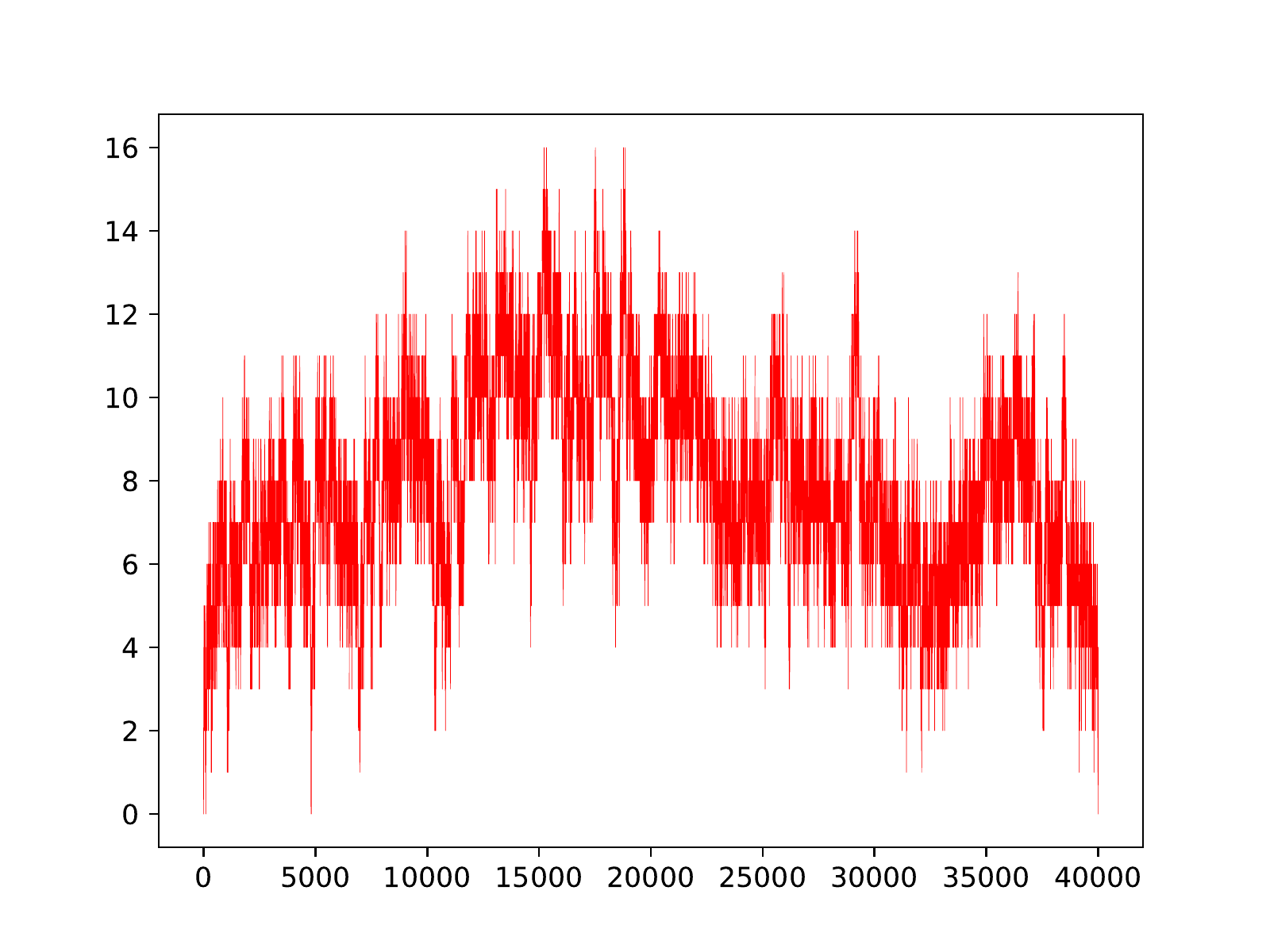}  \includegraphics[scale=0.34]{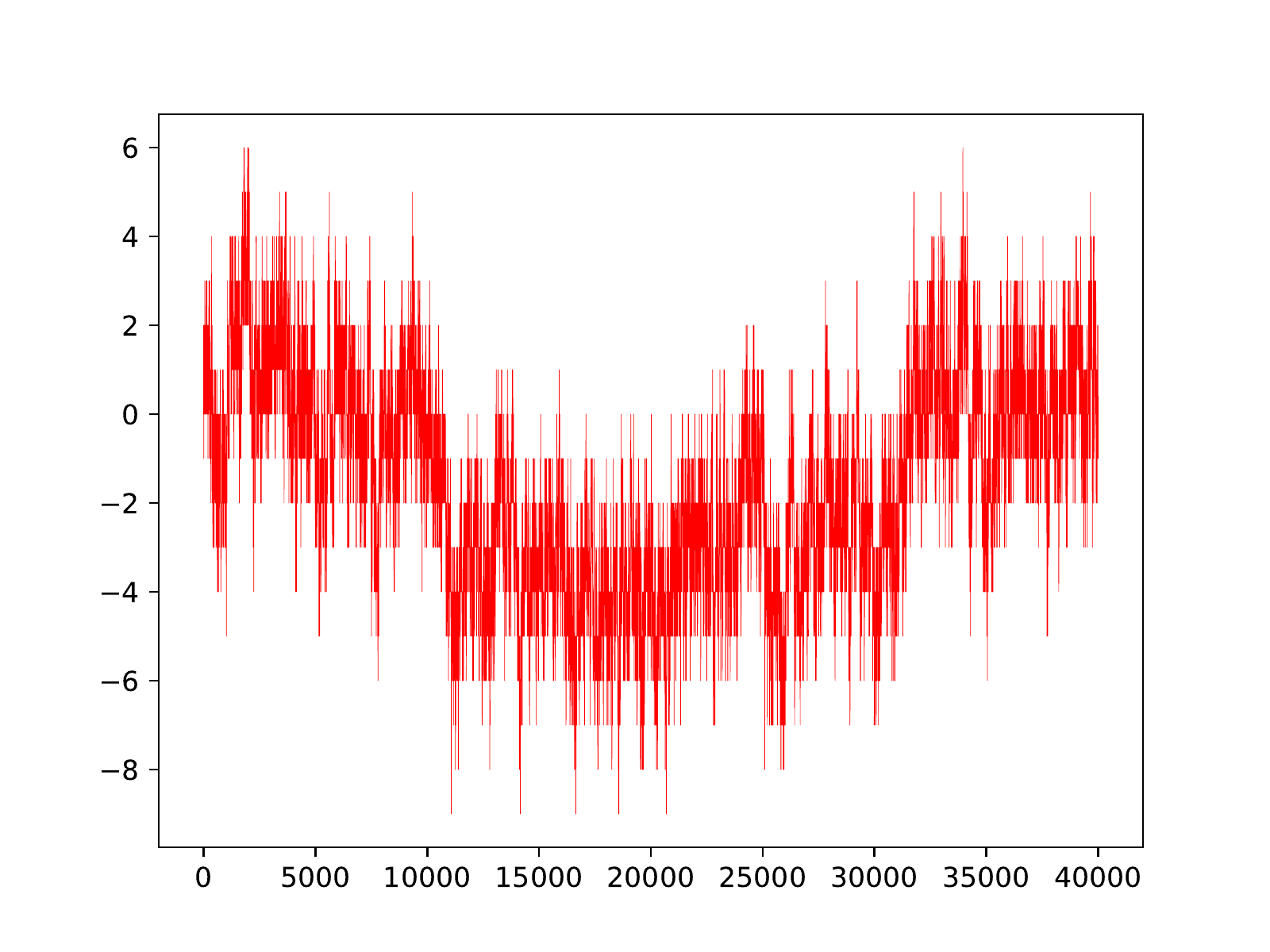} ~\\
\centering \includegraphics[scale=0.34]{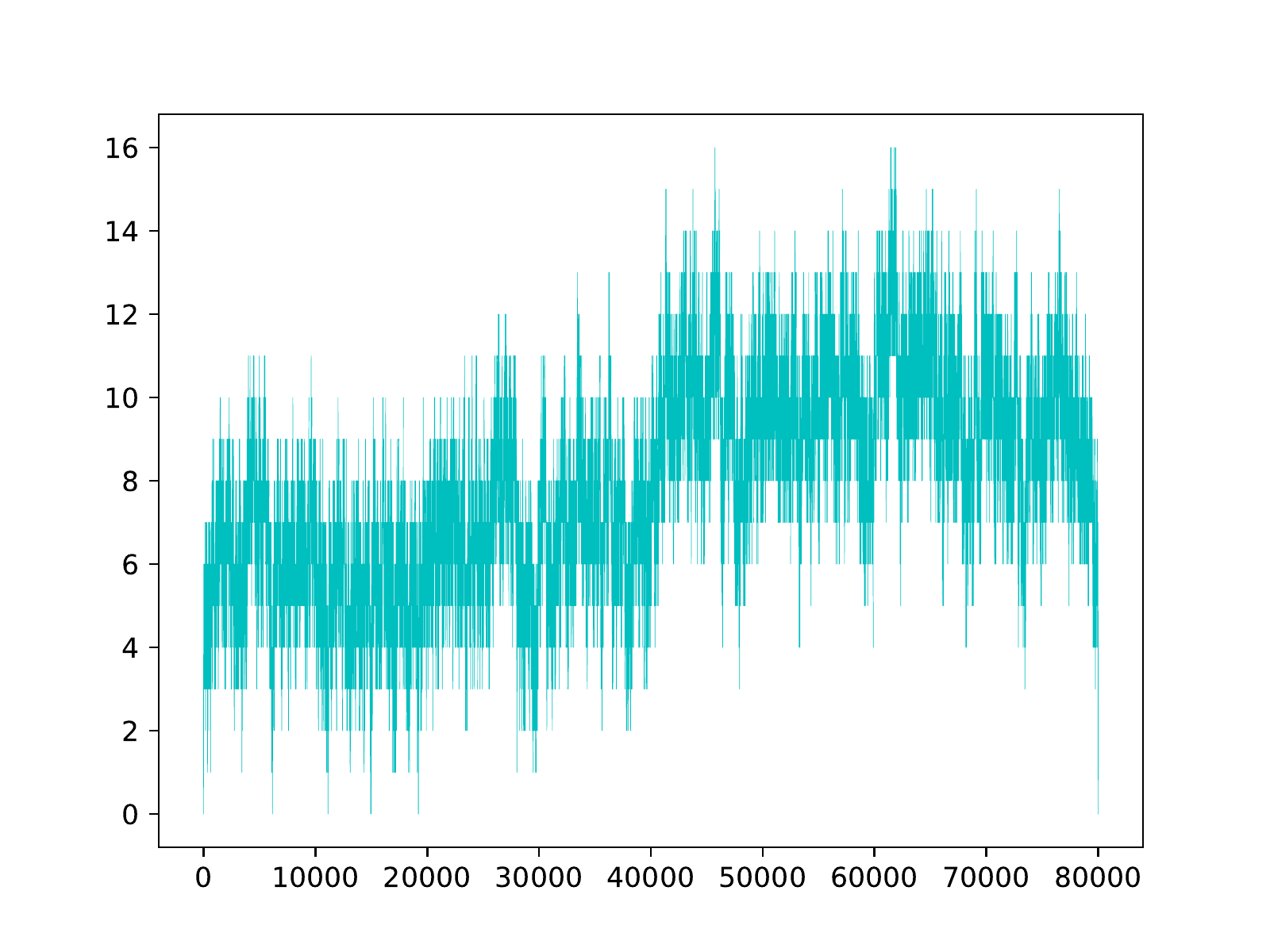}  \includegraphics[scale=0.34]{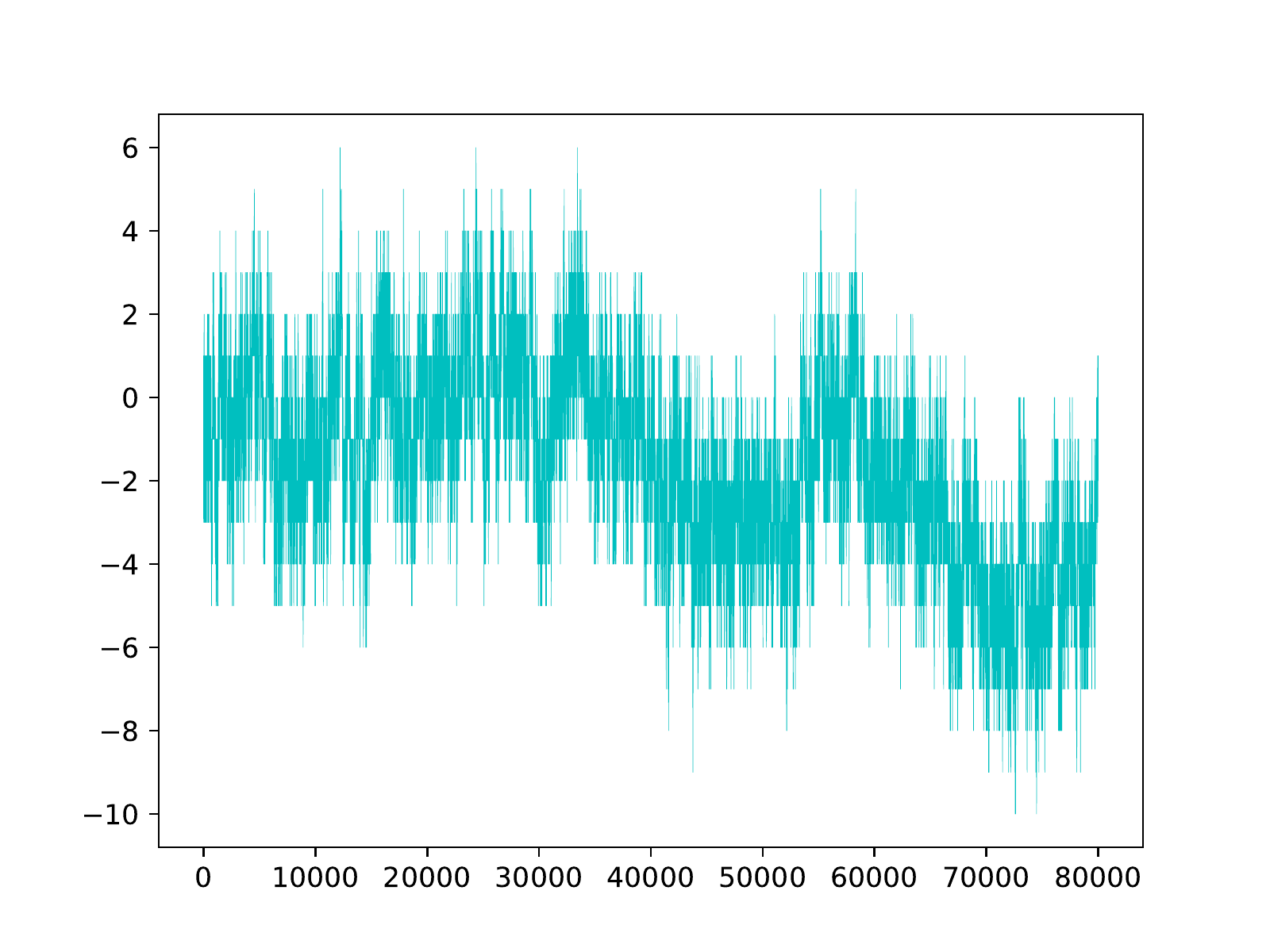} 
\caption{\label{fig:Bij-Sc15}Simulation of $\protect\RBS_{5000}[4]$. On the $k$th line are the processes $\bC^{(k)}_{5000}$ and $\bL^{(k)}_{5000}$.  The range decreases when $k$ increases. The ``irregularities'' of the processes increase with $k$.}   
\end{figure}

\begin{figure}[h!]
  \centerline{ \includegraphics[scale=0.4]{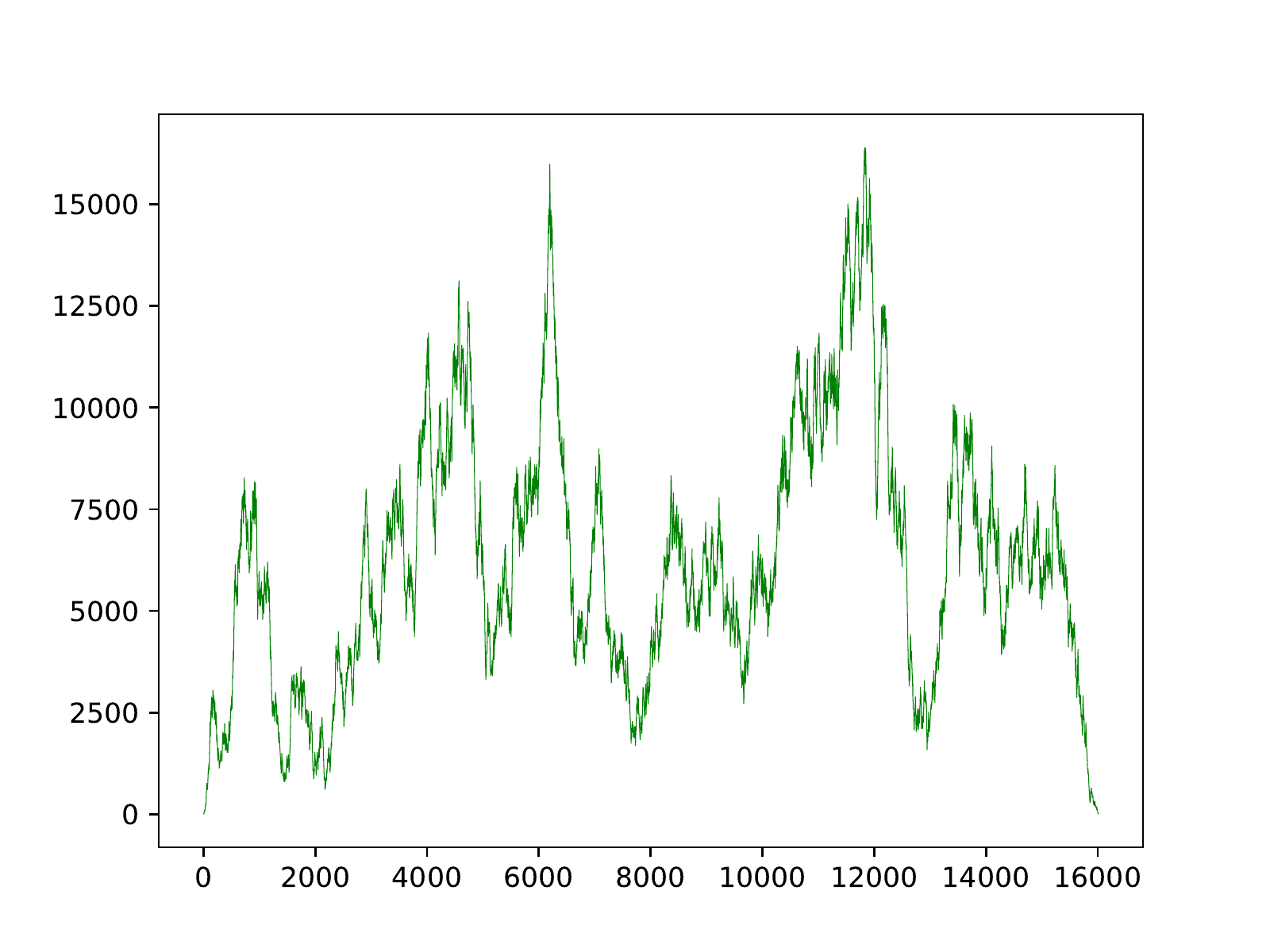}
    \includegraphics[scale=0.4]{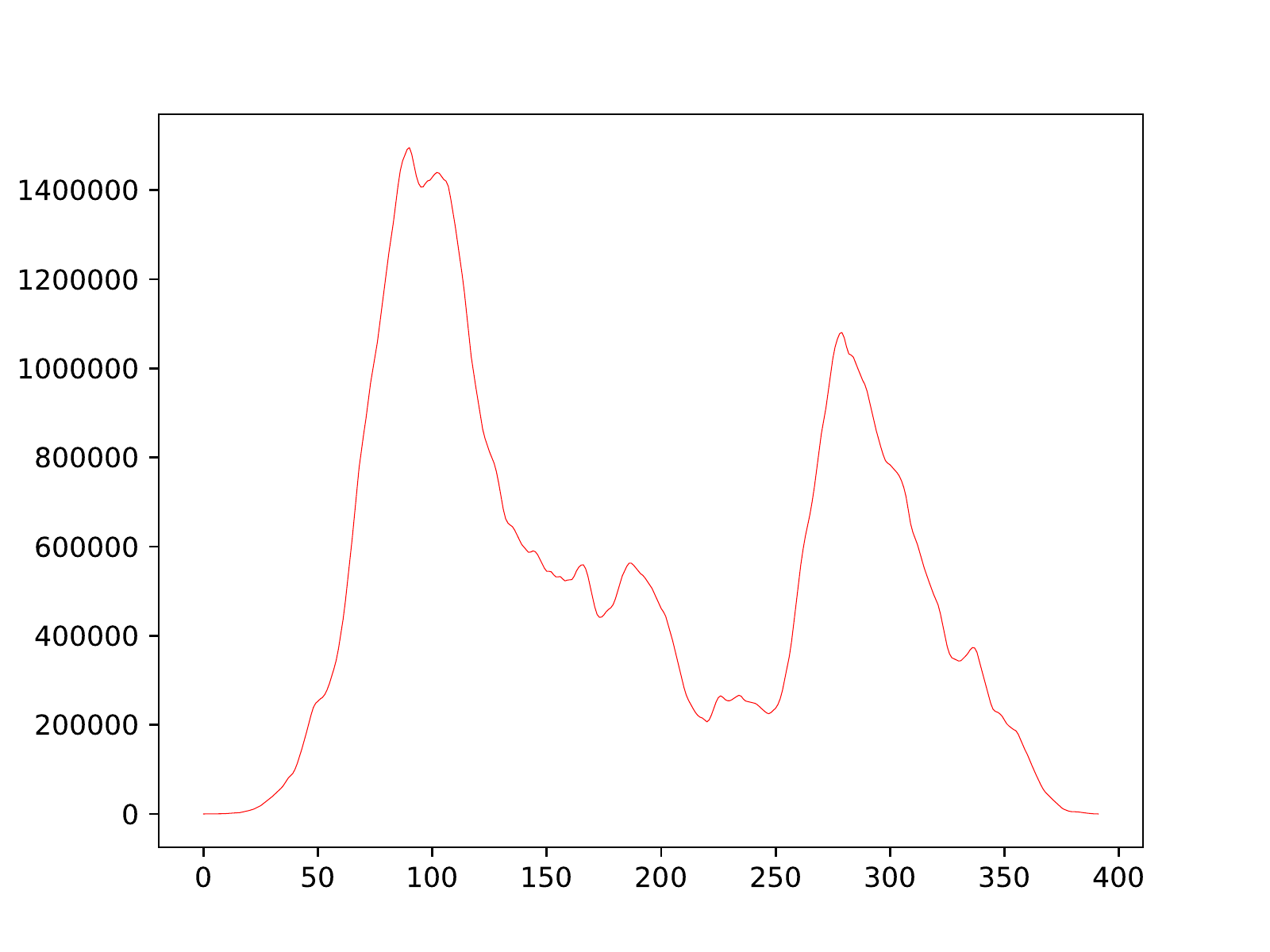}}~\\
  \centerline{  \includegraphics[scale=0.4]{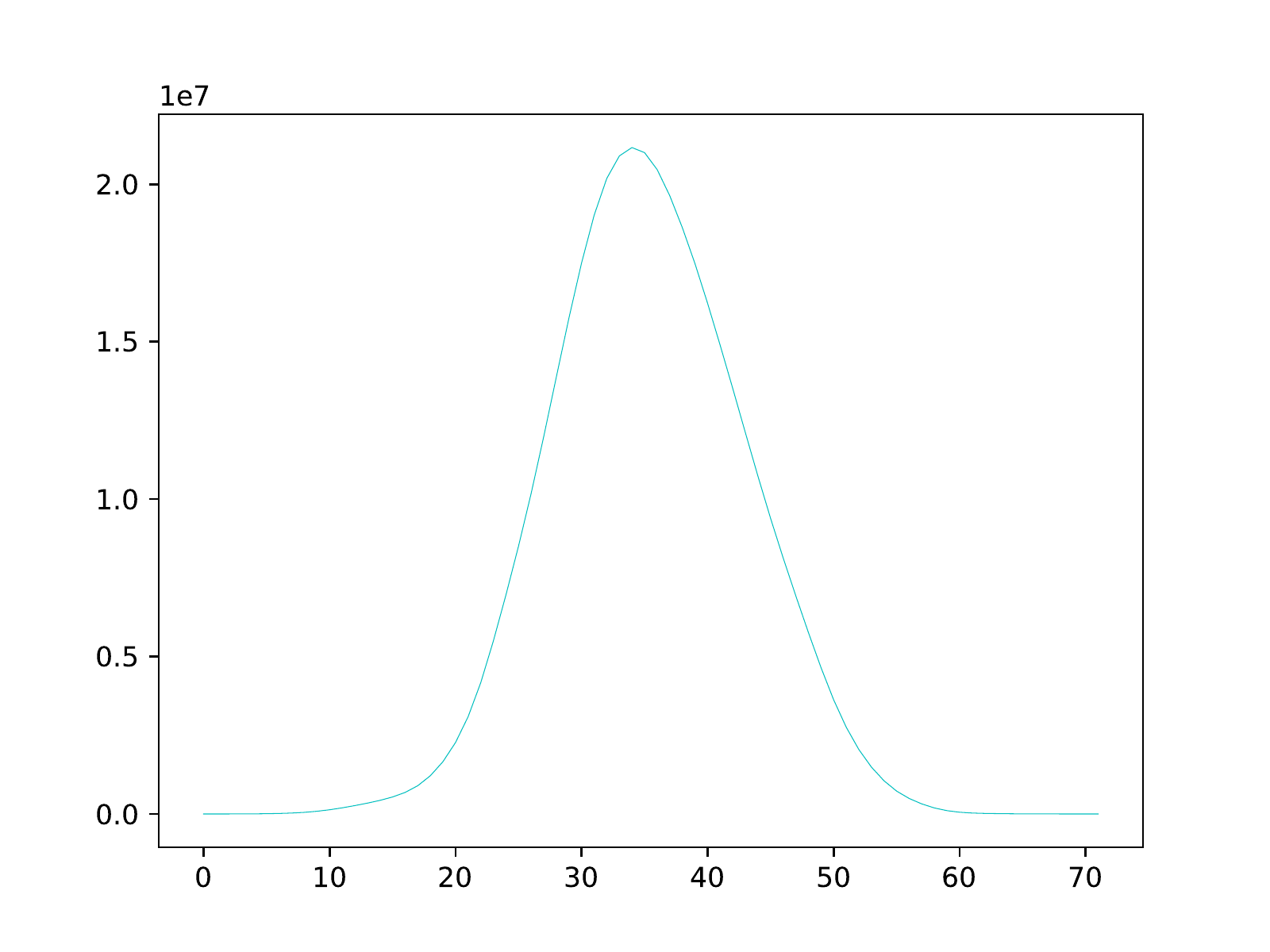}
  \includegraphics[scale=0.4]{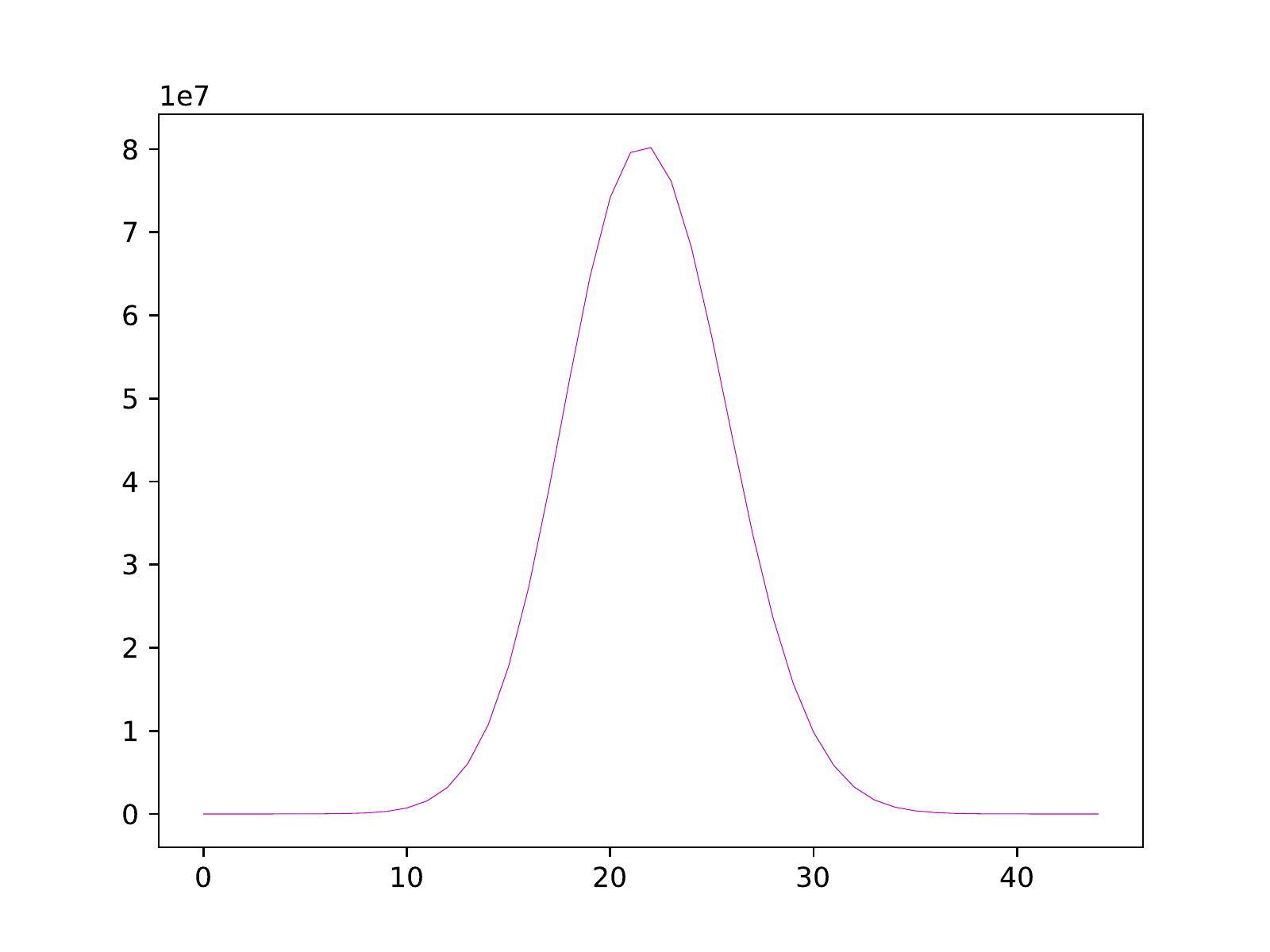}}
  \caption{\label{fig:prof0123}Profiles of the iterated trees 0, 1, 2, 3 starting from a tree with $50\times 10^6$ nodes (the profile of a rooted graph, is the sequence $(N_i,i\geq 0)$, where $N_i$ is the number of nodes at distance $i$ to the root). Observe how the profiles become smoother and smoother: asymptotically, the first profile converges to the local time of the Brownian excursion (Biane \& Yor \cite{BY}), which is known to be itself the Brownian excursion up to a change of time. The second one converges to the density of the integrated super Brownian excursion, the density of ISE (see \cite{MMsnake,janmarck05} and Chassaing \& Louchard \cite{CS} for the same result for random rooted quadrangulations with $n$ faces): it is known to be differentiable (Bousquet-Mélou \& Janson \cite{BMJ}), but is expected not to possess a second derivative; it appears on this simulation that the following profiles  become even smoother.}
  \end{figure}
\noindent\bls\, The discrete conjugation map $\Phi_N^{L\to H}$ is length preserving; its image is the set of height processes of trees with $N$ edges (they are indexed by $\cro{0,N}$). \\
\noindent\bls\, If one observes $L\cro{0,N}$, the two extremal values $L_0$ and $L_N$ correspond to the labels of the same corner of the root. This redundancy is suppressed during the construction of $H=\Phi_N^{L\to H}(L\cro{0,N})$: the $-1$ and $+1$ shifts in \eref{eq:qfqeds} correspond to the creation of a new vertex (which is drawn in blue in Fig.~\ref{fig:Conj}). This size difference is not ``a problem'': the standard encoding of quadrangulations with $n$ faces by pairs of trees shares this characteristic \cite{MM}. Hence, the construction we propose fits in the case $ \D=2$ with the case of rooted pointed quadrangulations. This is the main reason why we use this convention. \\
\bls\, The map $\Phi_N^{H\to C}$ sends $\bbH_N$ onto $\Dyck_{2N}$. Dyck paths indexed by $\cro{0,2N}$ are used to construct branching random walks whose label processes are indexed by $\cro{0,2N}$. Therefore, each iteration in the construction  of the iterated snake multiplies the number of edges by 2. The number of edges of the tree $\btnj n j$   is thus 
\ben
N_n^{(j)} = 2^{j-1}n.
\een

\paragraph{Normalized versions.} Let $\RBS_n[\D]$ be the $\D$th random discrete snake of size $n$.
  We need to fix some normalizing sequences to state the convergence of $\bC^{(j)}_n$ and $\bL^{(j)}_n$: the normalization is fixed so that $\bC^{(1)}_n$, after normalization, converges to the Brownian excursion $\se$ (the right normalization is given in \eref{eq:ehet}). The normalization of $\bL^{(j)}_n$ is chosen in accordance with \eref{eq:dsdsk}. When one iterates, since $\bC^{(j+1)}_n$ (at first order) is obtained by conjugation of $\bL^{(j)}_n$, the normalization of $\bC^{(j+1)}_n$ will be the same as that of  $\bL^{(j)}_n$. Hence, if $\alpha^{(j)}_n$ and $\beta^{(j)}_n$ are the natural normalizations of $\bC_n^{(j)}$ and $\bL_n^{(j)}$, we have  $\alpha^{(1)}_n=\sqrt{2n}$, and for $j\geq 1$,
    \[ \beta^{(j)}_n= \sqrt{(2/3) \alpha^{(j)}_n},~~ \alpha^{(j+1)}_n=\beta^{(j)}_n,\]
      which gives
      \[\alpha^{(j)}_n=  (2n)^{1/2^{j}} (2/3)^{1-1/2^{j-1}},~~
        \beta^{(j)}_n= (2n)^{1/2^{j+1}} (2/3)^{1-1/2^{j}}.\]
We then fix      
 \ben
  \label{eq:norm-processes}
\bpar{lcl}  \bc^{(j)}_n(t)&=&\dis{\bC^{(j)}_n(2^j n t)}\,/\,{\alpha^{(j)}_n}, \\
  \Bell^{(j)}_n(t)&=&\dis{\bL^{(j)}_n(2^j n t)}\,/\,{\beta^{(j)}_n}.\epar
  \een
  The main important feature, is the order $n^{1/2^{j}}$ of (the normalization of) $\bC^{(j)}_n$ and  $n^{1/2^{j+1}}$ of $\bL^{(j)}_n$.

We call {\it normalized $ \D$th random discrete snake} the process
\ben \label{eq:disn}
\rbs_n[ \D]:=\l(\l[ \bc^{(1)}_n,\Bell^{(1)}_n\r], \ldots,\l[ \bc^{( \D)}_n,\Bell^{( \D)}_n\r] \r).
\een
Each of these processes $\bC^{(j)}_n$ and $\bL^{(j)}_n$ is considered as being interpolated between discrete points, and then $\rbs_n[ \D]$ belongs to $\RS^{ \D}$.
 The corner set of  $\bc^{(j)}_n$ is
 \ben \label{eq:ASNj}
 {\sf AS}_{n,j}=\l\{\frac{k}{2N_n^{(j)}},~~ 0\leq k \leq 2N_n^{(j)}\r\}.
\een
The natural corner measure $\lambda^{(j)}_n$ on the normalized tree (denoted by $\btj j_n$) 
encoded by $\bc^{(j)}_n$ is
\ben\label{eq:muj}
\lambda^{(j)}_n=\frac{1}{2N_n^{(j)}}\sum_{x \in {\sf AS}_{n,j}} \delta_x,
\een
which is the uniform measure on the (normalized) corner set of the $j$th tree. The measured version of the normalized $D$th random discrete snake
is
\[\rbs_n[ \D]:=\l(\l[ \bc^{(i)}_n,\Bell^{(i)}_n,\lambda^{(j)}_n\r], 1\leq j \leq n \r).\]

  \subsection{Feuilletages as measured spaces }\label{sec:meas}

We now give the definition of a feuilletage associated to a $ \D$-snake. This generalizes Def.~\ref{defi:iBm} for measured (non-necessarily random) objects.
 \begin{defi}\label{defi:RRR} The $D$-feuilletage associated to a  $ \D$-snake $
 {E_D}=\l( \l[ f^{(j)},w^{(j)},\mu^{(j)}\r] , 1\leq j \leq  \D \r)$ is the space denoted by $\Ref\l(
 {E_D}\r)=[0,1]/\sim_{\D}$, where  $\sim_{\D}$ is the coarsest equivalence relation on $[0,1]$ refining all the following equivalence relations $\sim_{[m]}$ for $1\leq m \leq \D$, defined for  $x,y\in [0,1]$  by  
\ben\label{eq:qsfyur2}
x \sim_{[m]} y 
\equi \l[ D_{f^{(m)}}\l(x-A^{(m)}\mod 1, y-A^{(m)}\mod 1\r)=0 \textrm{ for } x,y \in {\sf Support}(\mu^{(m)})\r],
\een
where $A^{(m)}= \min \argmin w^{(1)}+\cdots+\min \argmin w^{(m-1)}$,
in other words, only the points of $[0,1]$ encoding some corners for $f^{(m)}$ will be identified under condition \eref{eq:qsfyur2}.
\end{defi}

\begin{rem}
The $D$th random feuilletage as defined in Def.  \ref{defi:iBm}) satisfies the identity
  \ben
  \RR{\D} \eqd  \Ref\l([\bh^{(i)},\Bell^{(i)},\lambda],1\leq i \leq \D\r).\een
 \end{rem}
Before representing normalized discrete feuilletages as foldings of some discrete snakes, we discuss the non-normalized feuilletages.

\subsubsection{Discrete iterated feuilletages (before normalization)}

The construction of the $ \D$th random discrete  feuilletage  relies on the  $\D$th random discrete snake of size $n$. It can be viewed as a procedure ``for gluing'' the $N_n^{(j)} +1$ nodes of the $j$th tree in the $N_n^{(j)}=2 N_n^{(j-1)} $ corners of the $(j-1)$th tree, one node per corner, except for the exceeding node, the root, which is not glued. The nodes of the $j$th tree glued in different corners of the same node of the $(j-1)$th tree are identified. This point of view is detailed in Sec.~\ref{sec:GraphSection}.

In the sequel, for  $\RBS_n[ \D]$ the $\D$th random discrete snake, we will set
\ben
{\bf a}^{(j)}_n := \min \argmin \bL^{(j)}_n, ~~\textrm{ for any } 1\leq j \leq  \D.
\een
The discrete conjugation map $\Phi_{N_n^{(j+1)}}^{L\to H}$ sends the label of the corner $c$ in $\bL_n^{(j)}$ (the index of the corner of the tree $\btnj n j$ on which $\bL_n^{(j)}$ has been defined) to the index $\Ba_n^{(j)}+c-1$ of $\bH_n^{(j+1)}$, in other words, the $(\Ba_n^{(j)}+c-1 \mod N_n^{(j+1)})$th node of the tree $\btnj n {j+1}$ encoded by $\bH_n^{(j+1)}$. Hence, in the tree $\btnj n {j+1}$, the $c$th node, for $c\geq 1$
corresponds to the corner $\Ba_n^{(j)}+c-1 \mod N_n^{(j+1)}$ of $\bL_n^{(j)}$. Again, when we deal with iterated identifications as we will do, it is easier to use as reference, the index set of the initial tree $\btnj n 1$. 
\begin{defi} 
\label{defi:RDS} 
Let $ \D\geq 1$ and $n\geq 1$ be two fixed parameters, and $\RBS_n[ \D]=\big(\big[ \bC^{(j)}_n,\bL^{(j)}_n\big], 1\leq j \leq  \D\big)$  a $\D$th   random discrete snake. Let $\btnj n j$ be the random  planar tree with contour process $\bC^{(j)}_n$; we recall that the number of edges of $\btnj n j$ is $N_n^{(j)}=  2^{j-1}\,n$. 

We call $ \D$th  random discrete  feuilletage $\DR{n}{\D}$  of size $n$  the graph  obtained by identification of the nodes of the $\btnj n j$'s, for $j=1,\ldots, \D$ as follows. For all $j\geq 2$, for all pairs $(c,c')\in \cro{1,N_n^{(j)}}^2$ such that the corners $c$ and $c'$ of $\btnj n {j-1}$ are both corners of the same node, identify in $\btnj n j$ the two nodes $\Ba_n^{(j-1)}+c-1 \mod N_n^{(j)}$ and $\Ba_n^{(j-1)}+c'-1 \mod N_n^{(j)}$.
\end{defi}
\noindent\bls\, The edges of $\DR{n}{\D}$ coincide with those of $\btnj n \D$, so that the number of edges of $\DR{n}{\D}$ is $N_n^{(\D)}=  2^{\D-1}\,n$.\\
\noindent\bls\, The vertex set of  $\DR{n}{\D}$  coincides with that  of $\btnj n 1$ union the set $\{r^{(j)}_n, 2\leq j \leq  \D\}$ where $r^{(j)}_n$ is the root of $\btnj n j$, since as discussed above the definition, the root of $\btnj n j$ is new and won't be identified with any node of the $\cup_{m\leq j-1} \btnj n m$. The number of vertices of $\DR{n}{\D}$ is thus $n+\D$.

It remains to normalize this object and to represent it using the map $\Ref$.
\subsubsection{Normalized discrete feuilletage as  foldings of normalized discrete snakes} 
\label{sec:nvotk} Recall the definition of the normalizations \eref{eq:norm-processes}.
\begin{defi}\label{def:discrr} 
  We call $D$th normalized random discrete feuilletage the space
  \ben
  {\bf r}_{n}[\D]= \Ref\l(\l[ \bc^{(j)}_n,\Bell^{(j)}_n,\lambda_n^{(j)}\r], \textrm{ for } j \in \cro{1,D} \r).
  \een
\end{defi}

  \color{black}
Two important remarks have to be done:
\begin{rem} \label{rem:qfzrj} The tree $\btnj n j$ encoded by the contour $\bC_n^{(j)}$ defined in Definition \ref{defi:RDS} must be submitted to a rescaling of order $1/n^{1/2^j}$ to converge, so that the scaling depends strongly on $j$. A priori, this fact makes unnatural the distance $d^{(1)}$ for which walking on $\bt^{(j)}$ has a cost ``independent from $j$'', when these trees appear as limits of the $\btnj n j$ after much different normalizations. The distance $d^{(2)}$ does not have this flaw, since all trees $\btnj n j$ for $j<D$ are used to make identifications, to somehow create shortcuts: their inner distances are not really used.
\end{rem}
The second remark concerns the convergence of (rooted) discrete feuilletages:
\begin{rem} It is possible to equip the space of feuilletages with the induced distance between the measured snakes encoding them, or with the (Prohorov)-Gromov-Hausdorff distance between isometry classes of compact metric spaces, and try to prove that, for one of these metrics, ${\sf\bf r}_{n}[\D]\dd \RR{\D}$, which is reasonable guess for both topologies. However, as said several times before, we are not able to prove it for the moment, for several technical reasons. We will prove the convergence for a pointed version instead, for a metric defined between corresponding pointed snakes.
\end{rem}

\section{Pointed variants}
\label{sec:pointed}
\setcounter{equation}{0}
The previous constructions, discrete or continuous have the disadvantage of relying on $\Conj$ and its discrete analogue $\Phi^{H\to C}\circ\Phi^{L\to H}$, which  is not continuous, as said previously. 
The discontinuity of $\Conj$ comes from the following situation: consider $f\in C^0[0,1]$ such that
\[ 0~ < ~a=\min \argmin f ~<~ b= \max \argmin f ~<~1. \]
It is easy to construct a sequence of functions $(f_n)$ in $C^0[0,1]$ such
that $\|f_n- f\|_{\infty}\to 0$, but $$\min \argmin f_n \to b,$$
or such that $\min \argmin f_n$ has $\{a,b\}$ for set of accumulation points  or even, that does not have $a$ nor $b$ as an accumulation point if $\#\argmin f >2$.\par
This implies that we cannot deduce the convergence of $(f_n,\Conj(f_n))$ to $(f,\Conj(f))$ when $\|f_n-f\|_{\infty}\to 0$, and even if we know that $(f_n,\Conj(f_n))$ converges uniformly to $(f,g)$, we cannot deduce that $g= \Conj(f)$.
The minimal property which would remove this problem would be the proof of the a.s.~uniqueness of $\argmin \Bell^{(j)}$ for all $j$'s, and for the moment, we are not able to prove this uniqueness. \par

The strategy we adopt instead is to use the fact that $\Conj(f_n)$ and $f_n$ seen as $1$-periodic functions are equal up to a change of origin: instead of working on $(C[0,1],\|.\|_{\infty})$, we will use a topology which allows us to identify functions that are equal up to the change of origin evoked above. The chance is that this topology coincides with the right topology on pointed trees (that is, which makes two rooted trees equivalent  if they are rooted at corners of the same root vertex). This permits to still get an intuitive understanding of the phenomenon into play.

\subsection{Pointed real trees}
\label{sec:pointed-trees}
By definition, a rooted tree $T_g=[0,1]/\sim_g$ is rooted at a corner. The different corners of the root are $g^{-1}(0) \setminus \{1\}$.
We call {\it pointed tree} an equivalence class of  trees rooted at the different corners of the same root vertex. For any $a\in [0,1]$,  formally define the $a$-shift $\Psi_a$ as the following map  
defined on $C^0[0,1]$:
\ben\label{eq:Psia}
\app{\Psi_a}{C^0[0,1]}{C^0[0,1]}{g}{x\mapsto\Psi_a(g) (x)= g(x+a \mod 1)-g(a),}
\een
which is a conjugation map, meaning the exchange of ``two sections'' of the graph of $g$, but starting from a given corner $a$ instead of $\min \argmin g$. 
Of course, for any $f\in C^0[0,1]$, 
\ben \label{eq:V} 
 \Conj(f)=\Psi_{\min \argmin(f)}(f).
\een

Introduce the following equivalence relation on $C^0[0,1]$: we say that $f\sim_{\Psi} g$ if there exists $a\in[0,1]$ such that $g=\Psi_a(f)$. Denote the quotient space by
\ben
\QC= C^0[0,1]/{\sim_\Psi},
\een
and denote by $\bar{f}$ the class of a function $f$.
\begin{rem} A measure component can be added: $\Psi_a$ then acts on $C^0[0,1]\times {\cal M}([0,1])$ with values in $C^0[0,1]\times {\cal M}([0,1])$, with (a slight abuse of langage): $\Psi_a(f,\mu)=(f',\mu')$ for $f'=\Psi_a(f)$ as defined above and $\mu'(\cdot)=\mu(\cdot+a \mod 1)$.\end{rem}
The following proposition is proven in Sec.~\ref{sec:pppt}.
\begin{pro}\label{pro:top}
The map $D_\Psi: \QC^2 \to \R^+$ defined by  
\ben\label{eq:fqf}
D_\Psi(\ov{g_1},\ov{g_2})= \inf_{a} \| g_1-\Psi_a(g_2)\|_{\infty}
\een
is a distance on $\QC$ and equipped with this distance, $\QC$ is a Polish space (for the measured extension, add the total variation distance between the measure components).  
\end{pro}
For any $f\in C^{0}[0,1]$ denote by
\ben
\PTC(f)=\l\{\Psi_a(f), a \in \argmin(f)\r\}.
\een

The following result is immediate: 
\begin{lem} \label{defi:efa}
Two functions $f$ and $g$ in $C^+[0,1]$ are in the same tree class (that is, $\PTC(f)=\PTC(g)$) iff there exists $a\in[0,1]$ such that
\ben\label{eq:fga}
f(\cdot )= g(\cdot +a \mod 1),
\een
in which case \ben\label{eq:qff}
a\in \argmin (g)= g^{-1}(0).
\een
Hence, two rooted trees are in the same pointed tree class if each tree can be obtained from the other by a rerooting at a corner of the root vertex (for the measured version, the push-forward measure of the root corner of $\mu_f$ is required to coincide with that of $\mu_g$, as explained in Rem.~\ref{eq:fsddq}).
\end{lem}

The next lemma is a direct consequence of Proposition \ref{pro:top} and its proof.
\begin{lem}\label{lem:Dpsi}
  The map $D_{\Psi}$ is a distance on the set of pointed tree classes (with a simple adaptation for measured pointed tree classes).
\end{lem}

\begin{defi} A sequence of pointed trees with classes of contour processes $\bar{c_n}$ is said to converge to $\bar{c}$ if $D_{\Psi}(\bar{c_n},\bar{c})\to 0$. 
  \end{defi}
The following proposition will be proven in Sec.~\ref{sec:proof-of-prop-cv}.
\begin{pro} 
  \label{pro:cv} Let $\by,\by_0,\by_1,\cdots$ be a sequence of processes taking their values in $C^0[0,1]$ such that $\by_n \dd \by$ in $(C[0,1],\|.\|_{\infty})$. 
  \bir
  \itr The sequence $(\Conj(\by_n), n \geq 0)$ is tight in $C[0,1]$. The limits of the converging subsequences (in $C[0,1]$) are all in the tree class of $\Conj(\by)$.
  \itr If a.s.~$\#\argmin \by=1$, then $\Conj(\by_n)\dd \Conj(\by)$ in $C[0,1]$.

  \itr    $\bar\by_n\dd \bar\by$ in $\QC$.
  \itr     $\bar\Conj(\by_n)\dd \bar\Conj(\by)$ in  $\QC$.
  \eir
\end{pro}

\subsection{Pointed snakes}
\label{sec:pointed-snakes}

We extend the notion of pointed trees to pointed snakes:
\begin{defi} Two  snakes $(f_1,w_1)$ and $(f_2,w_2)$ taken in $\RS$ are said to be the same pointed snake if there exists $a\in[0,1]$ such that,
\ben\label{eq:fd}
\l(f_1(x), w_1(x)\r) = \l(f_2(x+a \mod 1), w_2(x+a \mod 1)\r), ~~~\textrm{ for all } x \in [0,1].
\een
\end{defi}
Again, as for the case of pointed trees (Def.~\ref{defi:efa}), if \eref{eq:fd} holds, then $a\in \argmin (f_2)=f_2^{-1}(0)$ (notice that $w_2(a)=0$ as $a$ is a corner of the root, so that $\Psi_a(w_2)(x)=w_2(x+a \mod 1)$). 
``To be the same pointed snake'' is an equivalence relation $\sim_{\bullet}$. 
The state space of pointed snakes is
\ben
\PS = \RS /\sim_{\bullet}.
\een
Denote by $\pi_{\bullet}$ the canonical projection from $\RS$ to $\PS$. 
A distance on  $\PS$ is given by the following extension of $D_{\Psi}$ (we keep the same notation):
\ben
D_{\Psi}((f_1,w_1),(f_2,w_2))= \inf_a  \l(\|\Psi_a(f_1)-f_2\|_{\infty}+ \|\Psi_a(w_1)-w_2\|_{\infty}\r).
\een
We call $ \D$-pointed snake an element of $\PS{}^{,  \D}:=(\PS{})^{ \D}$. 
We equip this set with the distance
\ben\label{eq:DPK}
D_{\Psi, \D}\Bigl[([f_j,w_j], 1 \leq j \leq  \D),([f'_j,w'_j], 1 \leq j \leq  \D)\Bigr]=\sum_{j=1}^{ \D} D_{\Psi}\l(\l(f_j,w_j\r),\l(f'_j,w'_j\r) \r).
\een
\begin{defi}
\label{def:consistent} An element $\l[(f_i,w_i),1\leq i \leq  \D\r]$ of $\RS^{ \D}$ is said to be consistent if $f_{i+1}\sim_{\Psi} w_i$ for every $1\leq i \leq  \D-1$.
\end{defi}
Again this extends to measured pointed snakes: two measured (rooted) 
snakes  $(f,w,\mu)$ and   $(f',w',\mu')$ are said to be in the same measured pointed snake class, if for some $a\in \argmin f$,
\ben
\bigl(f(a+ x \mod 1), w(a +x \mod 1), \mu(a+ . \mod 1)  \bigr)  = \l( f'(x), w'(x), \mu' \r), ~~\textrm{ for all }x \in[0,1].
\een \normalsize
We extend the projection $\pi_{\bullet}$ to measured (rooted) snakes. 
The set of measured pointed  snakes obtained in such a way can be equipped with a metric which extends $D_{\Psi, \D}$ by adding the total variation distance between the corner measures of the trees.

\subsection*{Branching random walks and pointed snakes} \label{sec:BRWPS}

Consider a branching random walk with underlying tree a rooted tree $T$ as defined in Section \ref{eq:ds}. The branching random walk is defined by some spatial increments placed on the vertices of $T$ different from the root. A change of root corner can be done without modification of the spatial increments, which amounts, eventually, to just shifting the encoding processes of this labeled tree as follows.
Let $(C_T,L_{T})$ be the tour of a discrete snake with underlying rooted tree $T$, where 
$C_T=(|c_t(k)|,0 \leq k \leq 2\|T\|)$ is  the contour process of $T$, and let
\[R=\l(\argmin C_T\r)\setminus \{2\|T\|\}\]
be the set of corners of the root vertex of $T$. For each element $a \in R$,  call $T(a)$ the tree $T$ rerooted at its corner $a$. Since $C_{T(a)}=L_{T(a)}=0$, the tour of the discrete snake now indexed by $T(a)$ is simply, 
\ben
C_{T(a)}(k)& = &C_T(a +k \mod 2\|T\|), \\
L_{T(a)}(k)& = &L_T(a +k \mod 2\|T\|) \textrm{ for }0\leq k \leq 2\|T\|.
\een
The set of discrete snakes $\l((C_{T(a)},L_{T(a)}),a\in R\r)$ forms a discrete (non-normalized) pointed snake.
\begin{rem}\label{rem:qdeth} The normalized processes $\bc^{(j)}_n$,$\Bell^{(j)}_n$ were defined in \eqref{eq:norm-processes}. $\bh^{(j+1)}_n$ is obtained by conjugating $\Bell^{(j)}_n$. We do not impose the consistence condition (Def.~\ref{def:consistent}) in the initial definition of  $ \D$th pointed snakes, because in the discrete case, the processes $\bc^{(j+1)}_n$ and $\Bell^{(j)}_n$ are not in general in the same class modulo $\sim_{\Psi}$ (while the continuous $\D$th Brownian snake is consistent).
\end{rem}

\begin{defi}\label{def:PBS} We call $ \D$th pointed Brownian snake a process taking its values in $\PS{}^{,  \D}$,
  \ben \label{eq:dethher}
  \pbs[ \D]:=\l(\pi_{\bullet}\l[ \bh^{(j)},\Bell^{(j)}\r], 1\leq j \leq  \D \r),
  \een 
  where $\rbs[ \D]=\l(\l[ \bh^{(j)},\Bell^{(j)}\r], 1\leq j \leq  \D \r)$ is a $\D$th Brownian snake. For the measured version, the measure of each tree is $\lambda$.\par
The process
 \ben \label{eq:dethhern}
  \pbs_n[ \D]:=\l(\pi_{\bullet}\l[ \bc^{(j)}_n,\Bell^{(j)}_n\r], 1\leq j \leq  \D \r).
  \een
  is called $\D$th (normalized) random pointed discrete snake. For the measured version, the measure of the $j$th tree is $\lambda_{n}^{(j)}$ as defined in \eref{eq:muj}.
  \end{defi}
 Notice that $\lambda_n^{(j)}$ is invariant by change of root. Also, we only define pointed random discrete objects for the normalized versions, so that we will drop the ``normalized" in the name of the objects to avoid lengthy names.

\subsection*{$ \D$th random pointed discrete snake and convergence}

\begin{theo}\label{theo:CVpointedSnakes} The sequence of (normalized)  $ \D$th random pointed discrete snake converges 
in distribution towards the $\D$th pointed Brownian snake:
\[ \pbs_n[ \D]\dd  \pbs[ \D]\textrm{ in }\l(\PS{}^{, \D},D_{\Psi, \D}\r).\]
The result extends to the measured version.
\end{theo}
This will be proven in Sec.~\ref{sec:proof-of-cv-pointed-snakes}.
\begin{rem}
  Theorem \ref{theo:CVpointedSnakes} may seem somewhat weaker than it is: recall Section \ref{sec:TS} about the synchronization of trees. Since the rooted tree $\bc^{(1)}$ is a Brownian excursion, and thus a.s.~reaches 0 only at 0 and 1, the pointed tree class of $\bc^{(1)}$ possesses a unique rooted tree. Again, using Section~\ref{sec:TS}, a.s. there is a single well-defined way to synchronize the trees $\bc^{(1)},\cdots,\bc^{( \D)}$ since none of the  $\bc^{(j)}$ are periodic (see the argument of the third point of Theorem \ref{theo:ini-rec}). However, the synchronization does not apply to the discrete snakes $\pbs_n[\D]$, at least not directly, because $\bc_n^{(j)}$ is not obtained from $\Bell_n^{(j-1)}$ by a simple change of origin (as explained in Rem.~\ref{rem:qdeth}).
\end{rem}

\begin{rem}\label{rem:BT}  Theorem \ref{theo:CVpointedSnakes} implies that the right scale for the tree encoded by $\bh_n^{( \D)}$ has to be $n^{1/2^ \D}$, since normalized by this quantity, it converges (as a pointed tree) to a non-trivial limit.
\end{rem}

We are not able to write an analogue of Theorem \ref{theo:CVpointedSnakes} for {\it rooted} 
iterated snakes as said several times, because of the non-continuity of $\Conj$: pointed iterated snakes are simpler and it is the case even for $\pbs[2]$, which appears in relation with the rooted-pointed quadrangulations.

\subsection{Convergence of random pointed iterated discrete feuilletages}

\label{sec:measure}
We introduce the notions of iterated  random feuilletages, rooted and pointed, discrete and continuous. All these objects are encoded by iterated snakes (rooted and pointed), in which some corner measures encode the positions of the nodes.
In this subsection, we define a topology for which the convergence of discrete pointed snakes extends to the convergence of pointed discrete random feuilletages.
\begin{defi}  Consider a measured pointed $\D$-snake $E_ \D^\bullet=(\pi_\bullet[f^{(i)},w^{(i)},\mu^{(i)}], 1\leq i \leq  \D)$ in $\PS{}^{, \D}$. 
   We call pointed feuilletage associated with $E_ \D^\bullet$, denoted by $\Ref\l(E_ \D^\bullet\r)$,  
  the set of (rooted) feuilletages $\l\{\Ref\l(
  {E_D}\r) \textrm { for }
  {E_D} \in E_ \D^\bullet\r\}$ associated with the (rooted) snakes in the class of  $E_ \D^\bullet$.  
\end{defi}

Recall the definitions \ref{defi:iBm} and \ref{def:discrr} of $\RR{\D}$ and of ${\sf\bf r}_{n}[\D]$ defined using  $\rbs[\D]$ and  $\rbs_n[\D]$.
\begin{defi} We call:\\
  \bls (measured)  $D$th pointed random feuilletage, the space
  \[{\bf r}^{\bullet}[\D]=\Ref\l(\pi_{\bullet}\l[\bh^{(i)},\Bell^{(i)},\lambda\r],1\leq i \leq \D\r),\]
  \bls (measured normalized)  $D$th  pointed random discrete feuilletage, the space
  \[{\bf r}^{\bullet}_{n}[\D]=\Ref\l(\pi_{\bullet}\l[ \bc^{(i)}_n,\Bell^{(i)}_n,\lambda_n^{(i)}\r] ,1\leq i \leq \D\r).\]
\end{defi}

 For this topology, the convergence of discrete pointed snakes (Theorem \ref{theo:CVpointedSnakes}) extends immediately to the convergence of pointed discrete random feuilletages: 
 \begin{theo}\label{theo:CVMapsSnakes}For any $ \D\geq 1$,
   \ben
   {\bf r}^{\bullet}_{n}[\D] \dd {\bf r}^{\bullet}[\D]
\een 
   for the topology induced by $D_{\Psi}$.
 \end{theo}

However, this convergence does not imply the convergence for the Gromov-Hausdorff distance.
\begin{rem}\label{rem:QT}  Theorem \ref{theo:CVMapsSnakes} implies that the sequence of diameters of $ {\bf r}^{\bullet}_{n}[\D]$ (seen as graphs with edge-lengths) is tight, so that, without edge normalization, an upper bound on the scale of the {$\D$th  random  discrete feuilletage}  
  is  $n^{1/2^ \D}$ (we expect $n^{1/2^ \D}$ to be the right normalization).
\end{rem}

\subsection{A variant for the feuilletage}

We present here a variant which has the property to provide a feuilletage which scales clearly in $n^{1/2^D}$.  However, we will just sketch this alternative construction, as it seems to us that it breaks the invariance by change of pointed vertex; moreover the representations of the trees $\t_n^{(D)}$ using linear processes and the study of their asymptotics seems less tractable.

In the construction detailed in this paper, the tree  $\t_n^{(3)}$ is first constructed using the label process $L_n^{(2)}$ of the tree  $\t_n^{(2)}$. Two nodes of $\t_n^{(3)}$ are subject to two identifications: gluing the nodes of  $\t_n^{(3)}$ in the corners of $\t_n^{(2)}$ forms a quadrangulation  $\Qrec n 2 {3}$, and the labels $\bL_n^{(3)}$ can be used to recover the distances in this quadrangulation: hence the right scale of  $\Qrec n 2 {3}$ is $n^{1/8}$. However, the corners of $\t_n^{(2)}$ that are identified using $\t_n^{(1)}$ may have different  labels (for $L_n^{(2)}$): the identification by $\t_n^{(1)}$ produces some shortcuts in $\Qrec n 2 {3}$;  it is not easy to prove the scale $n^{1/8}$ for   $\Rn n 3$.

We propose the following variant:
we can adapt the labeling of $\t_n^{(2)}$, so that when identifying the vertices of the variant tree $\t_n^{(3),\star}$ according to the vertices of $\t_n^{(1)}$, only vertices that have the same label are identified. To do this, we can for instance keep in memory for each vertex $v$ of $\t_n^{(1)}$, the list $c_1^v,..., c_{{\sf deg}(v)}^v$ of its corners (in $\t_n^{(1)}$). Every vertex of  $\t_n^{(2)}$ but one corresponds to a corner of $\t_n^{(1)}$. Then, to construct the variant labeling of $\t_n^{(2)}$, proceed as follows: following the contour sequence, turn around around the tree  $\t_n^{(2)}$; when visiting the $k$th node  $u_k$ of $\t_n^{(2)}$, two cases arise:\\
\bls If the node $u_k$ is identified by $\t_n^{(1)}$ with a node $u_j$ of $\t_n^{(2)}$ with smaller index $(j<k)$, then label $u_k$ with the label of $u_j$.\\
\bls If the node $u_k$ is not identified by $\t_n^{(1)}$ with a node of smaller index, then label this node (in $\t_n^{(2)}$) with the label of its father (in $\t_n^{(2)}$) in which is added a random variable $\nu$-distributed. 
By this construction, the labeling attributed to the $k$ first vertices of $\t_n^{(2)}$ forms an interval of $\mathbb{Z}$, from what we can see that still a tree $\t_n^{(3),\star}$ can be constructed by adding (in $T_n^{(2)}$) an edge between $u_k$ and the last corner with label that of $u_k$ minus 1).\par
It allows building a planar map $\Qrec n 2 {3}{}^{\star}$ when identifying the vertices of $\t_n^{(3),\star}$ according to those of $\t_n^{(2)}$, but now when identifying the vertices of $\Qrec n 2 {3}{}^{\star}$  according to those of $\t_n^{(1)}$, we identify vertices which are at the same distance to the pointed vertex in $\Qrec n 2 {3}{}^{\star}$, so that the distances are conserved. We now sample $n+2$ random variables $\nu$ distributed to label  ${\bf T}_n^{(2)}$ instead of $2n+1$ times in the usual construction, however the standard diameter of ${\bf T}_n^{(3),\star}$ and of $\bQ_ n^{(2,3),{\star}}$, and now of the resulting feuilletage has scale $n^{1/8}$.\par
A similar construction can be proposed for $\Rn n D$ and $D>3$: the labeling of $\t_n^{(D-1),\star}$ is then produced so that the identifications by the trees $(\t_n^{(j),\star},j\leq D-2)$ only identify nodes in  $\t_n^{(D-1),\star}$ with the same label.

\section{Proofs}
\label{sec:proofs}
\setcounter{equation}{0}
\subsection{Some cornerstones of the proofs}
We factorize here some ingredients of the proofs of the theorems.
\subsubsection{Main arguments to prove the continuity of processes, and of limits of processes}
The $q$-H\"older exponent of a function $f:[0,1]\to \R$ is defined as
  \[\hol_q(f):=\sup_{0\leq x <y \leq 1} \frac{|f(x)-f(y)|}{|x-y|^q}.\]
The function $f$ is said to be $q$-H\"older for some $q\in(0,1)$ (shorthand for H\"older with exponent $q$) if $\hol_q(f)<+\infty$. 
Needless to say that a $q$-H\"older function is continuous.\par
A random process $\bX$ is said to be  Hölder continuous with exponent $q$, if a.s.~$\hol_q(\bX)$ is finite. 
We recall Kolmogorov's continuity criterion (see e.g.~Kallenberg \cite[Theo.2.23]{Kal} or  Revuz-Yor \cite[Theo.2.1 p.26]{RY}):
\begin{theo} Let $\bX$ be a process indexed by $[0,1]$ taking its values in $\R$; if for some $a,b>0$,
  \[\E\l(\l|\bX_s-\bX_t\r|^a\r)\leq C |s-t|^{1+b}, \textrm{ for all }s,t\in[0,1],\]
  then $\bX$ has a continuous version, and for any $c\in(0,b/a)$, it is a.s.~Hölder continuous with exponent $c$. 
\end{theo}

\begin{defi}
A sequence of processes $(\bX_n)$ defined on $[0,1]$ is $q$-Hölder tight, if for any $`e>0$, there exists $M>0$ and $N\geq 1$, such that for any $n\ge N$,
\ben\label{eq:sefdqd1}
`P\bigl( \hol_q(\bX_n)  \leq M\bigr)\geq 1-`e,
\een
and if $(\bX_n(0))$ is tight.
\end{defi}

Under this condition, if $(\bX_n)$ is $q$-H\"older tight for some $q\in(0,1]$, then $(\bX_n)$ is tight on $C[0,1]$ (since the subset of $C[0,1]$ of functions that satisfy $|f(0)|\leq C$ and $\hol_q(f)\leq M$ is a compact subset of $C[0,1]$).
Another criterion for tightness is the famous ``moment condition''  which is a bit more restrictive: additionally to the tightness of $(\bX_n(0))$, the existence of positive constants $\alpha, \beta, \gamma$ such that
\ben\label{eq:sefdqd}
\sup_{n \geq N} \E(|\bX_{n}(x)-\bX_n(y)|^{\alpha}) \leq \gamma |x-y|^{\beta}, ~~ \textrm{ for any } x,y \in[0,1]
\een
ensures that for any $c\in(0,\beta/\alpha)$, any limit of any converging subsequence of the sequence  $(\bX_n)$  is a.s. Hölder continuous with exponent $c$ (Kallemberg \cite[Cor.14.9]{Kal}).

\subsubsection{Height and contour processes are asymptotically indistinguishable}

The next theorem is folklore (and can be found in \cite{mm01} in the case of critical Galton-Watson trees whose offspring distributions own exponential moments). For the sake of completeness we provide a proof in the degree of generality needed here. 
\begin{theo}\label{theo:ini-rec3}
  Let $(r(n))$ be a sequence of elements of $(0,+\infty)$, such that $ r(n)\to+\infty$ and $r(n) = o(n)$. For any $n$, let $\mu_n$ be a distribution on {$\bbT_{n}$}, the set of trees with $n$ edges.
Denote by $\bh_n$ and $\bc_n$ the normalized  height and contour processes of  $\bT_n$  picked according to $\mu_n$, and defined by 
  \ben\label{eq:dsdsk2}
 \bh_n(x)&:=&\frac{H_{\bT_n}(n x)}{r(n)},~~\textrm{ for } x\in[0,1],\\
\bc_n(x)&:=&\frac{C_{\bT_n}( 2nx)}{r(n)}, \textrm{ for }x \in [0,1].
\een 
If $\bh_n\dd \bh$ in $(C[0,1],\|.\|_\infty)$,
then $(\bh_n,\bc_n)\dd (\bh,\bh)$ in $C[0,1]^2$ (in other words limiting contour and height processes are asymptotically indistinguishable, $\bc=\bh$).
\end{theo}
\begin{proof}
 We start with some considerations valid for any deterministic planar tree $T$. Let $u_0,\cdots, u_{\|T\|}$ be the nodes of $T$ sorted according to the lex.~order.   
Denote by
\ben m_T(k)=\inf\bigl\{ j, c_T(j)=u_k\}\een
the first visit time  of $u_k$ by the depth first traversal, or equivalently, the index of the first corner of $u_k$. A simple induction allows proving that
  \ben\label{eq:ree1}m_T(k)=2k-H_T(k), \textrm{ for any } k \in\cro{0,\|T\|}\een
(see \cite[Lemma 2]{mm01}), and for any $p$, and any $k\in \cro{m_T(p),m_T(p+1)-1}$, 
\ben\label{eq:ree2}
H_T(p+1)-1\leq C_T(k)\leq H_T(p),
\een
by  \cite[Lemma 3]{mm01} (this can be proved by a simple induction again).
Hence, for any $k$,
\ben \label{eq:ree3}
\bpar{l} 2k-\max H_{T}\leq m_T(k) \leq 2k, ~~\forall k \in\cro{0,\|T\|},\\
  m_T(k) \leq 2k \leq m_T(k)+\max H_{T}\leq  m_T( \min\{k+\max H_{T},\|T\|\}),
  \epar
  \een
  since $m_T(k+1)\geq m_T(k)+1$.

Now, defining $\Delta_{T}:=\max_p\l|H_{T}(p)-\H_{T}(p+1)\r|+1$, from 
 \eref{eq:ree3} and \eref{eq:ree2},
\ben\label{eq:contrib2}
 \sup_{p \in \cro{0,\|T\|}} \l\{\l|C_T(m_T(p))-C_T(2p)\r|\r\} \leq \Delta_{T}+ w_{\max H_{T}}(H_T),
 \een
 where $w_{\delta}(f)=\max \{ |f(x)-f(y)|, |x-y|\leq \delta\}$ is the modulus of continuity of $f$. Indeed, $C_T(m_T(p))=H_T(p)$ and $C_T(2p)$ coincide, up to an additive term bounded by $\Delta_{T}$, with $H_{T}(j)$ for $j$ such that $m_T(j)\leq 2p\leq m_T(j+1)-1$ (by \eref{eq:ree2}). Therefore, by  \eref{eq:ree3}, $2j-\max H_T\leq 2p\leq 2(j+1)-1$, so that $j\in \cro{(2p-1)/2,p+\max H_T/2}= \cro{p,p+\max H_T/2}$.  \par 
Let us now prove that under the hypotheses of the theorem, for a random tree $\bT_n$ taken under $\mu_n$,
 \ben
 \sup_p\frac{\l|C_{\bT_n}(2p)-H_{\bT_n}(p)\r|}{r(n)}\dd 0.
 \een
 Taking into account the convergence $\bh_n\dd\bh$, this implies that $\|\bc_n-\bh_n\|_{\infty}\dd 0$, and the continuity of $\bc_n$ allows concluding. 
 By  \eref{eq:contrib2}, it suffices to prove that $\Delta_{\bT_n}/r(n)\dd 0$, and $ w_{\max H_{\bT_n}}(H_{\bT_n})/r(n)\dd0$.\\
 {\bls}  By \eref{eq:ree2}, since $\bh_n\dd \bh$, for any $`e', `e>0$, there exists $\delta>0$ such that
 \ben\label{eq:fluctuationhn}
 \lim_n \P\bigl(w_{\delta}(\bh_n)>`e\bigr)\leq `e'.
 \een
 In particular, for any $`e>0$, $\lim_n\P\bigl( \Delta_{\bT_n}\geq `er(n)\bigr)\to 0$.\\ 
{\bls}  For some $M>0$, denote by $E_{n,M}$ the event
\[E_{n,M}= \{\max H_{\bT_n}\leq M r(n)\}.\]
Since the normalized  height process converges in $(C[0,1],\|.\|_\infty)$, then $r(n)^{-1}\max H_{\bT_n}$ is tight, so that for any $`e>0$, there exists $M$ such that $`P(E_{n,M})\geq 1-`e$ for $n$ large enough.
Hence, for some $a>0$, write
\be
\P\bigl(w_{\max H_{\bT_n}}(H_{\bT_n})/r(n)\geq a\bigr)&\leq& `e+\P\bigl(w_{\max H_{\bT_n}}(H_{\bT_n})/r(n)\geq a , E_{n,M}\bigr)\\
&\leq & `e+\P\bigl(w_{Mr(n)}(H_{\bT_n})/r(n)\geq a \bigr).
\ee 
Since $w_{Mr(n)}(H_{\bT_n})/r(n)= w_{Mr(n)/n}(\bh_n)$ and $r(n)=o(n)$ by hypothesis (we have $Mr(n)/n\to 0$), and by \eref{eq:fluctuationhn}, $\P(w_{\max H_{\bT_n}}(H_{\bT_n})/r(n)\geq a)\to 0$ for any $a>0$. 
 \end{proof}

\subsection{Proof of Theorem \ref{theo:biendefini}}
\label{sec:proof-of-cont}

By Kolmogorov's continuity criterion, the Brownian motion, which satisfies $B_t-B_s\eqd \sqrt{(t-s)}B_1$, is $q$-H\"older for any  $q\in(0,1/2)$, because  $ \E\l(|B_t-B_s|^{\gamma}\r) \leq c_\gamma|t-s|^{\gamma/2}$,
  where $c_\gamma= \E(|B_1|^{\gamma})<+\infty$. The Brownian excursion $\se$ is also $q$-H\"older for any  $q\in(0,1/2)$, since $\se$ can be represented as the rescaled excursion of a Brownian motion that straddles 1:
  \[(\se_{t}, 0\leq t \leq 1) \eqd \l( \frac{|B_{d+(g-d)t}|}{\sqrt{d-g}}, 0\leq t \leq 1 \r),\]
  where $g=\sup\{t<1, B_t=0\}$ and $d=\inf\{t>1, B_t=0\}$. From this representation, since  $\sqrt{d-g}$ is a.s. finite, one sees that the Brownian excursion is $q$-H\"older.\par
  Now, assume by induction that we have proven that $\bh^{(j)}$ is $q$-H\"older for any $q\in[0,1/2^j)$ and more precisely that we proved that for any $q\in [0,1/2^j)$, for any $`e>0$, there exists $M$ such that
   \ben\label{eq:etgz} \P\bigl(\hol_q(\bh^{(j)}) \geq M\bigr) \leq `e.\een
  Let us prove that this property is also true for $j+1$.\par
  Consider the event  $E_{j,q,M}=\{\hol_q(\bh^{(j)})\leq M\}$. 
 Since  $\Bell^{(j)}(x) - \Bell^{(j)}(y)\eqd {\cal N}(0, D_{\bh^{j}}(x,y))$, for any $a\geq 1$, any $x,y\in[0,1]$, 
 \ben
 \E\l(\l|\Bell^{(j)}(x) - \Bell^{(j)}(y) \r|^{a} \,|\, E_{j,q,M}\r)&= & \E(|N_1|^a) \E\l(D_{\bh^{j}}(x,y)^{a/2}\,|\, E_{j,q,M}\r)\\
 &\le & \E(|N_1|^a) 2^{a/2} M^{a/2}(x-y)^{qa/2} \\
 &\leq& C(x-y)^{qa/2} \label{bounddistance1}
 \een
 because  $D_{\bh^{j}}(x,y) \leq 2M|x-y|^{q}$ on $E_{j,q,M}$, as for any $u$ in $[s,t]$, $|\bh^{j}(s)- \bh^{j}(u)|+|\bh^{j}(t)-\bh^{j}(u)| \leq 2M(t-s)^{q}$.
  Hence, conditionally on $E_{j,q,M}$, {$\Bell^{(j)}$} satisfies Kolmogorov's criterion, and it is then $q/2$-H\"older. It follows readily that for $M'$ large enough, $\P\bigl(\hol_{q/2}(\Bell^{(j)}) \geq M'\bigr)\leq `e$.
 Now, if a function $f$ is $q$-H\"older, so does $\Conj(f)$, from what we conclude. \hfill $\Box$

 \subsection{Proof of Proposition \ref{pro:top} }
 \label{sec:pppt}
 First, $D_{\Psi}$ is a distance on $\QC$: since $D_{\Psi}$ is symmetric, in fact the main point is that in \eref{eq:fqf}, the $\inf$ can be replaced by a $\min$. For this, take a sequence $a_n$ such that $\| g_1-\Psi_{a_n}(g_2)\|_{\infty} \to \inf_{a} \| g_1-\Psi_a(g_2)\|_{\infty}$. Extract from this sequence a converging subsequence (this is possible since $(a_n)$ lives in the compact set $[0,1]$); let $b$ be the limit of this subsequence $(a_{n_k},k\geq 0)$. Since equicontinuity and pointwise convergence imply uniform convergence (on a compact set), $\|\Psi_{a_{n_k}}(g_2)- \Psi_b(g_2)\|_{\infty}\to 0$ (since the functions in $\{\Psi_a(g_2),a\in [0,1]\}$ are equicontinuous). Therefore, $D_{\Psi}(\bar{g_1},\bar{g_2})=\lim_k \| g_1-\Psi_{a_{n_k}}(g_2)\|_{\infty}=  \| g_1-\Psi_{b}(g_2)\|_{\infty}$. This shows that the infimum is reached and can be replaced by a $\min$. From here, we see that if $D_{\Psi}(\bar{f},\bar{g})=0$, then there exists $b\in [0,1]$ such that $f=\Psi_b(g)$, so that $\bar{f}=\bar{g}$ in $\QC$. Using this property, the triangular inequality for $D_{\Psi}$ follows: take $b_1,b_2$ such that
 \be
 D_{\Psi}(\bar{f_1},\bar{f_2})+D_{\Psi}(\bar{f_2},\bar{f_3})&=&\|f_1-\Psi_{b_1}(f_2)\|_{\infty}+\|f_2-\Psi_{b_2}(f_3)\|_{\infty} \\
 &=&\|f_1-\Psi_{b_1}(f_2)\|_{\infty}+\|\Psi_{b_1}(f_2)-\Psi_{b_2+b_1}(f_3)\|_{\infty}\\
 &\geq& \|f_1-\Psi_{b_2+b_1}(f_3)\|_{\infty}\geq \min_c  \|f_1-\Psi_{c}(f_3)\|_{\infty} = D_\Psi(\bar{f_1},\bar{f_3}),
 \ee
 from what we conclude. 
 
 Now, $\QC$ is Polish: just use the fact that the canonical projection from $C[0,1]$ to $\QC$ is 1-Lipschitz: $D_{\Psi}(\bar{f},\bar{g}) \leq \|f-g\|_{\infty}$. This allows seeing that the canonical projection of a countable dense subset of $C[0,1]$ is countable and dense in $\QC$. Now take a Cauchy sequence $(\bar{f_n})$ in $\QC$. Take any sequence $(b_n)$ in $[0,1]$, and an element $f_n$ in $\bar{f}_n$ for each $n$. We claim that the sequence $(\Psi_{b_n}(f_n))$ contains a converging subsequence in $C[0,1]$. By the Arzela-Ascoli theorem it suffices to check that this sequence is bounded and uniformly continuous. First, the sequence $\l(\|\Psi_{b_n}(f_n)\|_{\infty}\r)$ is bounded (because ${\sf Range}(f_n):=\max f_n - \min f_n$ is a class invariant: if along a subsequence, ${\sf Range}(f_n)\to +\infty$, then $(\bar{f_n})$ cannot be Cauchy, since $D_{\Psi}(f,g)\geq |{\sf Range}(f)-{\sf Range}(g)|$). \par
 In the same way, the following ``circular'' class invariant continuity modulus of $f$,
 \[\bar{w}_{\delta}(f):=\max \{ |f(x)-f(y)|, d_{\R/\Z}(x,y)\leq \delta\},\] considering in this formula $f$ as 1-periodic over $\R$,  can be compared to the standard modulus of continuity:
 \ben \label{eq:qdqds}
 w_\delta(f) \leq \bar{w}_\delta(f) \leq 2 w_\delta(f).
 \een
 Assume that $(\bar{f_n})$ is Cauchy in $\QC$: there exists an array $(b_{n,m},n, m\geq 0)$ such that $\|f_n-\Psi_{b_{n,m}}(f_m)\|_{\infty}\to 0$, for $n,m \geq N$ and $N \to +\infty$. Let $`e>0$ be fixed, and $N$ be large enough (and fixed) such that
 \ben\label{eq:borne}
 \sup_{n,m\geq N}\l\|f_n-\Psi_{b_{n,m}}(f_m)\r\|_{\infty}\leq `e.
\een
 Let also $\delta>0$ be small enough such that
 \[w_{\delta}(f_N) \leq `e.\] 
 By \eref{eq:borne}, $w_{\delta}\l(\Psi_{b_{n,m}}(f_m)\r)< 2`e$, and then $w_{\delta}\l(f_m\r)\leq 4`e$ for all $m\geq N$. Hence the sequence $(f_m)$ is bounded and equicontinuous, so that it is relatively compact (by the Arzela-Ascoli theorem). Hence, there exists a converging subsequence $(f_{n_k})$ in $C[0,1]$. Let $f$ be the limit of this sequence. From here, it is easy to conclude that $\bar{f_n}\to \bar{f}$ in $QC[0,1]$: by the triangular inequality, $D_{\Psi}(\bar{f_n},f)\leq D_{\Psi}(\bar{f_n},\bar{f}_{n_k})+D_{\Psi}(\bar{f_{n_k}},\bar{f})$, and then taken the converging subsequence for $(f_{n_k})$, we get that $\bar{f_{n_k}}$ converges to $\bar{f}$ in $\QC$ and then $({\bar f}_n)$ converges to $\bar{f}$ too.

\subsection{Proof of Proposition \ref{pro:cv}}
\label{sec:proof-of-prop-cv}
\noindent $(i)$ If $\by_n\dd \by$, by  Skhorohod's theorem, there exists a probability space $\w{\Omega}$ where some copies $\w{\by}_n \eqd \by_n$, $\w{\by} \eqd \by$, are defined, for which $\w{\by_n} \as \w{\by}$.
  By Arzela-Ascoli, a subset ${\cal K}$ of $C[0,1]$ is relatively compact iff
  $\sup_{g \in {\cal K}} \|g\|_{\infty} <+\infty$, and if for any $\delta$, $\sup_{g \in K} w_{\delta}(g)<+\infty$. Since for any $g \in C[0,1]^+$, $g(0)=g(1)=0$, then for any $a\in[0,1]$, $w_{\delta}(\Psi_a(g))\leq 2 w_{\delta}(g)$.
    It follows that if ${\cal K}$ is a subset of $C[0,1]^+$ relatively compact in $C[0,1]$, then ${\cal K}':=\bigcup_{a\in[0,1]} \Psi_a({\cal K})$ is also relatively compact. From the tightness of $(\by_n,n\geq 0)$, we can therefore deduce the tightness of $(\Conj(\by_n),n\geq 0)$. 

     Consider the sets $A=\min \argmin \w{\by}$ and  $a_n = \min \argmin \w{\by_n}$. First, $\|\w{\by_n}-\w{\by}\|_{\infty}\to 0 \imp \min \w{\by_n}\to \min\w{\by}$, and this entails that $d(a_n,A)\to 0$. Indeed, if for a subsequence $d(a_{n_m},A)\not\to0$, then, by compacity, a subsequence $a'_{n_{m}}$ (of this subsequence) would converge in $[0,1]$ to a point $x\notin A$ and at this point $\w{\by}_x\leq \min \w{\by}$, a contradiction. Hence, the accumulation points of $(a_n)$ belong to $A$. If a converging subsequence $a_{m_n}$ tends to some $a\in A$,  $\Conj(\ell_{m_n})\to \Psi_a(\ell)$ in $C[0,1]$.\\ 
$(ii)$  Use the same Skhorohod embedding as in $(i)$. Now, in $C[0,1]$, the map $g\mapsto \min \argmin g$ is not continuous, but if $g\in C[0,1]$ reaches its minimum only once, then $\|g_n-g\|_{\infty} \to 0 \imp \min \argmin g_n \to \min \argmin g$. Hence, on the space $\w{\Omega}$, $\min \argmin \w{\by}_n \as \min \argmin \w{\by}$,  and  since $\|\w{\by_n}-\w{\by}\|_{\infty}\as 0$,  from \eref{eq:V}, we conclude.\\
$(iii)$ This is a consequence of $(i)$: since the sequence $(\Conj({\by_n}),n\geq 1)$ is tight in $C[0,1]$, from each subsequence $(\Conj({\by_{n_m}}),m\geq 1)$ of this sequence, one can extract a weakly converging subsequence  $(\Conj({\by_{n_{m_k}}}),k\geq 1)$, and still by $(i)$ the accumulation point in $\QC$ is $\bar{\by}$. Using the Skhorohod embedding, we can find a probability space on which the copies of these random variables converge a.s.~(written with an extra $\star$):
 \[D_{\Psi}(\Conj(\by_{n_{m_k}}^\star), \Conj(\by^\star))\to0.\]
 Since $D_{\Psi}(\Conj(\by_{n_{m_k}}^\star), \Conj(\by^\star))=D_{\Psi}(\by_{n_{m_k}}^\star, \by^\star)$,  along this subsequence, $\overline{\by^\star_{n_{m_k}}}\as \overline{\by^\star}$ in $QC^0[0,1]$, from what we deduce that  $\overline{\by_{n_{m_k}}^\star}\dd \overline{\by^\star}$ in $QC^0[0,1]$. Since the limit coincides with  $\bar{\by}$ in distribution in $\QC$, its distributions does not depend on the extracted subsequence, $(iii)$ holds.\\
$(iv)$ This is a consequence of $(iii)$, since the class of $\Conj(\by_n)$ (resp.~$\Conj(\by)$) in $QC^0[0,1]$ is the same as that of $\by_n$ (resp.~$\by$). 
 
\subsection{Proof of Theorem \ref{theo:CVpointedSnakes}}
\label{sec:proof-of-cv-pointed-snakes}

The proof is done by induction. Before writing the proof, we need to state several propositions. The convergence for the case $ \D=1$, is a consequence of a result already known:
\begin{pro}\label{theo:bs} [Marckert \& Mokkadem \cite{MMsnake}, Janson \& Marckert \cite{janmarck05}] Let $\nu$ be a centered distribution having moments of order $4+`e$ for some $`e>0$, and variance $\sigma^2>0$. Consider $(\bT_n,{\bf L}_{{\bf T}_n})$ a branching random walk, constructed on a random tree $\bT_n$ picked uniformly in $\Tset_n$.  The following convergence in distribution holds in $C([0,1],\R^2)$ equipped with the topology of uniform convergence:
   \ben\label{eq:ehet}
  \l(\frac{C_{\bT_n}(2nx)}{\sqrt{2n}}, \frac{\bL_{\bT_n}(2nx)}{n^{1/4} 2^{1/4}\sigma}\r)_{x\in[0,1]} \to (\se,r),\een 
  where $\se$ is the normalized Brownian excursion and, conditionally on $\se$,  $r$ is distributed as a centered Gaussian process with covariance matrix
  $$\cov(r_x,r_y)= \W{\se}( x,y).$$
  \end{pro}
The convergence of the first marginal in Proposition \ref{theo:bs} is equivalent to the convergence of uniform planar trees to Aldous' continuum random tree:  if  $\bT_n$ is picked uniformly in $\Tset_n$, then
\ben\label{eq:grgze}
\frac{C_{\bT_n}(2n.)}{\sqrt{n}}\dd \sqrt{2}\,\se.
\een
This theorem can be found in Aldous \cite{Aldous2, mm01} as a particular case of the convergence of the contour process of Galton-Watson trees conditioned by their size $n$ (as $n\to+\infty$). \par
For the second marginal convergence, the main ingredient is the central limit theorem: for $2nx \in \mathbb{N}$, conditionally on $C_{2nx}$, $R_{2nx}$ is a sum of $C_{2nx}$ centered i.i.d.~r.v.~with variance $\sigma^2$; hence  
  \[\frac{R_{2nx}}{\sigma \sqrt{C_{2nx}}} = \frac{R_{2nx}}{\sigma (2n)^{1/4}}\frac{(2n)^{1/4}}{\sqrt{C_{2nx}}} \]
  is close to a normal random variable ${\cal N}(0,1)$. One therefore perceives, according to the convergence $\frac{(2n)^{1/4}}{\sqrt{C_{2nx}}}\dd 1/\sqrt{\se_x}$, the one-dimensional convergence of the second marginal, as stated in \eref{eq:ehet}. (The finite dimensional convergence can be proved using this argument, but the tightness needs additional work). 
  
Now, in order to complete the induction we need the three following propositions.
\begin{pro}\label{pro:uqn}
Let $(r(n))$ be a sequence such that $r(n)\to +\infty$ and $r(n)=o(n)$. Assume that $(\bL_n,n\geq 0)$ is a sequence of processes such that for every $n$, $\bL_n=\bL_n([0,n])$ takes its values in $\bbL_n$, and such that the normalized  and interpolated process
  $\Bell_n:=\l(\frac{\bL_n(nt)}{r(n)},0\leq t \leq 1\r)$ converges in distribution
  \[ \Bell_n \dd \Bell \textrm{ in } C[0,1],\]
where a.s., $\Bell$ has no period in the following sense: if $\Bell(a+x \mod 1)=\Bell(x)$ for all $x\in[0,1]$, then $a \in \Z$.
  In this case, letting 
   \[\left\{\begin{array}{ccl}
      \bC_n&:=&\Phi_n^{H \to C}(\Phi_n^{L\to H}(\bL_n))\\
      \bc_n(t)&=& \bC_n(2nt)/r(n), t\in[0,1]\\
      \bc'_n&=&\Conj(\Bell_n)
    \end{array}\right.\]
    we have 
    \ben
    \| \bc_n-\bc_n'\|_{\infty}\to 0 \textrm{ in probability, as }n\to+\infty.
    \een
  \end{pro}
This proposition does not imply the convergence of $\bc_n$ or $\bc'_n$; again, the convergence of $(\Bell_n)$ does not imply that of $(\Conj(\Bell_n))$.
    \begin{proof}  
      Write $\bH_n= \Phi_{n}^{L\to H}(\bL_n)$, $\bh_n(\cdot)=\bH_n(n\cdot)/r(n)$ and $\bC_n= \Phi_{n}^{H\to C}(\bH_n)$. From Proposition \ref{pro:cv}, the sequence $(\Conj(\Bell_n))$ is tight in $C[0,1]$. Consider a converging subsequence $(\Conj(\Bell_{n_k}))$, and let us observe that $\Conj(\Bell_{n_k})$ and $(\bh_{n_k})$ are asymptotically indistinguishable, in the sense that
 \[\|\bc_{n_k}'-\bh_{n_k}\|_\infty=\|\Conj(\Bell_{n_k})-\bh_{n_k}\|_\infty \to 0 \textrm{ in probability when } k\to +\infty.\]
  The reason is that 
  $\|\Conj(\Bell_n)-\bh_n\|_{\infty}\leq 2/r(n)$ since there is at most one abscissa discrepancy of one (normalized) step between the two constructions, and since in the definition of $\Phi_{n}^{L\to H}$ (Defi.~\ref{def:LtoH})  there is an additional  $+1$ (which after normalization becomes $1/r(n)$).  \\     
 Now, Theorem \ref{theo:ini-rec3} allows us to write $\|\bh_{n_k}-\bc_{n_k}\|\to 0$ (in proba.) and then to conclude.
\end{proof}

\begin{pro}\label{pro:uqn2} Under the hypotheses of Proposition \ref{pro:uqn},
  $(\Bell_n,\bar\bc_n)$ converges in $C[0,1]\times QC[0,1]$ to some limiting process $(\Bell, \bar{\bc})$, and  a.s., for any $\bc\in\bar\bc$, there exists a unique $a\in[0,1)$ such that
\[ {\bc(. -a \mod 1) = \Bell},\]
in other words, there is almost surely a unique shift sending $\Bell$ onto $\bc$.
\end{pro}
\begin{proof}The convergence of $\Bell_n \dd \Bell$ in $C^0[0,1]$ implies the convergence of $\bar{\Bell_n}  \dd \bar{\Bell}$ in $\QC$. Since for any $n$, $D_{\Psi}(\bar{\bc_n'},\bar{\Bell_n})=0$, we can deduce that $D_{\Psi}(\bar{\bc},\bar{\Bell})=0$. Since $\Bell$ has a.s.~no period, for $\bc$ fixed in $\bar{\bc}$, there exists a unique $a\in[0,1]$ such that $\bc=\Psi_a(\Bell)$. 
\end{proof}

\begin{theo}\label{theo:ini-rec}
  Let $(r(n))$ be a sequence such that $r(n)\to +\infty$. For any $n$, let $\mu_n$ be a distribution on $\bbT_{n}$. Consider a branching random walk with underlying tree $\bT_n$, a random tree with law $\mu_n$,  and spatial increment $\nu$ (see \eref{eq:nu}). Denote by $\bL_n$ the associated corner label process. Let $(\bc_n,\Bell_n)$ be the normalized versions of the contour and label processes defined by 
 \ben\label{eq:dsdsk}
 \bc_n(x)&=&\frac{C_{\bT_n}( 2nx)}{r(n)}, \textrm{ for }x \in [0,1]\\
 \Bell_n(x)&=& \frac{\bL_n(2nx)}{\sqrt{2r(n)/3}}, \textrm{ for }x\in[0,1].
    \een
    If   $\bc_n \dd \bc$ in $(C[0,1],\|.\|_\infty)$ and if the sequence $(\bc_n)$ is $q$-Hölder tight for some $q\in(0,1)$, then:\\
    \bls\, the pair $(\bc_n,\Bell_n)$ converges in distribution to $(\bc, \Bell)$ in $(C[0,1]^2,\|.\|_{\infty})$, where conditionally on $\bc$, $\Bell$ is a Gaussian process with covariance matrix
  \ben\label{eq:cov}
  \cov(\Bell(x),\Bell(y))=\W{\bc}(x,y);\een
  \bls\, the sequence $(\Bell_n)$ is  $q'$-Hölder  tight, for any $q'\in(0,q/2)$.\\
  \bls\, If $\bc$ is different from the zero process with probability 1, then a.s.~$\Bell$ has no period (in the sense of Proposition \ref{pro:uqn}).  
\end{theo}

\begin{proof}
We prove separately the convergence of the finite dimensional distributions, and
 the tightness.\\
{\bls}   Since $\bc_n\dd\bc$, by Skhorohod representation theorem, there is a probability space on which  some copies $\tilde{\bc}_n \eqd \bc_n$, $\tilde{\bc}\eqd \bc$ are defined, such that $\tilde{\bc}_n\as \tilde{\bc}$. Let us work on this space. Now, by the standard central limit theorem, one proves easily the convergence of the finite dimensional distributions of the sequence of processes $\Bell_n$ to those  of the Gaussian process  with covariance matrix  specified in \eref{eq:cov}. \\
{\bls}   To prove the property about the tightness, consider 
the following set $F_{M,q}=\{f \in C^+[0,1],  \hol_q(f) \leq  M\}$  for some $M>0$, and $q \in (0,1)$.
Then, as already explained below \eqref{bounddistance1}, for $f\in F_{M,q}$, the distance $D_f$ defined in \eqref{eq:Dg} satisfies
\ben
\label{eq:bound-dist}
D_{f}(s,t) \leq 2 M |t-s|^{q},~~ \textrm{ for any } (s,t)\in[0,1]^2.
\een
Let us assume that $(\bc_n)$ is $q$-Hölder tight, i.e. for any $`e>0$, there exists $M>0$, $N\geq 1$, such that, for any $n\geq N$, 
\[`P(\bc_n \in F_{M,q})\geq 1-`e.\]
Conditionally on $\bc_n \in F_{M,q}$,  for $x$ and $y$ such that $x\leq y$ and $2n x$ and $2ny$ are integers,
$\Bell_n(x)-\Bell_n(y)$ is a sum of $D_{C_{\bT_n}}(2nx,2n y)$ i.i.d.~r.v.~$(\bX_j,j\geq 1)$ with distribution $\nu$, divided by $\sqrt{2r(n)/3}$.
Since $\nu$ is centered and has all its moments, the Marcinkiewicz–Zygmund inequality gives 
\ben
\E(|\bX_1+\cdots+\bX_{p}|^{s}) \leq m(s)
{p}^{{s}/2},
\een
where $m(s)$ is a constant (independent from $p$). Therefore, conditionally on $C_{\bT_n}$, for $x,y$ such that $2nx$ and $2ny$ are integers\footnote{Between integer points, $\bL_n$ is linear, and it is folklore and easy to check, that the tightness can be proved by proving the moment condition only at these discretization points.}
\[\E\l(\l|\frac{\bL_n\bigl(2nx)-\bL_n(2ny\bigr)}{\sqrt{r(n)}} \r|^
{s} \Big| C_{\bT_n}\r)\leq m(
{s})~\l(\frac{D_{C_{\bT_n}}\bigl(2nx,2ny\bigr)}{r(n)}\r)^{
{s}/2} = m(s)\l(D_{\bc_n} \bigl(x,y\bigr)\r)^{s/2},  \]
by the normalization \eref{eq:dsdsk}. Using the fact that $(\bc_n)$ is $q$-Hölder tight and  \eqref{eq:bound-dist}, this gives
\[\E\l(\bigl|\sqrt{2/3}\l(\Bell_n(x)-\Bell_n(y)\r) \bigr|^s\Big|C_{\bT_n}, F_{M,q}\r)\leq 
 m'(s, M)~\l|x-y\r|^{ 
 sq/2},  \]
with $m'(s, M)=m(s)(2M)^{s/2}$. The standard moment criterium \eref{eq:sefdqd} allows  us to conclude.\\
\bls Use again the probability space of the first point, on which  some copies $\tilde{\bc}_n \eqd \bc_n$, $\tilde{\bc}\eqd \bc$ are defined, such that $\tilde{\bc}_n\as \tilde{\bc}$. Take any element $c\in C^+[0,1]$ different from the zero function. Conditionally on $\tilde{\bc} =c$, the tree $T_c$  possesses a  non-trivial branch, so that a.s.~$\Bell$ is not identically zero (its range contains that of the Brownian motion living on any non-trivial branch of $T_c$). Now, under this condition, if $\Bell$ is periodic, its period must be rational: otherwise, since $\Bell_0=0$, $\Bell$ would be identically 0 which is excluded (because $\Bell$ is a.s.~continuous, and if $q$ is a period, then $nq$ is too for $n\in \mathbb{Z}$, and $nq \mod 1$ is dense in $[0,1]$ for $q$ non-rational). 

  Since $\mathbb{Q}$ is countable, it suffices then to show that for a rational number $q \in (0,1)$, with probability 1, $q$ is not a period of $\Bell$.  
  We start to show that there exists $x\in [0,1]$ so that $D_c(x,x+q)\neq 0$. Indeed,  $D_c(x,x+q)=0$ implies that $x$ and $x+q$ are corners of the same node in $T_c$; since nodes are non-crossing, $D_c(x,x+q)=0$ for every $x$, implies that $[0,1]$ is the set of corners of a unique node, so that $c$ is identically 0, which is excluded. 
Now, take $x$ such that $D_c(x,x+q)\neq 0$: for this $x$, since $\Bell$ is the label process of $c$, $\Bell(x)-\Bell(x+q)\sim {\cal N}(0,D_c(x,x+q))$, and then, a.s. $\Bell(x)\neq \Bell(x+q)$ and then $q$ is not a period of $\Bell$.
\end{proof} 

We go on with a kind of converse of Proposition \ref{pro:cv}:
\begin{lem}\label{lem:qsgdqs} Let $\bar{c}, \bar{c_1},\bar{c_2},\cdots$ be a deterministic sequence in $\QC$ such that $\bar{c_n}\to \bar{c}$ in $(\QC,D_{\Psi})$. Consider for each $n$ an element $c_n$ taken in $\bar{c_n}$, and such that $c_n\in C^+[0,1]$. The sequence $(c_n)$ is relatively compact in $C[0,1]$: every accumulation  point of this sequence belongs to $\PTC(\bar{c})$.
\end{lem}
\begin{proof} As in the proof of  Proposition \ref{pro:top} (section \ref{sec:pppt}), we start by observing that the sequence $(c_n)$ is relatively compact: the reason is that $\bar{c_n}\to \bar{c}$ in $\QC$ is equivalent to the existence of a sequence $(a_n)$ such that $\|c_n(a_n+. \mod 1)-c\|_{\infty}\to 0$; from there follows the fact that $(\|c_n\|_{\infty})$ is bounded and $(c_n)$ equicontinuous. The sequence $(c_n)$ is therefore relatively compact by Arzela-Ascoli, and owns some accumulation points (that are clearly elements of $C^+[0,1]$). Consider an accumulation point $c^{\star}\in C^+[0,1]$. Since along a subsequence $\|c_{n_k}-c^{\star}\|_{\infty}\to 0$, clearly $D_\Psi\l(\bar{c_{n_k}},\bar{c^\star}\r)\to 0$ and since by hypothesis $D_\Psi\l(\bar{c_{n}},\bar{c}\r)\to 0$, we deduce that $D_{\Psi}(\bar{c^\star},\bar{c})=0$ so that $\bar{c}^{\star}$ and $\bar{c}=\QC$.
  \end{proof}
\begin{lem}\label{lem:qdgr} Let $\bC_n$ be a uniform Dyck path taken in ${\sf Dyck}_{2n}$, and $\bc_n(\cdot)=\frac{\bC_n(2n\cdot)}{\sqrt{2n}}$. The sequence $(\bc_n)$ is $q$-H\"older tight for any $q<1/2$.
  \end{lem}
As stated in \eref{eq:grgze},  $\bc_n\dd \se$.  
\begin{proof} This property is folklore: it can be transferred from the classical simple random walk (with increment $+1$ or $-1$) which owns the same property (with the same rescaling), fact that follows a simple application of the moment condition. To transfer this property from the random walk to the Dyck paths there is a two steps argument (\cite[Proof of Lemma 1]{janmarck05}).\\
  -- There is a 1 to $2n+1$ map which sends Dyck paths (at which an additional step $-1$ is appended) to bridges of size $2n+1$, which are paths with steps $+1$ and $-1$ with length $2n+1$ ending at $-1$: the $2n+1$ bridges associated to a (single) Dyck path are obtained by conjugations at one of the $2n+1$ abscissa from 0 to $2n$. These maps multiply $q$-H\"older exponents by at most 2, so that, it suffices to transfer the tightness property from random walks to bridges.\\
  -- Now, bridges are invariant by translation, so that it suffices to prove that their first half are $q$-H\"older tight for any $q<1/2$.
  Consider then a simple random walk $(S_0,\cdots,S_{2n+1})$, and $A$ be an event $(S_0,\cdots,S_{n+1})$ measurable (which depends on the first ``half'' of the trajectory).\\
  {\bf Claim:} Up to a multiplicative constant, the probability of $A$ under the bridge condition is bounded by the probability of $A$ for the random walk:
  \ben\label{eq:AS} \P(A~|~S_{2n+1}=-1)\leq  c\,\P(A)\een
  for a universal finite constant $c$ (independent from $n$ and from $A$).
  
Taking for a moment this claim as granted, the fact that the rescaled random walk is $q$-H\"older tight on $[0,1/2]$ (for any $q\in(0,1/2)$) implies that it is also the case for the rescaled bridge (taking some events of the type $A=\{|S_{ns}-S_{nt}|/\sqrt{n}\leq C|s-t|^q, s,t\leq n+1\})$ to use the claim). To prove the claim, write
  \ben\label{eq:ethjjta}
  \frac{\P(A,S_{2n+1}=-1)}{\P(S_{2n+1}=-1)}&=&\sum_{k=-n-1}^{n+1}  \frac{\P(A,S_{2n+1}=-1,S_{n+1}=k)}{\P(S_{2n+1}=-1)}\\
 \label{eq:ethjjtb} &=&\sum_{k=-n-1}^{n+1}  \frac{\P(S_{2n+1}=-1|S_{n+1}=k,A)}{\P(S_{2n+1}=-1)}  \P(A,S_{n+1}=k)
 \een
 (sum which must be restricted to the $k$ for which $\P(A,S_{n+1}=k)>0$),
  and by the Markov property of random walk,
  \[\sup_n\sup_k\frac{\P(S_{2n+1}=-1|S_{n+1}=k,A)}{\P(S_{2n+1}=-1)} \leq c:=\sup_n\sup_k \frac{\P(S_{2n+1}=-1|S_{n+1}=k)}{\P(S_{2n+1}=-1)}=\sup_n\sup_k\frac{\P(S_{n}=-k-1)}{\P(S_{2n+1}=-1)},
    \]
    since $\sup_k \P(S_{n}=-k-1)=\P(S_n\in\{0,1\})$ (depending on the parity of $n$), using Stirling formula, one sees that $c<+\infty$. Using this bound in \eref{eq:ethjjtb} allows us to get \eref{eq:AS}.
\end{proof}

\color{black}

\subsection*{End of proof of Theorem \ref{theo:CVpointedSnakes}}

In the case $\D=1$, the convergence in $C[0,1]^2$ stated in Proposition \ref{theo:bs} can be restated as $(\bc_n^{(1)},\Bell_n^{(1)})\dd(\bc^{(1)}, \Bell^{(1)})$ in $C[0,1]^2$: it implies the convergence in $\PS{}^{,1}$ of $\pi_{\bullet}(\bc_n^{(1)},\Bell_n^{(1)})\dd\pi_{\bullet}(\bc^{(1)}, \Bell^{(1)})$.   Lemma \ref{lem:qdgr}, allows us to see that $(\bc_n^{(1)})$ is $q$-H\"older tight for $q\in(0,1/2)$; Theorem \ref{theo:ini-rec} allows us to deduce that $(\Bell_n^{(1)})$ is $q$-H\"older tight for $q\in(0,1/4)$, and since  $\bc^{(1)}\eqd \sqrt{2}\se$ is a.s. different from the zero process, that $\Bell^{(1)}$ has a.s.~no period.
Propositions \ref{pro:cv} and \ref{pro:uqn} allow deducing from the convergence $\Bell_n^{(1)}\to \Bell^{(1)}$, the fact that $\bc_{n}^{(2)}\to \bc^{(2)}:=\Conj(\Bell^{(1)})$ in $\QC$.\par

Now, the proof is basically a proof by induction: by the incremental nature of the construction of the iterated discrete snakes, we have for any $j\geq 2$:
\[{\sf Law}\l(\l(\bc_n^{(j)},\Bell_n^{(j)}\r)~\Big|~\l(\l(\bc_n^{(m)},\Bell_n^{(m)}\r), m\leq j-1\r)\r)
= {\sf Law}\l(\l(\bc_n^{(j)},\Bell_n^{(j)}\r)~\Big|~\l(\bc_n^{(j-1)},\Bell_n^{(j-1)}\r)\r).\]
\par
    In order to prove the convergence of $ \pi_{\bullet}\bigl(\bc_n^{(j)},\Bell_n^{(j)}\bigr)$ knowing the convergence of $\bigl( \pi_{\bullet}\bigl(\bc_n^{(m)},\Bell_n^{(m)}\bigr),m\leq j-1\bigr)$  the simplest method consists in working on a space on which some copies of these processes are defined, that converge almost surely (whose existence is guaranteed by Skhorohod representation theorem). 
    On this space, since $\bar{\Bell_n^{(j-1)}}\to \bar{\Bell^{(j-1)}}$ a.s.~(in $\QC$), by the same argument as above (mainly Propositions \ref{pro:cv} and \ref{pro:uqn}), $\bar{\bc_n^{(j)}}$ converges to some process, $\bar{\bc^{(j)}}$ in $\QC$.
We suppose also by induction that  $(\bc_n^{(m)})$ is $q$-H\"older tight for $q\in(0,1/2^m)$  and $(\Bell_n^{(m)})$ is  $q$-H\"older tight for $q\in(0,1/2^{m+1})$ for $m\leq j-1$.

First, the fact that $(\bc_n^{(j)})$ is $q$-H\"older tight for any $q\in(0,1/2^j)$ is a consequence of the preservation of this property by conjugation, and is then inherited from this property of $(\Bell_n^{(j-1)})$. 
 To discuss the convergence of the label process of the normalized branching random walk having $\bar{\bc_n^{(j)}}$ as an underlying tree, let us 
 come back to the world of rooted trees: by Lemma \ref{lem:qsgdqs}, choose an element $\bc_n\in C^{+}[0,1] \cap \bar{\bc_n^{(j)}}$ (for example take $\bc_n=\bc_n^{(j)}$)
 for each $n$, build a rooted tree with this contour, and on this tree a branching random walk as explained in \eref{eq:ds}. After that, build the corresponding normalized tour of snake $(\bc_n,\Bell_n)$ associated to this rooted tree with the normalization specified in Theorem \ref{theo:CVpointedSnakes} (recall Section \ref{sec:BRWPS} for the effect of choosing $\bc_n$ or another tree in the tree class of $\bar{\bc_n^{(j)}}$). From Lemma \ref{lem:qsgdqs}, $(\bc_{n})$ has all its accumulation points in $\PTC\bigl(\bar{\bc^{(j)}}\bigr)$. Take a converging subsequence $(\bc_{n_k})$ in $(C[0,1],\|.\|_\infty)$, and let $\bc^{\star}\in C^+[0,1]$  be its limit.
 By Theorem \ref{theo:ini-rec}, we get $(\bc_{n_k},\Bell_{n_k})\dd (\bc^\star,\Bell)$,
where $\Bell$ is a Gaussian process with covariance matrix
$\cov(\Bell(x),\Bell(y))=\W{\bc^\star}(x,y)$. We get also by this same theorem, that $(\Bell_{n})$ is $q$-H\"older tight for any $q\in(0,1/2^{j+1})$ (for a bound on the $q$-H\"older exponent independent from the subsequence, since the $q$-H\"older exponents of the different elements $\bc_n \in  \bar{\bc_n^{(j)}}$ have their ratio bounded by $2$). 
Hence, taking into account that $\pi^{\bullet}(\bc^{\star},\Bell)$ is a random variable in $\PS$ having a distribution which does not depend on the element $\bc^{\star} \in \bar{\bc_n^{(j)}}$, we conclude.

\section{Combinatorial aspects of maps and iterated  discrete feuilletages}
\label{sec:GraphSection}\setcounter{equation}{0}
The aim of this section is twofold. We first come back on some known aspects of combinatorial maps. There are several points of view on these objects, some being suitable for combinatorial purposes, some for topological considerations, some allow making computations as exemplified in the paper, and eventually, allow producing asymptotic results and limit objects. After presenting these aspects shortly and providing references, we will stress a bit on quadrangulations encoded by labeled trees, since they are at the origin of the Brownian map (the second random  feuilletage $\RR{2}$ as presented in Section~\ref{seq:IBSIM}), and since we will generalize this encoding as a pair of trees, or equivalently as a pair of non-crossing partitions on an ordered set, for iterated discrete feuilletages.  

We then review the iterative construction of the discrete feuilletages in the combinatorial map picture, compute the critical exponent given by their asymptotic enumeration, and provide an encoding using $\D$ nested non-crossing partitions defined on an initial ordered  set. We briefly review other notions of generalized maps, and comment on their use in finding families of generalized maps that could lead to new asymptotic regimes.

\subsection{Different points of view on maps}

 Some references on maps include the books of Mohar  \cite{GraphsOnSurfaces}, of Landau \& Zvonkin \cite{LZ}, of Goulden and Jackson  \cite{GoulJack}, or the lecture notes of G.~Miermont and Le Gall (for instance \cite{MJFLG}).

\subsubsection{Maps as equivalence classes of graph drawings up to homeomorphisms}

Given a surface $S_g$ of genus $g\geq 0$, a ``pre-map'' is a drawing ${\cal G}$ of a connected graph on this surface, which satisfies the following properties:\\
-- the drawings of edges are injective and homeomorphic to segments,\\
-- the drawings of different edges $(e_1,e_2)$ and $(e_3,e_4)$ do not intersect, except  at their extremities if they correspond to the same vertices,\\
-- the complement of the drawing $S_g \setminus {\cal G}$ is cellular, that is, it is the union of some connected components (the so-called ``faces'') that are homeomorphic to disks.\\
Maps are equivalence classes of ``pre-maps'' for the following equivalence relation: two ``pre-maps'' are equivalent if one can be sent to the other by a direct homeomorphism of the surface $S_g$. Direct homeomorphisms preserve the genera of surfaces, as well as cyclic orders of edges of pre-maps around vertices. 
In order to avoid any non-trivial automorphisms, some additional features/markings are often considered: labeling of vertices, or distinction of some vertices (pointed maps) or half-edges (rooted maps), or both.
 \begin{figure}[h!]
    \centerline{\includegraphics{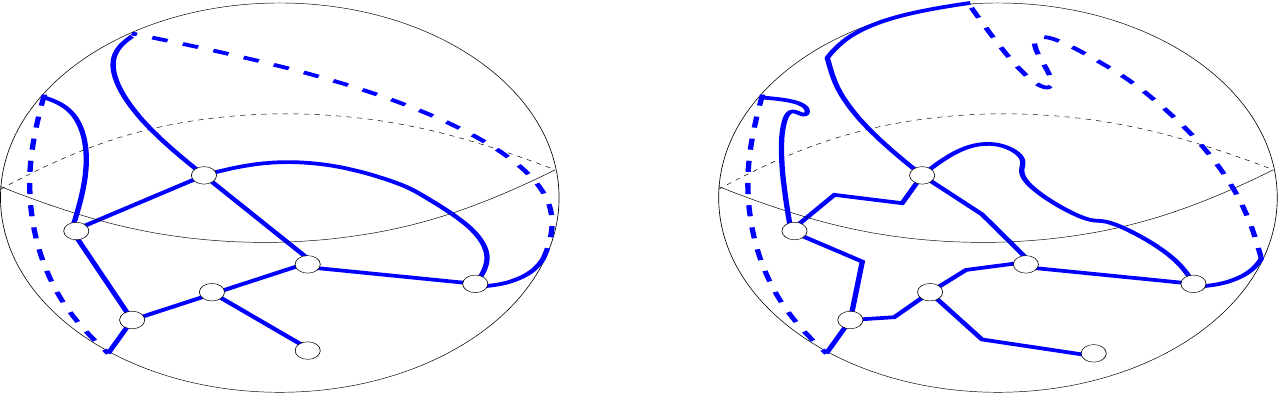}}
    \caption{\label{fig:PM}Two representations of the same planar map.}
  \end{figure}

\vspace{-0.7cm}
\subsubsection{Maps as gluings of polygons}

The preceding description of maps as graph drawings considered up to homeomorphisms produces a cellular decomposition of the surface ${\cal G}$. Topologically, each face $f$ is surrounded by a finite cyclic sequence of edges (or equivalently vertices), and then topologically, each face is a polygon. Hence, the surface ${\cal G}$ may be seen as a finite set of polygons that have been identified: each of the edges of every polygons has been identified with another unique edge of some polygon (which can be the same).

Conversely, if one has in hand a finite set of polygons with an even  total number of edges, and if one identifies the edges two-by-two in such a way to construct a connected surface, then naturally this surface comes with a notion of faces (the polygons), together with the ``drawing'' of the polygon boundaries, which are sequences of edges by definition; this produces a map, but the obtained surface can be non-orientable, and can have any non-negative genus.

\subsubsection{Maps as constrained permutations}
\label{sec:maps-as-perm}

In fact, maps as graph drawings up to homeomorphisms are combinatorial objects: there exists a finite number of maps with $n$ edges. One way to see this is to encode maps using permutations. Assume that a given map has $n$ edges, and label each half-edge by a number -- called an index -- from $0$ to $2n-1$.
The set of edges is then naturally encoded by a set $\alpha$ of pairs of half-edges, and each vertex by the cycle of edges around it (with a precise cyclic order): the set of vertices $\sigma$ -- as well as $\alpha$ can be seen as determining some permutations through their cycles. The main point here is that these permutations fully characterize the map.
Let us be more precise on the constraints inherited by these permutations.

\begin{defi}
\label{def:Maps} 
A connected indexed map\footnote{The common denomination would be {\it labeled map}, but we chose {\it indexed} to avoid confusion with the labeling of the trees and the label processes in the rest of the text.} with $n$ edges, is a pair $\cM=(\sigma, \alpha)$, where 
\begin{itemize}
\item $\sigma$ is a permutation on $\llbracket 0, 2n-1 \rrbracket$,
\item $\alpha$ is a fixed-point free involution on $\llbracket 0, 2n-1 \rrbracket$,
\item the group $<\sigma,\alpha>$ generated by  $\sigma$ and $\alpha$ acts transitively on $\llbracket 0, 2n -1\rrbracket$.
\end{itemize}
\end{defi}
The transitivity condition stands for the connectivity of the map. {\it Vertices} (resp.~{\it edges}, resp.~{\it faces}) correspond to the cycles in the unique decomposition of the permutations of $\sigma$ (resp.~$\alpha$, resp.~$\sigma\circ\alpha$) in disjoint cycles.  The elements of $\llbracket 0, 2n-1 \rrbracket$ are therefore seen as half-edges.

The number $V$, $n$ and $F$ of vertices, edges, and faces of a map $(\sigma, \alpha)$  satisfy Euler's formula:
\be
\label{eq:Euler}
V -n + F = 2 - 2g,
\ee
where $g$ is a non-negative integer, called the genus of a map (the genus of the surface on which the graph is drawn). Planar maps are maps with genus 0.

The {\it distance} in a map is the length of the shortest path between two vertices.

\
A map is said to be {\it pointed} if a particular vertex is distinguished.
Two consecutive half-edges $j=\sigma(i)$ and $i$ define a {\it corner} of the map.  
A map is said to be {\it rooted} if a particular corner is distinguished. It is equivalent to provide an orientation to a given edge, say from $i$ to $\alpha(i)$ (i.e.~to specify the first element of one of the disjoint pairs in $\alpha$). We call root vertex (or simply root) the vertex incident to the root corner.

Combinatorial maps are usually considered up to reordering of the half-edges. In this case, an unindexed map is an equivalence class $\{(\rho \circ \sigma\circ \rho^{-1}, \rho \circ \alpha\circ \rho^{-1})\mid \rho \in \cS_{2n}\}$.  
In this paper, we generally consider rooted unindexed maps\footnote{In the case of trees however, we rather consider rooted trees together with a canonical indexing of the half-edges, as detailed below.} for which the simultaneous conjugation of $ \sigma$ and $ \alpha$ by any $\rho \in \cS_{2n}$ leads to a different indexed combinatorial map, so that there are $(2n)!$ indexed rooted maps 
 corresponding to  the same unindexed map. 

\paragraph{Bipartite maps.}

A bipartite map is a map for which the vertices are colored in black and white, so that the edges only link black and white vertices. A bipartite map with $n$ edges is defined  by a pair of permutations $(\alpha_\circ, \alpha_\bullet)$ on $\llbracket1,n\rrbracket$. The elements of $\llbracket1,n\rrbracket$ correspond to the edges, and the cycles of $\alpha_\circ$ and  $\alpha_\bullet$ respectively define the white and black vertices. The faces correspond to the cycles of $\alpha_\circ\circ\alpha_\bullet$. To recover a map with the usual Definition \ref{def:Maps}, one has to take a black copy and a white copy of the edge-set, and define a canonical involution $\sigma$ which for $i \in \llbracket1,n\rrbracket$ maps $i_\circ$ to  $i_\bullet$.

\paragraph{Trees.}

A rooted tree with $n+1$ vertices is a rooted (connected) map with $n$ edges. From \eqref{eq:Euler}, we see that this imposes $F=1$ and $g=0$: the map has a single face and is planar. The face $\sigma\circ\alpha$ is a cycle of length $2n$ which organizes all the corners: using the root as a starting point (indexed 0), the corners can be re-indexed from 0 to $2n-1$: the ``counterclockwise corner sequence" of the tree. 
Note that this is not the convention used for processes in Sec.~\ref{sec:ISIMCO}, in which the corners of the trees are labeled from 0 to $2n-1$ going around the tree {\it clockwise}.  This should be kept in mind when comparing the various exemples. 
Apart from that, one then recovers the contour and height sequences of the tree, as defined in Sec.~\ref{sec:ISIMCO}. 

Labelings of trees have been defined in Def. \ref{def:LT}. A labeling provides an integer to every vertex (its label), such that the label of the root vertex is 0, and the labels of the extremities of any edge differ by at most 1.
The set of labeled rooted trees $(T,\ell)$ with $n$ edges will be denoted by $\LT n$,
and we define
\begin{equation} 
\label{eq:Min-Label}
m(T,\ell)=\min_{v\in\cV(T)} \ell(v).
\end{equation} 
\paragraph{Non-crossing partitions and trees.} Another encoding of a rooted planar tree with $n$ edges, which we will use later in this section, is as a pair $(C,\sigma),$ where $C$ is an ordered set $C=\{1,\ldots, 2n\}$ (the corner sequence of the tree), and $\sigma$ is a non-crossing permutation on this set, that is, a permutation for which the disjoint cycles respect the cyclic ordering of $C$ and have supports which do not cross for this ordering. 
\begin{defi} 
We say that a permutation $\sigma$ on a totally ordered set $C$ respects the ordering of $C$ if for each of its cycles $c$, there are only two consecutive elements $a<b$ in $c$ for which $\sigma(a)>\sigma(b)$.
\end{defi}

\begin{defi} 
A partition $C=\sqcup_{i} V_i$ of a finite totally ordered set $C$ is said to be non-crossing if there are no elements $p_1<q_1<p_2<q_2$ such that $p_1, p_2\in V_i$ and $q_1, q_2\in V_j$ with $i\neq j$. We say that a permutation $\sigma$ on $C$ is non-crossing if the partition it induces is non-crossing, and if in addition it respects the ordering of $C$.
\end{defi}

In addition, we require that the  
{\it Kreweras complement} \cite{KREWERAS1972333, nica_speicher_2006} of $\sigma$ on $C$ is a (non-crossing) {\it matching}, where a matching is a partition in pairs, and:
\begin{defi} 
Consider the totally ordered set $C=\{1<2<\ldots<N\}$ and make a copy $\bar C=\{\bar1<\bar 2<\ldots<\bar N\}$ of this set, so that $C\sqcup \bar C$ is ordered as $\{1<\bar 1< 2<\bar 2\ldots<N<\bar N\}$. If $\pi$ is a non-crossing partition on $C$, we define its Kreweras complement $\bar \pi$ as the  
maximal non-crossing partition on $\bar C$ for the inclusion order such that $\pi \cup \bar \pi $ is a non-crossing partition on $C\sqcup \bar C$, where   
the inclusion order is defined as  $C=\sqcup_{i} V_i \leq C'=\sqcup_{i} V'_i $ 
 if each $V'_i$ is the union of one or several $V_j$. 
\end{defi}
This definition is better understood on an exemple, see Fig.~\ref{fig:Kreweras-Dual}.
\begin{figure}[!h]
\centering
  \includegraphics[scale=0.8]{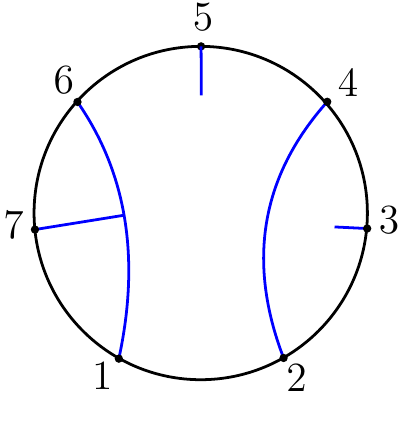}
  \hspace{1cm}
   \includegraphics[scale=0.8]{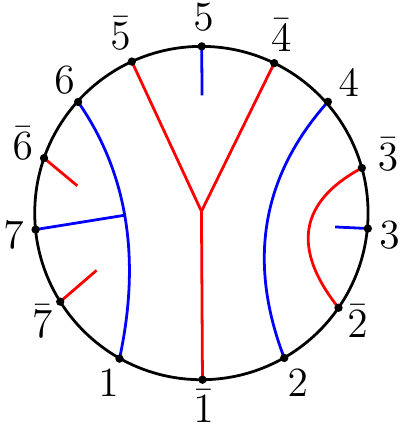}
    \hspace{1cm}
   \includegraphics[scale=0.8]{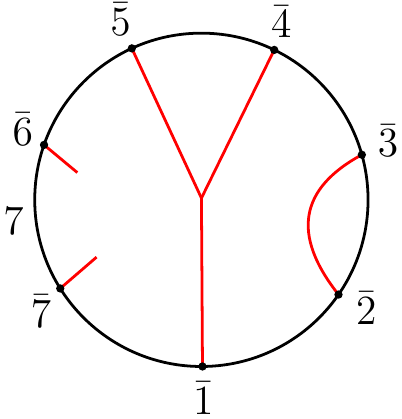}
\caption{\label{fig:Kreweras-Dual}A partition $\pi$ on an ordered set $C$ (left), and its Kreweras complement $\bar \pi$ on $\bar C$ (right), and the non-crossing partition $\pi \cup \bar \pi $ on $C\sqcup \bar C$ (middle). }
\end{figure}

The rooted planar tree  is then obtained from  $(C,\sigma)$  by drawing the ordered set $C$ on a circle and the cycles of $\sigma$ as shaded regions, adding vertices in the shaded regions, and edges between two vertices whenever the corresponding shaded regions face each other, as shown in Fig.~\ref{fig:Dual-nncross-tree}\footnote{Note that the information in the permutation is redundant: a tree is always bipartite, and one can delete the regions corresponding to e.g.~the white vertices without losing information. There are two bijections, one between planar trees with $n$ edges and non-crossing partitions on $n$ ordered elements, and one between between planar trees with $n$ edges and non-crossing partitions on $2n$ ordered elements, such that the Kreweras complement is a matching. One goes from the first description to the second one by completing the non-crossing partition on $n$ elements by its Kreweras complement $C \rightarrow C\sqcup \bar C$.  In the context of this paper, the permutations we obtain naturally have matchings as Kreweras complements. }. 
\begin{figure}[!h]
\centering
  \includegraphics[scale=1.1]{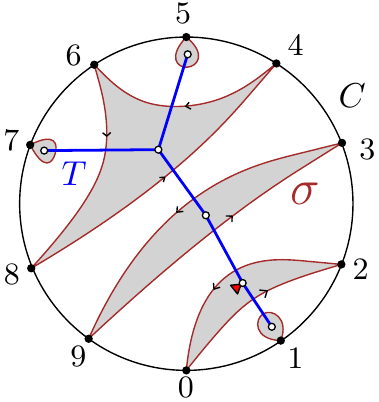}
\caption{\label{fig:Dual-nncross-tree}Encoding of a rooted planar tree by its counterclockwise corner sequence, together with a non-crossing permutation, whose Kreweras complement is a matching. }
\end{figure}

\subsection{Planar quadrangulations and labeled trees
}
\label{sec:CVS-bij}

\paragraph{The Cori-Vauquelin-Schaeffer bijection.}
A quadrangulation is a map whose faces all have length four. We denote by  $\Qset n$ the set of rooted and pointed quadrangulations with $n$ faces. 

\begin{theo}[Cori-Vauquelin 1981 \cite{CV}, Schaeffer 1998 \cite{Schae}] There is a 1-to-2 mapping between rooted and pointed quadrangulations with $n$ faces in $\Qset n$, and labeled trees with $n$ edges in  $\LT n$.

Furthermore, if the quadrangulation $Q$ is the image of $(T,\ell)$,  then the set of nodes of $T$ is sent bijectively onto the set of nodes of $Q$ deprived from the 
pointed vertex; if the node $u\in T$ is sent onto $v$, then, up to a global translation, $\ell(u)$ coincides with the distance of $v$ to the  pointed vertex in $Q$:
 $$\label{eq:qfq}  d_Q(\nu, w) = \ell(w) - m(T,\ell) + 1,$$ where $m(T,\ell)$ is the minimum of $\ell$ (defined in \eqref{eq:Min-Label}). 
\end{theo}

\begin{figure}[!h]
\centering
\raisebox{+1cm}{
 \includegraphics[scale=1]{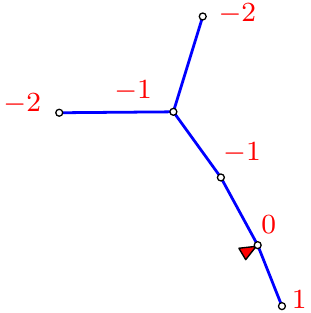}}
 \hspace{1.2cm}
 \includegraphics[scale=1]{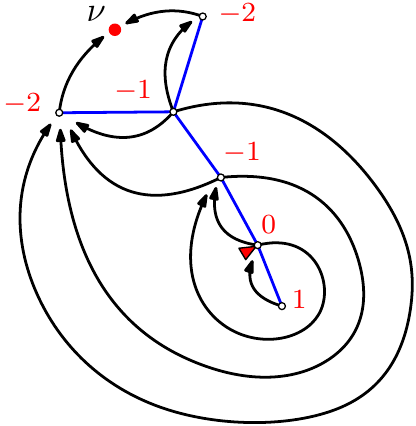}
 \hspace{1.2cm}
  \includegraphics[scale=1]{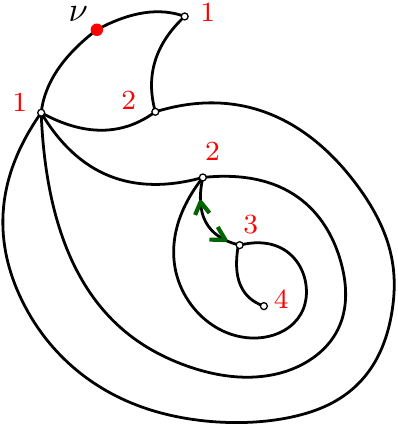}
\caption{\label{fig:Bij-Sc}Building two rooted pointed planar quadrangulations from a rooted labeled tree. They differ due to the orientation of the root-edge, that depends on the value of the parameter $\eta$.}
\end{figure}

Let us detail this mapping $\LT n \times \{0,1\} \leftrightarrow \Qset n$.\\

\noindent \emph{From  $\LT n \times \{0,1\}$ to $\leftrightarrow \Qset n$:} Consider a labeled rooted tree $(T,\ell)$,
and a parameter $\eta$  in $ \{0,1\}$. 
As usual, the depth traversal function $c_T$ of $T$ associates with each integer $k$ in  $\llbracket 0, 2n-1\rrbracket$, the node $c_T(k)$ (and again $c_T(k)$ can be seen as a corner, recall \eref{eq:corner}). Note that here however, we use the {\it counterclockwise}
contour sequence of $T$ to label the corners from 0 to $2n-1$. Consider also $L(k)=\ell(c_T(k))$, the label of the $k$th corner of $T$, for $k$ in  $\llbracket 0, 2n-1\rrbracket$. 
The description of the bijection is as follows: first draw  in {\bf red} a special vertex $\nu$ in the plane, and draw in {\bf blue} the tree $T$ in the plane (so that the drawing avoids $\nu$).\\
% --

Then, treat successively all the corners $(c_T(k),0\leq k <2n-1)$ (starting from the root corner $c_T(0)$) as follows:  
to each corner $c_T(k)$,
\begin{enumerate}[--]
\item if $L(k)=m(T,\ell)$ (see \eqref{eq:Min-Label}), add a {\bf black} edge between $c_T(k)$ and the special vertex $\nu$, 
\item if not, add a {\bf black} edge between $c_T(k)$ and the {\it preceding} 
corner with label $L(k)-1$. 
\end{enumerate}
At each step, this is done so that the {\bf black} map remains planar. There is always a unique corner on either $\nu$ or the preceding vertex carrying the label $L(k)-1$ such that it is possible. 

We now remove the tree $T$, that is the {\bf blue edges}, keeping only the vertices and all the {\bf black} edges. It can be proved that the resulting map is a connected planar quadrangulation. The map is pointed at $\nu$, and rooted by choosing an orientation for the edge that had been added at the very first step, in one of the two following ways:  if  $\eta=0$,  this edge is oriented {\it from} the corner indexed 0, and if $\eta=1$, this edge is oriented {\it towards} the corner indexed 0. 

\

\noindent \emph{From  $\Qset n$ to $\LT n \times \{0,1\}$:} We start by drawing in {\bf black} a connected, rooted planar quadrangulation, pointed at a vertex $\nu$. All the vertices are then labeled with their graph distance to $\nu$, and this labeling  is denoted by $d_\nu$. Because the map is planar and all its faces have even length, it is also bipartite. Consider a geodesic between some vertex and the root. Along this path, the vertices are alternatively black and white. All the vertices at a given distance from the root thus have the same color. Therefore, two vertices linked by an edge cannot have the same label, as their colors differ. 
As two such vertices have labels differing by at most 1, there are only two possibilities for the label patterns of the  vertices around a face, depicted in Fig.~\ref{fig:Faces-Bij}. 
\begin{figure}[!h]
\centering
 \includegraphics[scale=0.6]{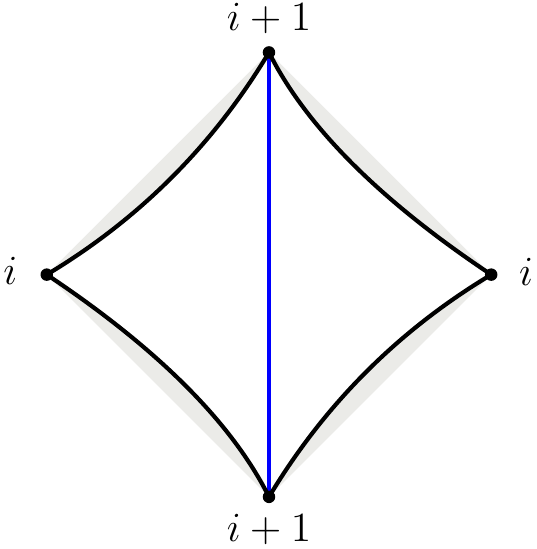}
 \hspace{1.5cm}
 \includegraphics[scale=0.6]{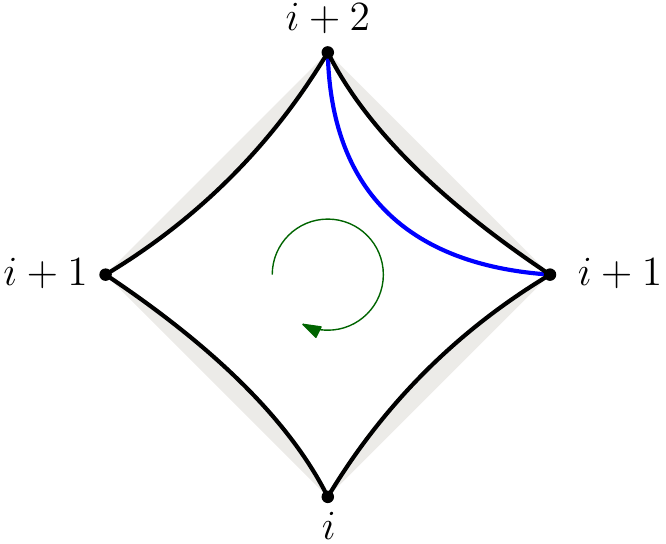}
\caption{\label{fig:Faces-Bij}The two possible patterns for the labels around a face of a planar quadrangulation. The blue edges are added when constructing a rooted tree from a quadrangulation. 
}
\end{figure}

We add one {\bf blue} edge per face of the quadrangulation as done in Fig.~\ref{fig:Faces-Bij}: between the corners incident to the vertices of higher labels for the case on the left of Fig.~\ref{fig:Faces-Bij}, and between the corner incident to the vertex of highest label and the corner that follows around the face, clockwise (if the face is the external face however, we take the following corner counterclockwise instead\footnote{This is because we represent the map on the plane, while it should be understood as drawn on the sphere.}). 
The labeled tree associated with the map is the {\bf blue} tree, and it is rooted and labeled as follows (the fact that it is a tree is not obvious): \\
-- the root edge in the quadrangulation was linking two corners of the {\bf blue} tree; we  choose the corner of higher label 
  as the root $c_0$ of the tree $T$, and the parameter $\eta$ takes the value 0 if that corner is at the origin of the root-edge of the quadrangulation, and 1 otherwise;\\
-- let $d_0$ denote the distance from $c_0$ to $\nu$ in the quadrangulation $Q$. For $w\in\cV(T)$, we define 
 $$\ell(w)=d_Q(\nu, w) - d_0,
 $$
where $d_Q(\nu, w)$ is the distance  from the vertex $w$ to $\nu$ in the quadrangulation.   
From the construction, $\ell$~defines a labeling on the rooted tree $T$, i.e. the label of the root vertex is 0, and the labels of two edges linked by an edge differ by at most 1.

\subsection{Edge disconnections in maps and tree-decompositions}

Let us consider a quadrangulation $Q\in \Qset n$  with $n$ faces, and denote by $\big(\t_n^{(1)},L_n^{(1)}\big)$ the corresponding labeled rooted tree in $\LT n$. We can draw the tree $\tt n 1$ on $Q$ by adding an edge in each one of its faces, as in Fig.~\ref{fig:Faces-Bij}.  The corresponding map, which contains both $\tt n 1$ and $Q$ as submaps, is denoted by $\tilde Q$.

\paragraph{The trees $\tt n 2$ and $\ttau_n^{(2)}$.}
Following the construction of \cite{MM}, another tree $\ttau_n^{(2)}$ with
$2n$ edges  is extracted from $\tilde Q$, by following the contour of $\tt n 1$, and ``ungluing" the vertices of $Q$ along the corners of $\tt n 1$ as shown in Fig.~\ref{fig:MatingTrees}.
Strictly speaking, the tree  thus obtained is naturally pointed  and rooted, because the quadrangulation is (the red vertex and  green arrows in Fig.~\ref{fig:MatingTrees}). In this section, we rather choose for practical reasons to root $\ttau_n^{(2)}$ on the corner that was glued to the root corner of $\tt n 1$\footnote{Note that this does not correspond  to the rooting induced by that of $Q$, for instance when the parameter $\eta$ is one. } (it simplifies the combinatorial encoding of feuilletages using nested non-crossing permutations in Sec.~\ref{sec:NPIM}). To obtain the tree $\tt n 2$ defined in the previous sections, the tree $\ttau_n^{(2)}$ has to be rerooted  on the corner of the pointed vertex that faces the min argmin of the labeling $L_n^{(1)}$ of $\tt n 1$\footnote{More precisely,  the min argmin is a corner of $\tt n 1$, and if $e$ is the edge added from this corner towards the pointed vertex $\nu$, $\ttau_n^{(2)}$ has to be rerooted on the corner of $\nu$ which precedes $e$ on $\nu$ counterclockwise in order to obtain $\tt n 2$ (see the orange arrow in Fig.~\ref{fig:MatingTrees}).}. This choice of root was needed to show the convergence. In the following, we use $\ttau_n^{(2)}$ (and its generalizations $\ttau_n^{(j)}$), but the reader should keep in mind that $\tt n 2$ (or $\tt n j$) is obtained by a simple rerooting.
\begin{figure}[!h]
\centering
 \includegraphics[scale=1.1]{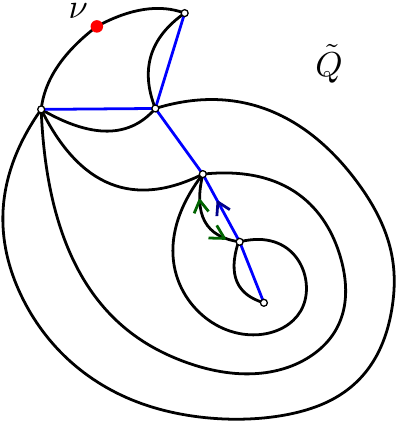}
 \hspace{0.25cm}
 \includegraphics[scale=1.1]{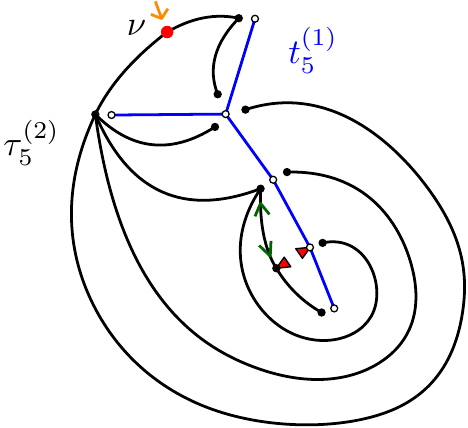}
 \hspace{0.25cm}
  \includegraphics[scale=1.1]{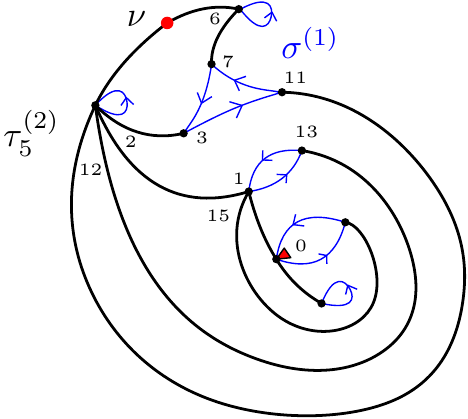}
\caption{\label{fig:MatingTrees}On the left is the map $\tilde Q$, with the edges of both the quadrangulation $Q$ and the tree $\tt n 1$. $\tilde Q$ can be understood as a discrete ``mating" of two trees $\tt n 1$ and $\ttau_n^{(2)}$, as illustrated in the central figure.  The red triangles indicate the root-corners of the two trees, while the green arrows indicate the two possible orientations for the root edge in the quadrangulation $Q$, and the orange arrow indicates the  corner on which to reroot in order to obtain the tree $\tt n 2$ of Sec.~\ref{sec:ISIMCO}, Fig.~\ref{fig:HtoT}. Equivalently, on the right, the tree $\tt n 1$ can be seen as a permutation $\sigma^{(1)}$ gluing some corners of $\ttau_n^{(2)}$ to form the vertices of $Q$ (see Fig.~\ref{fig:Dual-nncross-tree}  to recover $\tt n 1$ from $\sigma^{(1)}$). }
\end{figure}

The tree $\tt n 2$ is made of a subset of the shortest paths to the pointed vertex $\nu$ in the quadrangulation (it contains precisely one ``geodesic" to $\nu$ for every corner of $\tt n 1$).
As for the tree $\tt n 1$, in a sense it  represents a ``ridgeline" of the quadrangulation: on the left of Fig.~\ref{fig:Faces-Bij}, the edges of $\tt n 1$ locally connect vertices which are far from $\nu$, and on the right of Fig.~\ref{fig:Faces-Bij}, the edge of $\tt n 1$ necessarily connects two of the vertices that are the further away from $\nu$.

\paragraph{Encoding $Q$ with two permutations.} Equivalently, the map $Q$ can be encoded by the tree $\ttau_n^{(2)}$, together with a non-crossing \permutation on (a subset of) its  
corner sequence,  
whose complement is a matching. Indeed, we see that the vertices of $\tt n 1$ encode the gluing of the vertices of $\ttau_n^{(2)}$ to form the vertices of $Q$ (middle of Fig.~\ref{fig:MatingTrees}). For each vertex of $Q$, this gluing takes the form of a cycle on the corners of the vertices of $\ttau_n^{(2)}$ which are glued together (right of Fig.~\ref{fig:MatingTrees}). All of these disjoint cycles form a permutation $\sigma^{(1)}$, defined on a subset of corners of $\ttau_n^{(2)}$, with one corner per non-pointed vertex of $\ttau_n^{(2)}$ (or non-rooted for $\tt n 2$).

The planarity of $Q$ is encoded in the fact that $\tt n 1$ is a tree, or equivalently that $\sigma^{(1)}$ is a non-crossing \permutation on  
the {\it clockwise}  corner sequence 
$\bar C_2=\{0<4n-1<4n-2<\ldots< 1\}$ (the cycles of $\sigma^{(1)}$  induce a non-crossing partition on $C_2=\{0<\ldots<4n-1\}$, but it is $(\sigma^{(1)})^{-1}$ that respects the ordering of $C_2$, since when going around  $\tt n 1$ counterclockwise, we go around $\ttau_n^{(2)}$ clockwise). The fact that the Kreweras complement of $\sigma^{(1)}$ is a matching is equivalent to  $Q$ being a quadrangulation. Other characteristics of $Q$, such as  
the information on the distances to $\nu$ for instance,
are encoded in other properties of $\ttau_n^{(2)}$ and of $\sigma^{(1)}$. 
To summarize, consider the encoding of the tree $\ttau_n^{(2)}$ as a pair $(C_2, \sigma^{(2)})$, where $C_2$ is the ordered set $\{0<\ldots< 4n-1\}$, and $\sigma^{(2)}$ is a non-crossing \permutation on $C_2$, then $Q$ can be encoded as the triplet $$(C_2, \sigma^{(2)}, \sigma^{(1)}),$$ where $\sigma^{(1)}$ is also a non-crossing \permutation on a subset of $\bar C_2$. In addition, $\sigma^{(1)}$ and $\sigma^{(2)}$ satisfy other constraints: there is one element of the support of $\sigma^{(1)}$ per disjoint cycle of $\sigma^{(2)}$ appart from the cycle containing 1, their Kreweras complements are both matchings, etc. From the triplet $(C_2, \sigma^{(2)}, \sigma^{(1)})$, the tree $\t_n^{(1)}$ is recovered as $(\bar C_1, \sigma^{(1)})$, where $\bar C_1$ is $C_2$ with the opposite ordering (but still starting at 0), and restricted to the support of $\sigma^{(1)}$.  
In the example of Fig.~\ref{fig:MatingTrees} for instance, we have $\ttau_5^{(2)} =(C_2,\sigma^{(2)})$ and $\tt 5 1 =(\bar C_1,\sigma^{(1)})$, with 
$$
C_2=\{0<1<\cdots<19\},\qquad  \sigma^{(2)} = (0,18)(1,15,17)(2,4,10,12,14)(3)(5,9)(6,8)(7)(11)(13)(16), 
$$
$$
 \bar C_1=\{0<19< 16 < 13< 11< 8< 7< 4< 3< 1\},\qquad  \sigma^{(1)} = (0,16)(1,13)(3,11,7)(4)(8)(19), 
$$
where we recall that $C_2$ (resp.~$\bar C_1$) is the {\it counterclockwise} corner sequence of $\ttau_5^{(2)}$ (resp.~$\tt 5 1$). With this encoding, the corners of $\tt 5 1$ have the same indices in $C_2$ as a subset of the corners in $\ttau_5^{(2)}$: in the following we will say that the corners that have the same index are {\it dual} one to another.

\begin{rem} 
  \label{rem:Discrete-mating} As for the map $\tilde Q$, which contains both the edges of $\tt n 1$ and $\ttau_n^{(2)}$, it can be understood as a ``discrete mating" of the trees $\tt n 1$ and $\ttau_n^{(2)}$ : following the clockwise corner sequence of $\tt n 1$ and the counterclockwise corner sequence of $\ttau_n^{(2)}$, some of the corners of the trees are glued together to recover $\tilde Q$ (the dual corners). It is some kind of discrete version of the  mating of tree introduced in  \cite{MatingTrees}, but where one tree is twice the size of the other, and the trees are not glued along all of their corners or all of their leaves, but rather all of the corners of the small tree are glued to a subset of the corners of the bigger tree.
 \end{rem}

\paragraph{More generalities on graph explosions.}  More generally, from the Chapuy-Marcus-Schaeffer bijection \cite{ChapMarcSchaeff}, non-planar quadrangulations are bijectively encoded by ``discrete matings" of trees and maps of genus $g$ with a single face (unicellular), for which the edges of the unicellular map are then removed, or equivalently as a triplet $(C_2, \sigma^{(2)}, \sigma^{(1)})$, where $\sigma^{(2)}$ is still a non-crossing \permutation on $C_2$, but $\sigma^{(1)}$ is now a permutation of genus $g$,  defined as the genus of the bipartite map $(f_0, \sigma^{(1)})$, where if $C_1=\{0<i_1<\ldots<i_{k}\}$, then $f_0$ is the cycle $(0 i_1 \ldots i_{k})$  \cite{CoriGenPerm}, that is,  
\be
\label{eq:genus-perm}
 2g(\sigma^{(1)}) = 1+n - \#(\sigma^{(1)}) - \#(f_0\circ\sigma^{(1)}),
 \ee
where we have denoted by $\#$ the number of disjoint cycles of a permutation. Note that the unicellular map $U_M$  which generalizes $\tt n 1$ when the quadrangulation has genus $g$ is the map  whose vertices and edges are respectively given by $\sigma^{(1)}$ and $(\fC\circ\sigma^{(1)})^{-1}$.

The deconstruction of the map $Q$ as a tree and a permutation that encodes a gluing of a subset of its corners falls in the more general context of ``graph explosion". Trees are connected maps with no loops: if one takes a connected map $M$, removing successively some well-chosen edges of $M$ -- each edge being chosen among those belonging to cycles -- allows going from $M$ to a tree $T$. 

Another way to decompose maps relies on the ``explosion'' of some vertices. The idea is to preserve the edges and to rather modify the vertices by adding some copies of them. More precisely, for each cycle of the permutation encoding the vertices in Def.~\ref{def:Maps},  we may choose to split it in two or more disjoint consecutive cycles e.g.~$(123456)\rightarrow(1)(234)(56)$. This is illustrated in Fig.~\ref{fig:vertex-splitting}. When exploding a vertex, a corner is naturally distinguished on each copy of the vertex (the corners $(1), (4,2)$, and $(6,5)$ in this example).  Consider a map $M$ and explode some of its vertices to obtain a map $M'$; gathering the disjoint cycles $\sigma_v$ for each exploded vertex $v$, the explosion is thus encoded as a permutation $\sigma$ defined on some of the corners of $M'$, which allows reversing the explosion and reconstructing the vertices of $M$.
\begin{figure}[h!]
 \centering
\raisebox{+0.2cm}{ \includegraphics[scale=1.4]{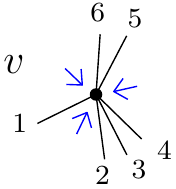}}
\hspace{1.7cm}\raisebox{+1.3cm}{$\rightarrow$}\hspace{1.5cm}
\includegraphics[scale=1.4]{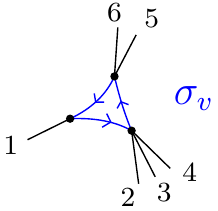}
\caption{The explosion of a vertex $v$ of a map $M$ is encoded by a cycle $\sigma_v$ on the distinguished corners of the resulting map $M'$. }
\label{fig:vertex-splitting}
\end{figure}

Now suppose that the map obtained after exploding some of the vertices of a map $M$ is a tree $M'=T$, encoded by its corner sequence (an ordered subset $C'$) and a non-crossing \permutation $\sigma'$ on $C'$. Then the support of the permutation $\sigma=\prod_v \sigma_v$ just described is a subset of $C'$, so that the map $M$ can be encoded by the triplet $(C',\sigma', \sigma)$. In addition, there is at most one element of the support of $\sigma$ per disjoint cycle of $\sigma'$.\footnote{Note that this last condition can be lifted, which corresponds to authorizing the explosion of the vertices in more than one cycle each. More precisely, if $\gamma^v$ is a cycle coding a vertex $v$ of $M$, exploded in several disjoint (non-necessarily consecutive) cycles $\gamma^v_1,\ldots, \gamma^v_p$, we can still encode the explosion of $v$ using some disjoint cycles $\sigma^v_1,\ldots, \sigma^v_p$ defined on the corners of the exploded map $M'$, provided that $\gamma^v_1\circ\ldots\gamma^v_p$ forms a non-crossing partition of  $\gamma$.} 
\

There are plenty of examples in the literature of such decompositions of maps in pairs of a tree $T$ (the tree $\tt n 2$ in our case) and a unicellular map $U_M$ of genus $g$, or equivalently an ordered set $C'$, a non-crossing \permutation $\sigma'$, and a permutation $\sigma$ of genus $g$, such as for instance: the construction of CVS \cite{CV, Schae} and Marckert-Mokkadem \cite{MM} detailed above, its generalization for quadrangulations of genus g \cite{ChapMarcSchaeff}, its generalization for maps with faces of even lengths \cite{Bouttier2004}, the bijection of Bernardi \cite{Bernardi06} for tree-rooted maps, 
or the decomposition in C-decorated trees of \cite{ChapFerFus} for unicellular maps. 

In these decompositions, the tree $T$ is often of a very specific kind. We would like to stress that on the other hand, some aspects of these decompositions are very general and do not rely on the details of the maps (restriction on the length of the faces or bipartiteness for instance) or on the way that the tree $T$ is extracted from the map (the rule to explode the vertices is based on the distances to the pointed vertex in CVS \cite{CV, Schae}, \cite{MM}, \cite{ChapMarcSchaeff}, and \cite{Bouttier2004}], on an orientation of the edges according to the spanning tree in Bernardi's bijection \cite{Bernardi06}  and on the notion of trisection in \cite{ChapFerFus}). Indeed, it is clear that given {\it any} map, there are many ways to explode it in a pair of a tree $T$ and a permutation $\sigma$ (or a unicellular map $U_M$), or equivalently of an ordered set $C'$, a non-crossing \permutation $\sigma'$, and a permutation $\sigma$, and conversely, {\it any} pair/triplet of such objects are combined to form a map. Furthermore, {\it for any explosion of a map $M$ in a tree $T$}, the genus of the map $M$  is  the genus of the permutation $\sigma$ as defined in \eqref{eq:genus-perm} (or equivalently the genus/excess of the unicellular map $U_M$). In particular, pairs of trees or pairs of non-crossing partitions on an ordered set always encode planar maps. This motivates the degree of generality chosen for the notion of $\D$-general feuilletage in Def.~\ref{def:general-refoldings}.

\subsection{Iterative construction and new propositions for the notion of  $\D$-combinatorial map}
\label{sec:NPIM}

\paragraph{Discrete  iterated feuilletages.}
To build a rooted pointed planar quadrangulation $\Qrec n 1 2\in \Qset n$ (the quadrangulation $Q$ of the previous section), a first rooted planar tree $\tt n 1$ is considered, which is then dressed with a labeling, from which the edges of $\Qrec n 1 2$ are built according to the rule of the CVS bijection  of Sec.~\ref{sec:CVS-bij}, after what the edges of  $\tt n 1$ are deleted and a tree $\ttau_n^{(2)}$ with $2n$ edges is extracted: the vertices of $\tt n 1$ encode the gluings of the vertices of $\ttau_n^{(2)}$, while the fact that $\tt n 1$ is a tree accounts for the planarity of $\Qrec n 1 2$.
\begin{figure}[!h]
\centering
\raisebox{+1.1cm}{ \includegraphics[scale=0.7]{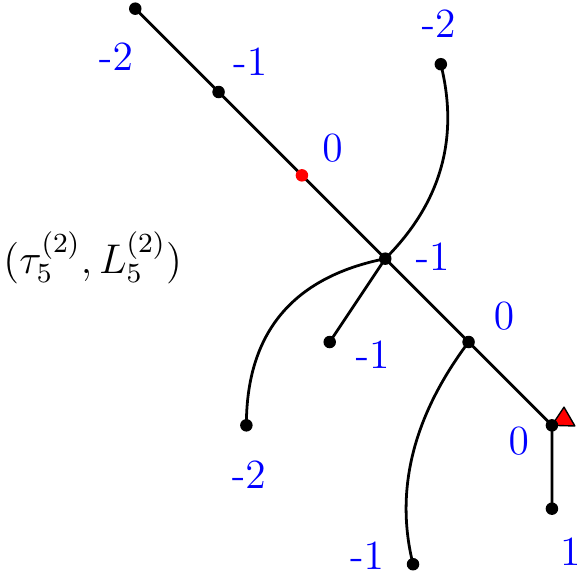}}
 \hspace{0.4cm}
 \includegraphics[scale=0.7]{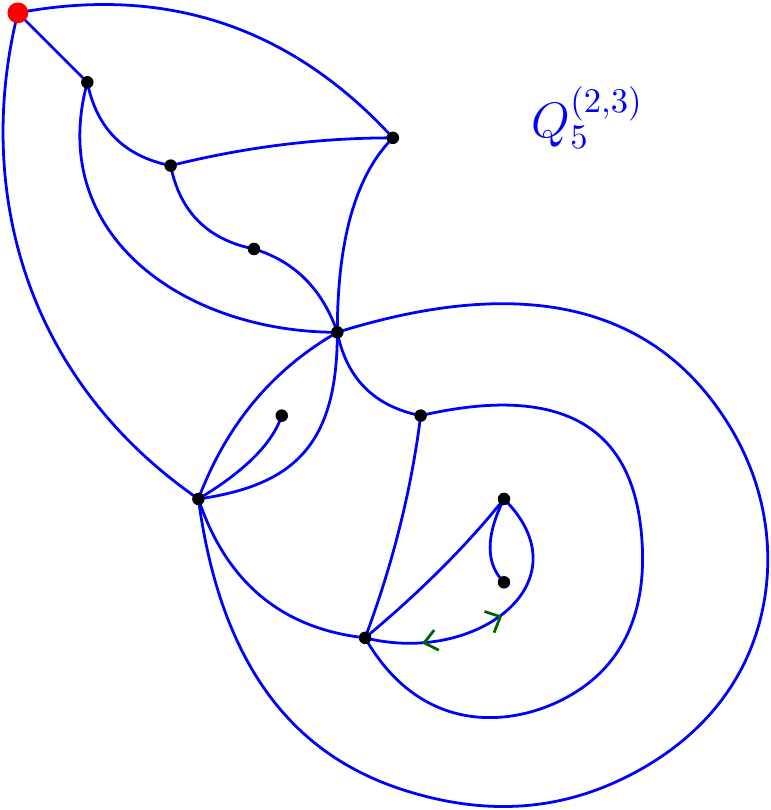}
 \hspace{0.4cm}
  \includegraphics[scale=0.7]{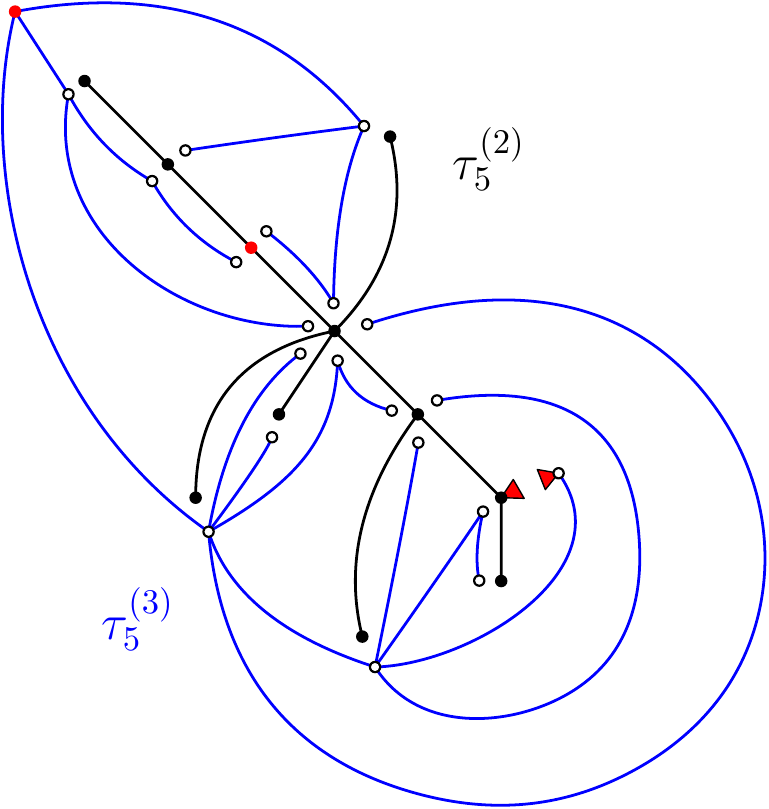}
\caption{\label{fig:Third-Tree-etc}The tree $\ttau_n^{(2)}$ is in turn provided with a labeling $L_5^{(2)}$ (left), from which a second rooted pointed planar quadrangulation $\Qrec 5 2 3$ is built (middle), as well as a third tree  $\ttau_5^{(3)}$ (right).}
\end{figure}

Now dress in turn  $\ttau_n^{(2)}$ with a labeling, and add new edges according to the rules of the  CVS bijection (see Fig.~\ref{fig:Third-Tree-etc}). Deleting the edges of  $\ttau_n^{(2)}$, we obtain a rooted pointed planar quadrangulation $Q_n^{(2,3)}\in \Qset{2n}$. If we also identify the vertices of $\ttau_n^{(2)}$ according to the vertices of $\tt n 1$, we obtain a 3-discrete  feuilletage  $\Rn n 3$ (right of Fig.~\ref{fig:Third-Recursive-Refolding}).

Just as for  $\ttau_n^{(2)}$ at the previous step, we can extract a rooted planar tree  $\ttau_n^{(3)}$ with $4n$ edges from the map $ \tilde Q_n^{(2,3)}$, which is  $ \Qrec n 2 3$, together with the edges of $\ttau_n^{(2)}$  (right of Fig.~\ref{fig:Third-Tree-etc}).  $\Rn n 3$ is encoded by the tree $\ttau_n^{(3)}$, identified by both the vertices of $\ttau_n^{(2)}$  and of $\tt n 1$. 
The corners of $\ttau_n^{(2)}$ are dual to corners of $\ttau_n^{(3)}$ (dual corners are introduced above Rem.~\ref{rem:Discrete-mating}), and the permutation $\sigma^{(1)}$ (or equivalently the vertices of $\tt n 1$) glues corners of $\ttau_n^{(3)}$ together, and consequently glues vertices  of $\Qrec n 2 3$ (Fig.~\ref{fig:Third-Recursive-Refolding}). We see that the planar quadrangulation  $\Qrec n 2 3$,  which has $2n$ faces, $4n$ edges, $2(n+1)$ vertices, is ``folded" repeatedly into an object which has as many vertices as $\tt n 1$ plus the pointed vertices of $\tt n 2$ and $\tt n 3$, i.e.~$n+3$,  thus the name {\it discrete  iterated feuilletage}.

\

In this description, we have been vague about the way the objects were rooted. The discrete  feuilletage $\Rn n 3$  described above has one oriented edge, which induces a rooting on $\Qrec n 2 3$, $\Qrec n 1 2$, $\ttau_n^{(3)}$, $\ttau_n^{(2)}$, and $\tt n 1$, and has two pointed vertices, one on $\Qrec n 2 3$, and one on $\Qrec n 1 2$. Indeed,  the oriented edge in $\Rn n 3$ is also an oriented edge of $\ttau_n^{(3)}$, or equivalently a root corner of $\ttau_n^{(3)}$. The root corner of $\ttau_n^{(2)}$ is the one facing the root corner of $\ttau_n^{(3)}$ (its dual corner), and the root corner of $\tt n 1$ is that dual to the root corner of $\ttau_n^{(2)}$. In $\ttau_n^{(3)}$, when adding the distances to the pointed vertex, there is only one edge incident to the root vertex, and whose other endpoint has a smaller label. The root of $\Qrec n 2 3$ is obtained by orienting this edge in one or the other way, according to the parameter $\eta$ of the CVS bijection. The root of $\Qrec n 1 2$ is obtained the same way from the root corner of $\ttau_n^{(2)}$. 
In the rooting convention of this section,  $\Rn n 3$ is thus a rooted discrete  feuilletage which in addition is pointed twice, once by distinguishing a vertex of $\Qrec n 2 3$, and once a vertex of $\Qrec n 1 2$.   Note that $\Rn n 3$ as just described is not rooted and pointed as the discrete feuilletages introduced in Sec.~\ref{sec:ISIMCO}, which are the deterministic combinatorial objects in the support of  $\DR{n}{3}$. The latter are indeed obtained by starting from $\tt n 3$ instead of $\ttau_n^{(3)}$. However, starting from  $\ttau_n^{(3)}$ simplifies the encoding using non-crossing partitions, which is why we make this choice in the present section.

\begin{figure}[!h]
\centering
 \includegraphics[scale=0.75]{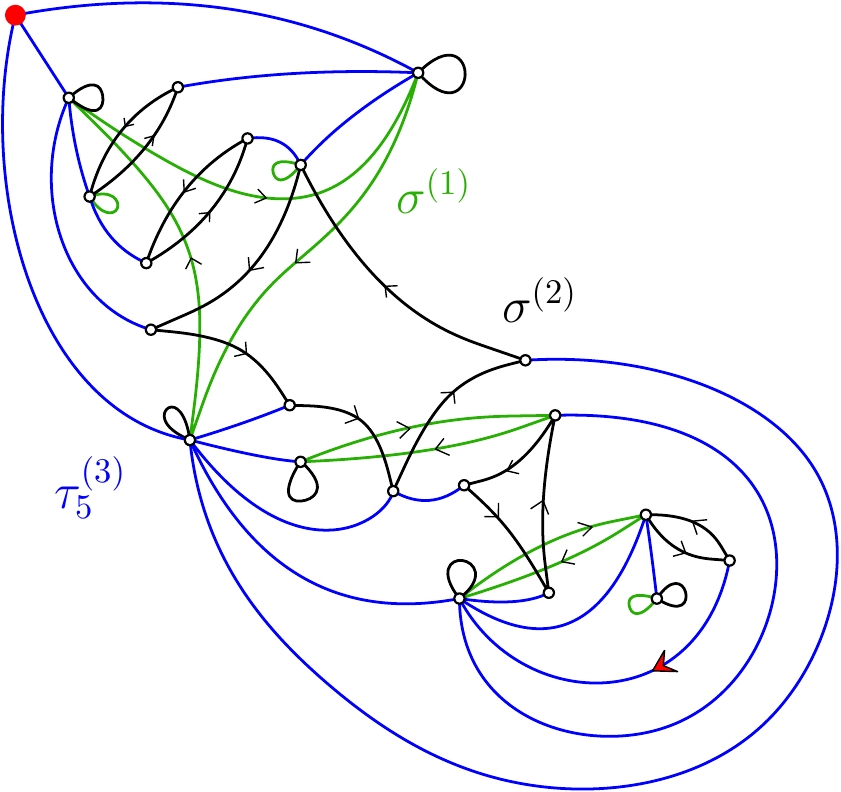}
 \hspace{-0.5cm}
 \includegraphics[scale=0.75]{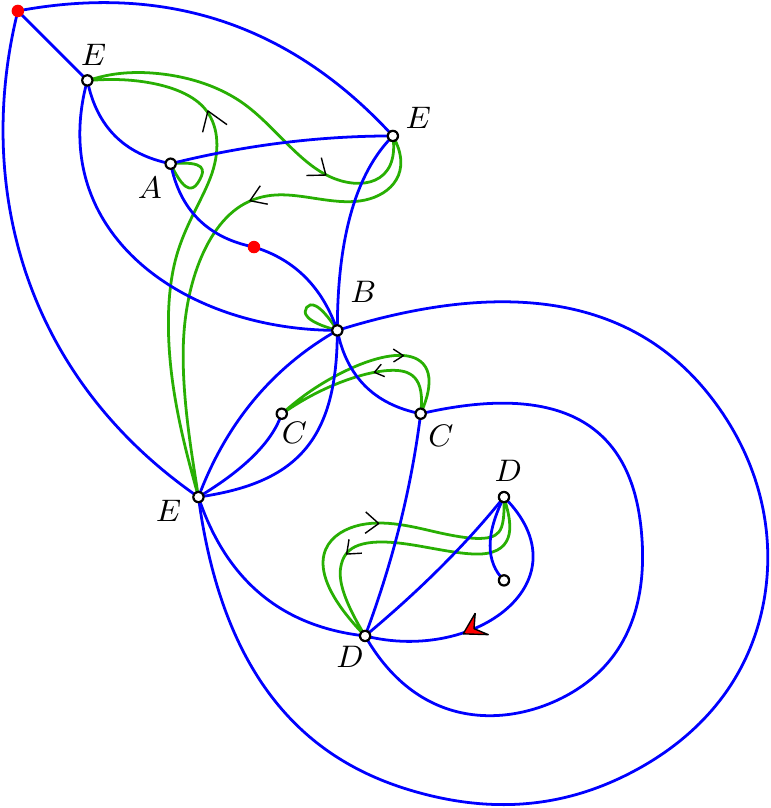}
 \hspace{-0.35cm}
 \raisebox{+0.9cm}{ 
 \includegraphics[scale=0.75]{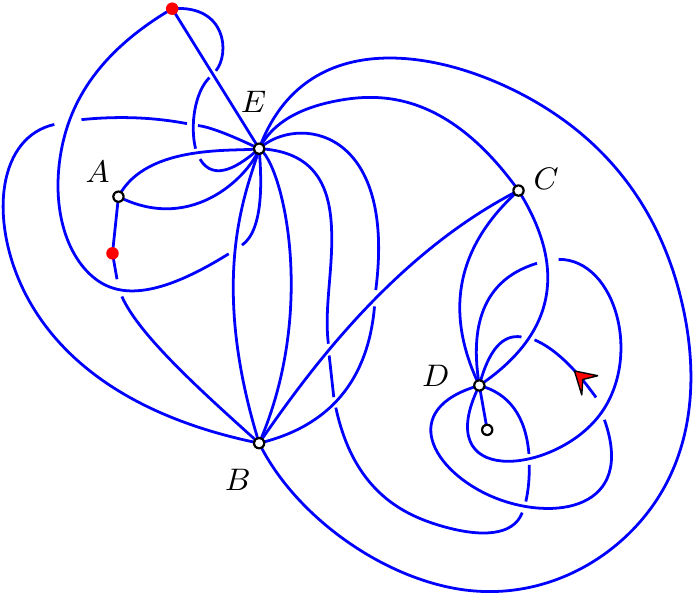}
 }
\caption{\label{fig:Third-Recursive-Refolding}The resulting 3-discrete  feuilletage $\Rn 5 3$, represented as $\ttau_5^{(3)}$ glued by $\sigma^{(2)}$ and $\sigma^{(1)}$ (left), as $\Qrec 5 2 3$ (ur to a rerooting) glued by $\sigma^{(1)}$ (middle), and as the completely folded object (right). To represent it, we used the convention \eqref{fig:Convention-foldings-to-map} (see Rem.~\ref{rem:Convention-foldings-to-map}).}
\end{figure}

\begin{rem}
\label{rem:Convention-foldings-to-map}
While the corners of $\ttau_n^{(3)}$ are dual to corners of $\ttau_n^{(2)}$ and in turn to corners of $\tt n 1$, the corners of $\ttau_n^{(3)}$ are not dual to corners of $\Qrec n 2 3$. However we may choose a convention, such as the following.
\ben
\label{fig:Convention-foldings-to-map}
 \includegraphics[scale=1.1]{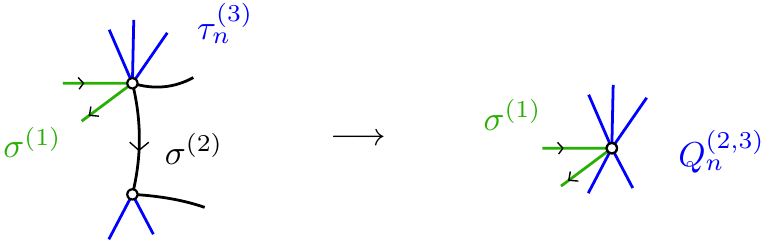}
\een
With such a choice, the 3-discrete feuilletage {$\Rn n 3$} can be seen as $\Qrec n 2 3$ (up to a rerooting), together with an additional permutation $\sigma^{(1)}$ now defined on its {\it corners}. The result is now a combinatorial map in the usual sense, of arbitrary genus (an embedding of {$\Rn n 3$} on a surface of a priori high genus). This is the convention used to draw the figure on the right  of Fig.~\ref{fig:Third-Recursive-Refolding}. For random objects, and with this convention, the genus of the map corresponding to $\DR n\D$ is obviously expected not to be bounded asymptotically.  
\end{rem}
This process can be iterated: dressing $\ttau_n^{(\D)}$ with a labeling, we build a rooted planar tree $\ttau_n^{(\D+1)}$ with $2^\D n$ edges {and $n+\D$ vertices}, such that $\ttau_n^{(\D+1)}$ identified according to the vertices of $\ttau_n^{(\D)}$ (up to a rerooting) forms a pointed rooted planar map $\Qrec n \D {\D+1}\in \Qset{2^\D n}$, and a discrete feuilletage {$\Rn n {\D+1}$} is obtained by identifying the  vertices of every  $\ttau_n^{(j)}$ for $3\le j\le D+1$ according to the vertices of $\ttau_n^{(j-1)}$ (as well as $\tt n 1$ for $j=2$). A {$\D$}-discrete  feuilletage has one oriented edge, which induces a rooting of all the  iterated trees $\ttau_n^{(j)}$ and all the quadrangulations $\Qrec n j {j+1}$, and is pointed $D-1$ times, once per quadrangulation $\Qrec n  j {j+1}$. Some examples of planar maps $\Qrec n j {j + 1}$ are shown in Fig.~\ref{fig:Qkkp1} for $j=1,2,3,4$.

\begin{figure}[h!]
  \centerline{ \includegraphics[scale=0.12]{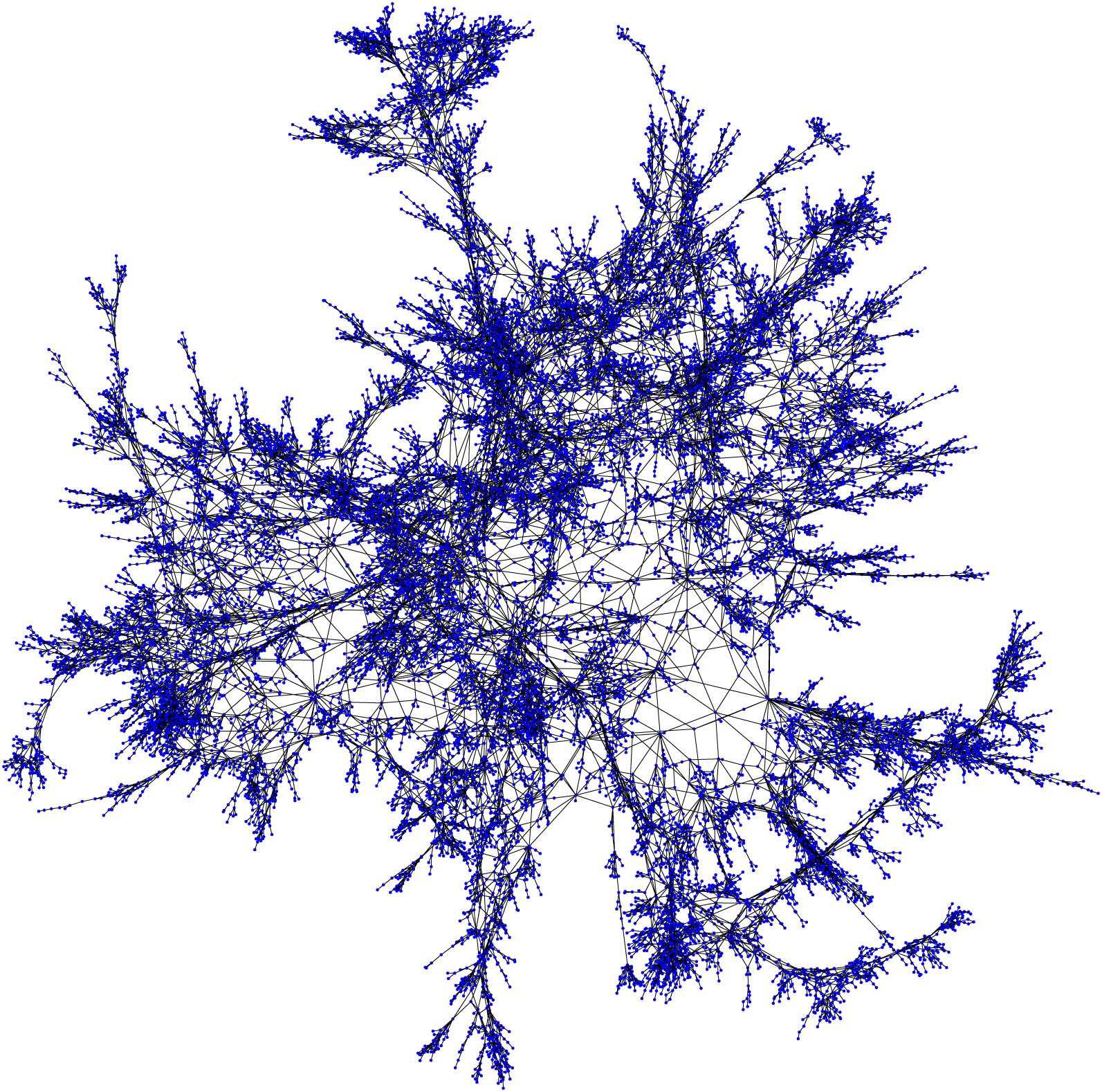}~~~ \includegraphics[scale=0.12]{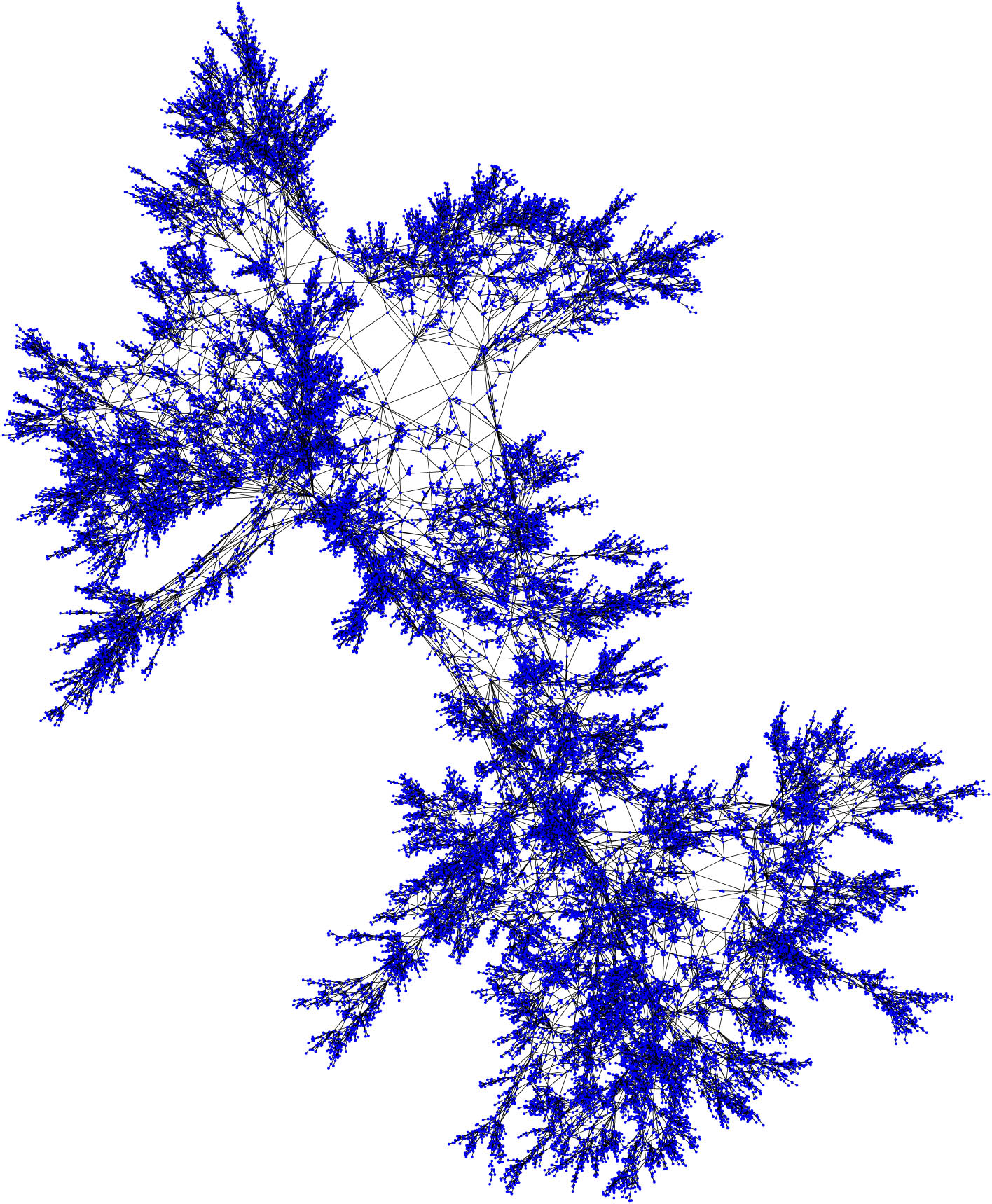}}~\\
  \centerline{\includegraphics[scale=0.12]{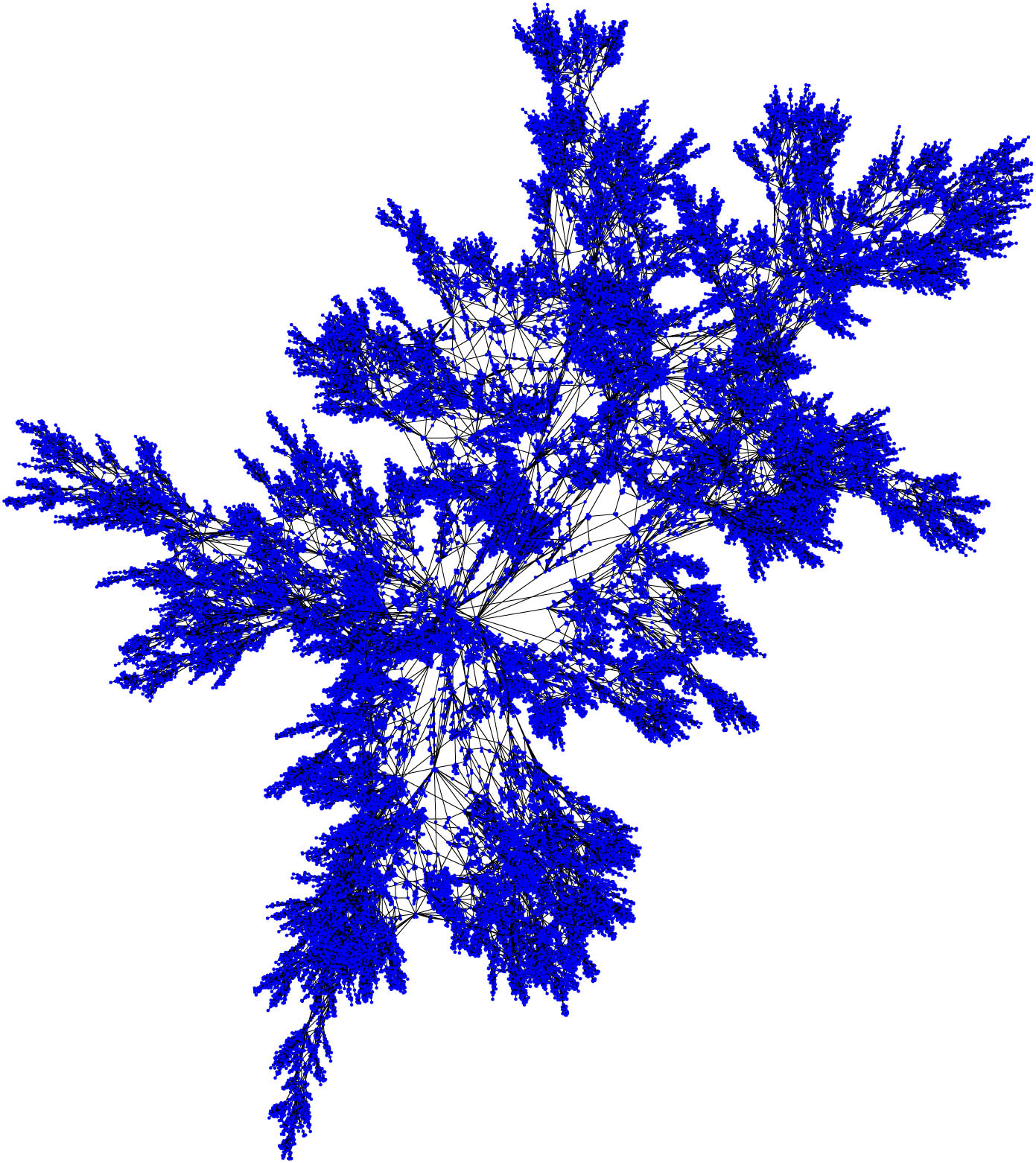}~~~\includegraphics[scale=0.13]{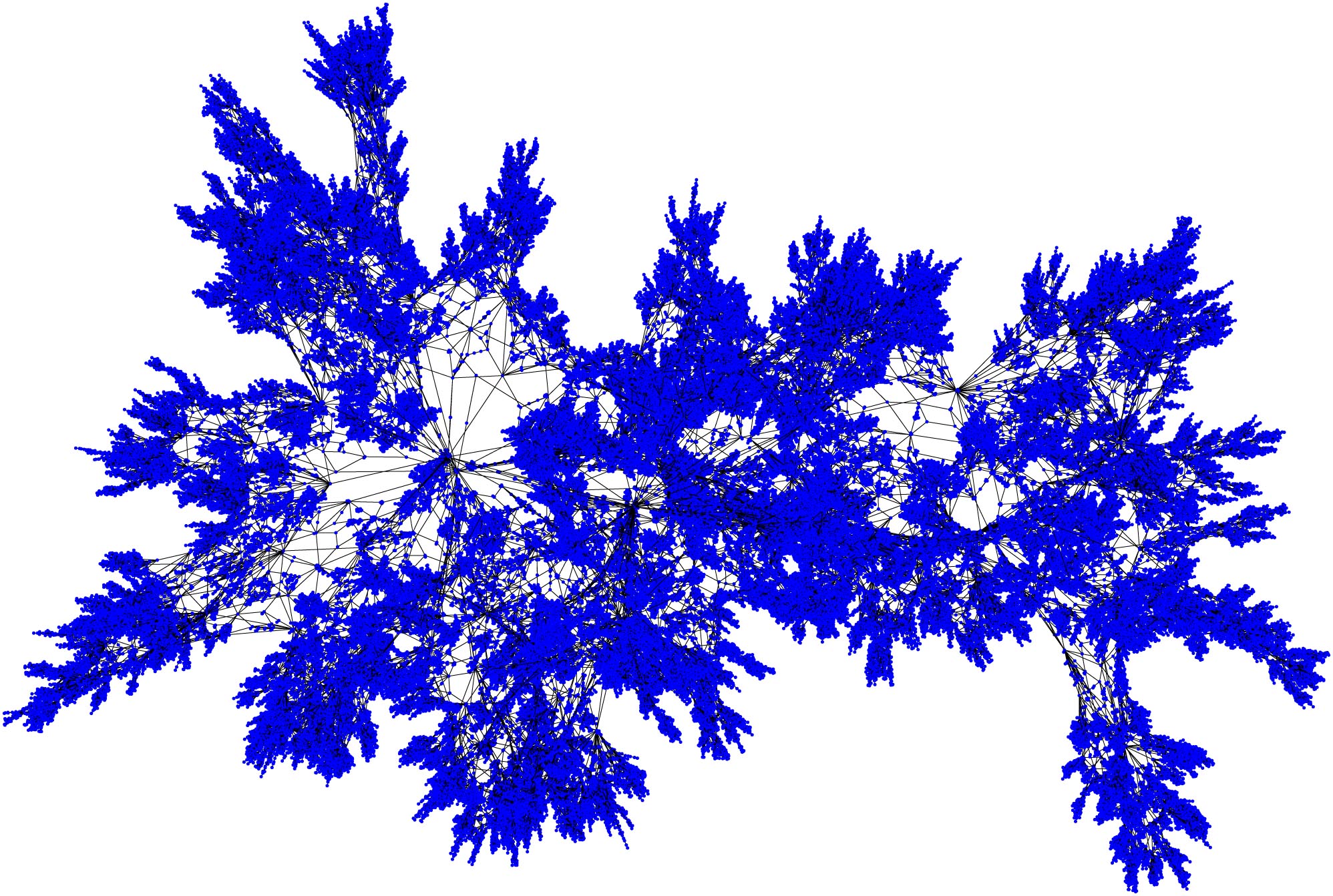}}
    \caption{\label{fig:Qkkp1}Planar maps $\Qrec n j {j + 1}$ obtained thanks to some labeled trees $(C_{25000}^{(j)},L_{25000}^{(j)})$ for $j$ from 1 to 4. Loops and multiples edges have been discarded for the drawing (first map has 25000 vertices, and for the next ones 50k, 100k, 200k). The diameter of $\Qrec n j {j + 1}$ scales asymptotically as $n^{1/2^{j+1}}$.}
  \end{figure}

\paragraph{Asymptotic enumeration of pointed and non-pointed discrete  iterated feuilletages.} The CVS bijection allows computing the numbers $m^{\bullet (2)}_n$ and  $m^{(2)}_n$ of rooted pointed and rooted planar quadrangulations respectively as
\ben 
m^{\bullet (2)}_n = 2\times 3^nC_n \qquad \text{and} \qquad(n+2)m^{(2)}_n = 2\times 3^n C_n,
\een
where the factor 2 comes from the choice of  parameter $\eta$ in the CVS bijection, $n+2$ comes from the choice of a pointed vertex in the quadrangulation $\Qrec n 1 2$ which has $n+2$ vertices, and $3^n$ comes from the choice of labeling. 

Using the fact that asymptotically, $C_n\sim \pi^{-1/2}4^nn^{-3/2}$, we obtain the asymptotics for $m^{\bullet (2)}_n$ and  $m^{(2)}_n$:
\ben
m^{\bullet (2)}_n   \sim 2\pi^{-1/2}\cdot 12^{n}\cdot n^{\gamma_2 - 1} \qquad \text{and} \qquad m^{(2)}_n  \sim 2\pi^{-1/2}\cdot 12^{n}\cdot n^{\gamma_2 - 2},
\een
with $\gamma_2 = -1/2 $. This critical exponent (also called string susceptibility in physics) is to be compared with that for trees, $\gamma_1 = 1/2 $. As this critical exponent is important in theoretical physics discussions, we compute it in the case of iterated discrete feuilletages, pointed or not.

We do not provide an exact formula for the numbers $m^{\bullet (\D)}_n$ and  $m^{(\D)}_n$ of rooted pointed and rooted non-pointed $\D$-discrete feuilletages. This depends on the precise definition of the rooting for instance, and at this point it is difficult to retain one or the other as more natural, from the point of view of the folded object. For instance, one has to decide how exactly the $\D-1$ pointed vertices are chosen: among the $n+\D$ vertices of $\Rn n 3$ (laying a factor $(n+\D)^{\D-1}$) or on the various quadrangulations (laying a factor $\prod_{j=0}^{\D-2} (2^jn+2)$). In any case, the asymptotics for these numbers behave as  
\ben
m^{\bullet (\D)}_n  \sim c_\D\cdot \lambda_\D^{n}\cdot n^{\gamma_\D +\D - 3} \qquad \text{and} \qquad m^{(\D)}_n \sim c_\D\cdot \lambda_\D^{n}\cdot n^{\gamma_\D - 2},
\een
where $\lambda_\D$ has a contribution of the form $3^{\sum_{j=0}^{\D-2} 2^j}$ coming from the enumeration of the iterative labelings of the trees, and the critical exponent is now
\ben
\gamma_\D = \frac 3 2 - \D.
\een
Again, this generalizes the well known critical exponents for rooted planar trees and maps, $\gamma_1=1/2$ and $\gamma_2=-1/2$.

\paragraph{A simple encoding using nested permutations.}

The planar tree $\ttau_n^{(\D)}$ is encoded by its corner sequence $C_\D=\{0<1<\cdots<2^{\D}n - 1\}$, together with a non-crossing \permutation $\sigma^{(\D)}$ on $C_\D$. The identification of the vertices of $\ttau_n^{(\D)}$ by the vertices of $\ttau_n^{(\D-1)}$ to obtain the planar map $\Qrec n {\D-1} \D$ takes the form of a non-crossing \permutation $\sigma^{(\D-1)}$ on a subset of $\bar C_\D$, which is $C_\D$ with the reversed order:  $\bar C_\D=\{0<2^{\D}n-1<2^{\D}n-1\cdots<1\}$. The support of  $\sigma^{(\D-1)}$ contains one element of $C_\D$ per cycle of $\sigma^{(\D)}$ at most. From this permutation, the tree $\ttau_n^{(\D-1)}$ is obtained as $(\bar C_{\D-1},  \sigma^{(\D-1)})$, where $\bar C_{\D-1}$ is $\bar C_\D$ restricted to the support of $\sigma_{\D-1}$. 

The identification of the vertices of $\ttau_n^{(\D-1)}$ by the vertices of $\ttau_n^{(\D-2)}$ takes the form of another non-crossing \permutation $\sigma^{(\D-2)}$ on a subset of $C_\D$, such that the support of  $\sigma^{(\D-2)}$ is included in that of $\sigma^{(\D-1)}$, and contains one element of $C_\D$ per cycle of $\sigma^{(\D-1)}$ at most, and so on. 
We propose to define a notion of ``$\D$-general feuilletages'' which encompasses our notion of  $\D$-discrete  feuilletages:
\begin{defi}
\label{def:general-refoldings}
We call  $\D$-general feuilletage a $(\D+1)$-uplet 
$$
\cM_\D=(C_\D, \sigma^{(\D)}, \ldots, \sigma^{(1)}),
$$ 
where \begin{itemize} 
\item $C_\D$ is the totally ordered set $C_\D=\{0 < 1,\cdots < N-1\}$ (the counterclockwise corner sequence of $\ttau_n^{(\D)}$, with $N=2^{\D}n$), 
\item $\sigma^{(\D)}$ is a non-crossing \permutation on $C_\D$, 
\item  for $1\le j\le \D$, $\sigma^{(j)}$ is a permutation on $C_\D$ whose support $C_j$ is included in that of $\sigma^{(j+1)}$ (the permutations are nested), $\sigma^{(j)}$ is a non-crossing \permutation on $C_j^\star$, and it has one element per disjoint cycle of $\sigma^{(j+1)}$ at most.\end{itemize}
\end{defi} 
We have used the notation $C_j^\star$, which is $C_j=\{0<i_1<\ldots < i_p\}$ if $j$ and  $\D$ have the same parity, and $\bar C_j=\{0<i_p<i_{p-1}\ldots < i_1\}$ otherwise, so that $\sigma^{(j)}$ is a non-crossing \permutation on $C_j^\star$ iff it induces a non-crossing partition on $C_j$ (and thus on $C_\D$), and it respects the ordering of $C_\D$ if  $j\equiv 0 \mod \D$ and of $\bar C_\D$ if $j\equiv 1 \mod \D$.
 
 For any $j$,  $(C_j^\star, \sigma^{(j)})$ is a planar tree (the tree $\ttau_n^{(j)}$ for a discrete feuilletage). Gluing the vertices of this planar tree according to the cycles of $\sigma^{(j-1)}$ provides a planar map, encoded by the triplet $(C_j^\star, \sigma^{(j)}, \sigma^{(j-1)})$ (for a discrete feuilletage, it is a planar quadrangulation $\Qrec n {j-1} j$)\footnote{The reason why we don't have to shift the first element of the set $C_j$ to obtain the various  iterated trees is precisely because the root of $\ttau_n^{(j+1)}$ is chosen as dual to the root of $\ttau_n^{(j)}$. }.

Given a  $\D$-general feuilletage $\cM_\D$ and for $k>1$,  $\cM_{\D-k+1}=(C_\D, \sigma^{(\D)}, \ldots, \sigma^{(k)})$ is of the same form as a $(\D-k+1)$-general feuilletage. For $k=\D$, it is a planar tree, for $k=\D-1$ it is a planar map, and $\cM_{\D-k+1}$ is obtained from $\cM_{\D-k}$ by folding it as many times as there are disjoint cycles in $\sigma^{(k)}$.

\paragraph{Another encoding as gluings of trees.} While the encoding above renders the  iterative gluings of the objects according to the cycles of the permutations $\sigma^{(j)}$, which applies well to the discrete  iterated feuilletages and to the distance $d^{(2)}_{\RR{\D}}$, we provide another combinatorial description for the object for which all the edges of every $\ttau_n^{(j)}$ are kept (the generalization of  $\tilde Q_n^{(j, j+1)}$).

\begin{defi} 
The second kind of  $\D$-generalized maps, with colored edges, is defined as
\begin{itemize}
\item{ $\D$ sets of darts $\cD_1, \ldots, \cD_D$. For $c\in\{1,\ldots,D\}$, the set $\cD_c$ has darts of color $c$.}
\item{ $\D$ involutions without fixed points $\alpha_1,\ldots, \alpha_D$, which assemble the darts into colored edges. }
\item{For every $c\in\{1,\ldots,D-1\}$, a permutation $\Sigma_{c, c+1}$ on $\cD_c\cup \cD_{c+1}$ such that the map $$ (\cD_c\cup \cD_{c+1}, \Sigma_{c, c+1}, \alpha_c\cup \alpha_{c+1})$$ is planar (it has vanishing genus \eqref{eq:Euler}).  
}
\item For $2\le c\le D-1$, the restriction of $\Sigma_{c, c+1}$ to $\cD_c$ coincides with the explosion of $\Sigma_{c-1, c}$ by $\cD_c$, and we denote it  by  $\Sigma_c$. We extend this notation for $\Sigma_1=\Sigma_{1, 2}\vert_{\cD_1}$ and $\Sigma_D=\Sigma_{D-1, D}$ exploded by $\cD_D$. For $1\le c\le D$, the map $(\cD_c, \Sigma_{c}, \alpha_c)$ must be a tree.
\end{itemize}
\end{defi}

This definition uses the explosion of $\Sigma_{c-1, c}$ by $\cD_c$, defined similarly as above. We illustrate it on an example: if $\gamma = (a_1, a_2, b_1, b_2, b_3, a_3, b_4, b_5, a_4)$ is a cycle of $\Sigma_{c-1, c}$ where the $a_i$ are in $\cD_{c-1}$ and the $b_i$ in $\cD_{c}$,   the disjoint cycles $(b_1, b_2, b_3)$, $(b_4, b_5)$ obtained by splitting $\gamma$ along the $a_i$ are disjoint cycles of the explosion of $\Sigma_{c-1, c}$ by $\cD_c$.

\subsection{Review on existing generalizations of maps in ``higher dimensions''}
\label{sec:RevGM}

\subsubsection{Colored triangulations}
\label{sec:TM}  
Colored triangulations are  $\D$-dimensional simplicial pseudo-complexes formed by considering  $\D$-simplices whose vertices are colored from 0 to  $\D$ (in this section, by ``triangulation" we always mean $\D$-dimensional triangulation). Such triangulations appear in the topology (see \cite{ItalianSurvey} for a survey) and combinatorial (see e.g.~\cite{Bonz18, CP18}) literature, and in the context of random tensor models \cite{GurauInvitation, GurauBook}. 

 To glue two  $\D$-simplices along two $(D-1)$-simplices, one has to match the colors of the incident vertices. This means that the information on the triangulation is fully encoded in its dual graph, which has a vertex for each  $\D$-simplex, and an edge of color $c\in\{0,\ldots,D\}$ between two vertices if they are glued along their $(D-1)$-simplices that are not incident to the vertex of color $c$.
A triangulation is orientable iff its dual graph is bipartite. 
The dual graphs of orientable colored triangulations are therefore $(D+1)$-regular, properly-edge-colored bipartite graphs. For triangulations with black and white simplices each labeled from 1 to $n$, we encode such a graph as a $(D+1)$-uplet of permutations $(\rho_0,\ldots,\rho_{D})$ acting on the $n$ white vertices, so that $\rho_c(k)=p$ iff an edge of color $c$ goes between the white vertex $k$ and the black vertex $p$. One can always fix $\rho_0$ to be the identity, in which case only $\D$ permutations are required.

\

Note that there is a bijection \cite{SWMaps, Bonzom:2016vc, Lionni:17} between orientable colored triangulations labeled this way and so-called $(D+1)$-constellations \cite{BOUSQUETMELOU2000337} (but without any assumption on the genus of the constellation),  
encoded the same way using a $(D+1)$-uplet of permutations $(\rho_0,\ldots,\rho_{D})$. The latter are maps with white vertices labeled from 1 to $n$ each with half-edges of each one of the $(D+1)$ colors, and vertices of color $c$ for $c\in\{0,\ldots,D\}$, given by the disjoint cycles of $\rho_c$. A slight modification of this mapping provides a bijection between  $\D$-dimensional colored triangulations rooted on a $(D-1)$-simplex and  $\D$-constellations encoded by a  $\D$-uplet of permutations $(\rho_1,\ldots,\rho_{D})$ \cite{Lionni:17, FLT18}. 

\subsubsection{$\D$-G-maps and  $\D$-maps}
\label{sec:DS}

In the constructions by Lienhardt \cite{genMapsLien}, two generalizations of maps are defined on a set of darts $\cD$. The notion of  $\D$-G-map is suited for orientable or non-orientable topological spaces. A  $\D$-G-map without boundary is defined by a $(D+2)$-uplet $(\cD,\alpha_0,\ldots,\alpha_{D})$, where the $\alpha_i$ are involutions on $\cD$ without fixed points. In addition, for $0\le i\le D-2$ and $i+2\le j\le D$, $\alpha_i\alpha_j$ is also an involution. The involution $\alpha_0$ gathers the darts into edges, $\alpha_1$ groups edges into vertices of valency 2, thus forming faces, $\alpha_2$ groups pairs of edges together thus gluing faces, and so on. When two darts are identified by $\alpha_2$, the other two-darts of the edges they belong to, must also be paired by $\alpha_2$, so that $\alpha_0\alpha_2$ is an involution, and so on.

A  $\D$-map is defined similarly as a $(D+1)$-uple $(\cD,\alpha_0,\ldots,\alpha_{D-1})$, where for $1\le i\le D-2$, $\alpha_i$ is an involution on $\cD$ without fixed point, $\alpha_{D-1}$ is a permutation on $\cD$, and for $0\le i< i+2< j\le D-1$, $\alpha_i\alpha_j$ is an involution.  $\D$-maps are suited to describe oriented spaces.

\subsubsection{Remarks on the various notions of generalized maps}

The notion of $\D$-generalized maps we introduced in Defi.~\ref{def:general-refoldings} has in common with the other notions reviewed above, their encodings by $\D$ or ($\D+1$)-uplets of permutations, with different types of constraints (both in the case of colored triangulations and $\D$-maps, we have encodings by $\D$-uplet as well as ($\D+1$)-uplets of permutations). As mentioned in Sec.~\ref{sec:intro-phys} of the introduction, the issue in trying to obtain interesting asymptotic objects as limits of  random generalized maps in the case of Sections~\ref{sec:TM} and \ref{sec:DS} is to find a suitable criterion of selection for generalized maps such that uniform generalized maps in the corresponding subset have interesting asymptotic properties. Indeed, as mentioned in the introduction:
\begin{enumerate}[$-$]
\item Maximizing the discrete Regge curvature for triangulations with $n$ $\D$-simplices and with edges of unit length, which is the natural approach in discrete quantum gravity, leads to the CRT in the continuum.
\item Considering all random colored triangulations with $\D+1$ permutations taken uniformly, the diameter is a.s.~bounded when the number of $\D$-simplices goes to infinity \cite{Carrance-Uniform}. This is also the case, at least numerically, for all random triangulations of the sphere  \cite{Ambjorn1992, Thorleifsson1999}.  These cases correspond to the so-called ``crumpled phase" in physics, for which most likely no scaling limit can be defined. This is also to be expected for uniform $3$-maps or 3-G-maps without any restrictions or for uniform $3$-maps or 3-G-maps representing the 3-sphere. Similarly, it is natural to conjecture that uniform random maps whose genus grows linearly with the number of faces (in expectation, \cite{Carrance-Uniform}) have a diameter that grows asymptotically logarithmically, because of their hyperbolic behavior \cite{Budz-Louf}.
\item There are ways to restrict the set of colored $\D$-dimensional triangulations that should lead to the Brownian map as a scaling limit (see the footnote \ref{foot:note-0}). 
\end{enumerate}
But so far, there are no selection criteria for the generalized maps of Sections~\ref{sec:TM} and \ref{sec:DS} that may lead to new scaling limits. On the other hand, although we do not have a representation of the iterated random discrete feuilletages as gluings of elementary $\D$-dimensional volumes, we do have the convergence to the iterated random continuous feuilletages, which are possible candidates to play the role of the Brownian map in higher dimensions. A possible line of research in this direction would be to search for models of $\D$-dimensional triangulations, for instance, that possess an encoding as $\D$-uplets of nested non-crossing partitions.

\section{Simulation of large discrete feuilletages} 
\label{sec:Simulations}
\setcounter{equation}{0}
To simulate the $ \D$th  discrete snake $\RBS_n[\D]$  
and then $\DR{n} \D$, we need:\\
$(1)$ an initial Dyck path with $2n$ steps,\\
$(2)$ to sample the label process $\Bell^{(1)}_n$,\\
$(3)$ to associate the height process $\bH_{n}^{(2)}$,\\
and then the iteration will create for each $j$ in $\cro{2, \D}$ the subsequent object (a Dyck path with size $4n$, label process of size $4n$,...), with a size doubling at each iteration. Some additional processes encoding the first corner of each node for each tree are also needed to compute the final map, in which all identifications need to be done.
This provides two types of problems for the simulation:
\begin{itemize}
\item[--] Imagine that $n=10^8$  and $ \D=5$ say, we need to sample about 20 processes whose lengths vary from $2\times n$ to $2^5 \times n$. This is really time consuming, but can be done on a standard personal computer in few minutes (around 30 min.~for this size, with a C program), by freeing the memory as soon as possible during the construction.
\item[--] The random tree has a natural scale $n^{1/2}$, random maps $n^{1/4}$, and, we hope $\DR n 3$ to have diameter $n^{1/8}$ (in any case  it is an upper bound). For $ \D=3$ and $n=10^8$, $n^{1/2^\D}=10$ (see the simulation Fig.~\ref{fig:prof0123}). Hence, the diameter of $\DR{n}{ \D}$  is expected to be around a small number of dozens. If one wants to have a diameter, say 10 times greater, the size $n$ has to be taken $10^8$ times larger. For current computers, it is absolutely out of question to simulate an object with $10^{16}$ nodes iterated $3$ times: millions of GB would be needed to store the data. We made some simulations presented below (with maximum size $250\times 10^6$, and $ \D=3$) for which we will see that the natural scaling $n^{1/2^ \D}$ and the size of the balls $B(r)$ around a random node which should be at least $r^{2^ \D}$ (which should reflect the asymptotic  Hausdorff dimension) are not apparent yet at this size.
\end{itemize}

The profile of a $\DR n D$ is the sequence $(q(i),i \geq 0)$, where $q(i)$ is the number of nodes at distance $i$ (for the graph distance) from the root vertex.
In Figure \ref{fig:MP}, two thousands $\DR n D$ of each size $(500k, 5M, 50M, 250M)$ have been simulated for  $ \D$ going from 2 (first picture, corresponding to unif.~quadrangulations) to $ 4$ (last picture). For each size and each $ \D$, the profile has been computed, and the mean profile (over the two thousands sampled objects) have been taken. For $ \D$ the iteration index and $n$ the size, the profile has been normalized in abscissa as follows:
\[\l(\frac{q\l(i/n^{1/2^ \D}\r)}{n^{1-1/2^ \D}},i\geq 0\r).\]
  \begin{figure}[h!]
    \centerline{\includegraphics[width=22cm]{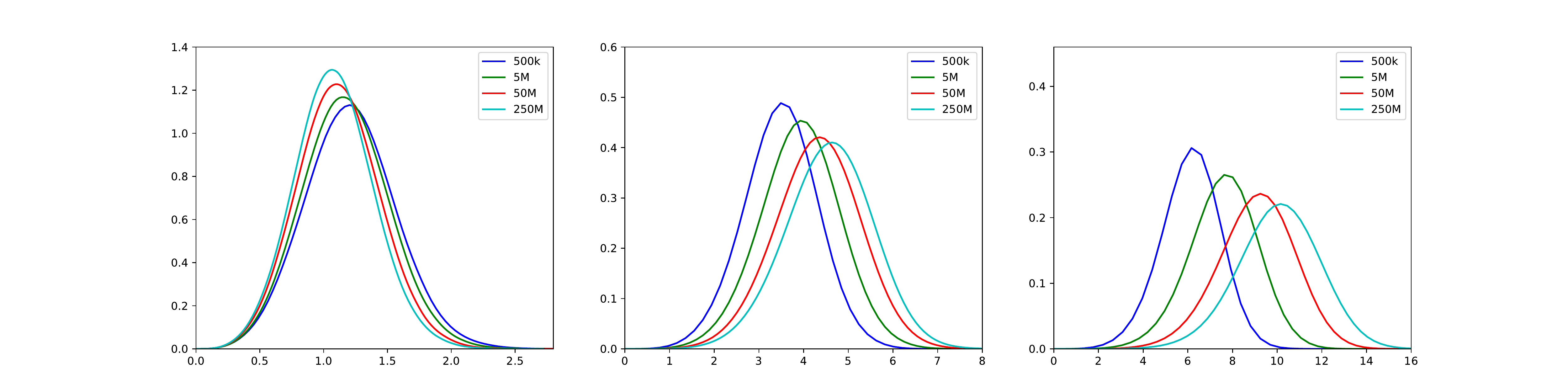}}
    \caption{\label{fig:MP}Simulation of the mean profile of $\DR{n} \D$ (for $ \D=2, 3$ and $4$) and the sizes $n$ are given in the upper right corner of each picture.}
    \end{figure}
    The convergence of the mean profile is expected for this normalization, or for a smaller normalization (since the profile could be smaller than $n^{1/2^ \D}$ for $ \D>2$). If on the first picture, one perceives some stabilizations, on the second and third ones, the mass still moves a lot when the size increases.

    In Fig.~\ref{fig:MP}, for the same statistical groups as those of Fig.~\ref{fig:MP}, we have computed the mean number $N_r$ of elements in the ball of radius $r$ around the root vertex for $ \D=2$, $3$ and $4$, and for the sizes given in the legends. We plot
    \[\log(N_r)/\log((r+1)).\]
    This quantity has been shown to be of order 4 in the first case (see Chassaing \& Durhuus \cite{CD2006}) and it is excepted to be resp.~at least 8 and 16 in the two next ones. In Figure \ref{fig:MP2} picture 2 and 3, the growth of the maximum of $\log(N_r)/\log((r+1))$ is apparent but even with the size $250\times 10^6$, it is absolutely clear that we are still very far from 8 and 16, which are lower bounds on the asymptotic regime.
  \begin{figure}[h!]
    \centerline{\includegraphics[width=22cm]{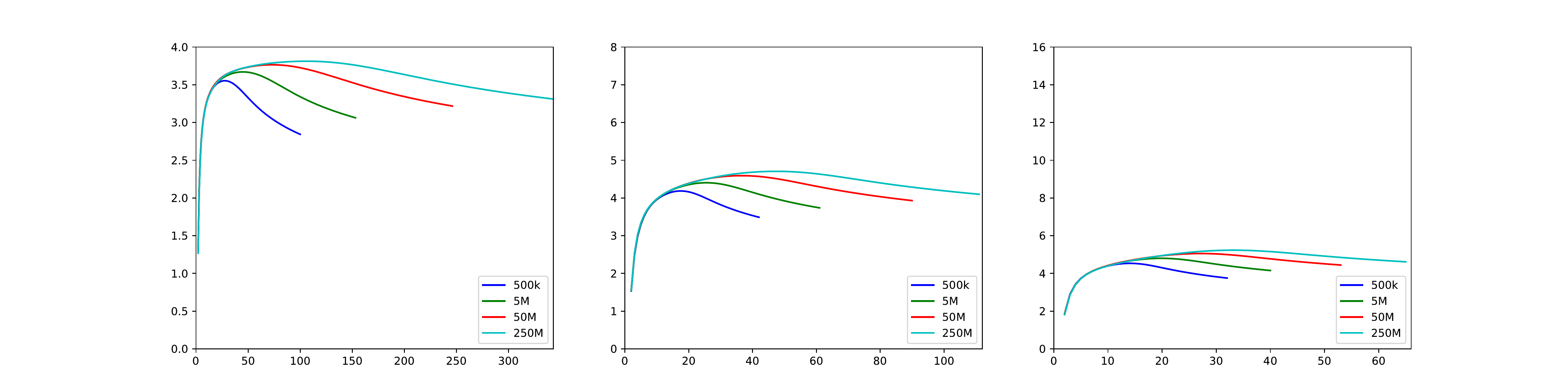}}
    \caption{\label{fig:MP2}Simulation of the mean size of balls of radius $r$ (in abscissa), in  $\DR{n} \D$ (for $ \D=2, 3$ and $4$).}
    \end{figure}
 
\renewcommand*{\refname}{}

\section*{References}
\small
\bibliographystyle{abbrv}
\vspace{-1.0cm}

%\bibliography{mybib}

\end{document}